\documentclass{memo-l}

\usepackage[utf8]{inputenc} 
\usepackage{verbatim}
\usepackage{amssymb}
\usepackage{amsthm}
\usepackage{amsmath}
\usepackage{amsfonts}
\usepackage{latexsym}
\usepackage{mathtools} 
\usepackage{enumitem} 
\usepackage[english]{babel} 
\usepackage{mathabx}

\usepackage[usenames,dvipsnames]{xcolor} 
\definecolor{dkblue}{RGB}{1,31,91} 
\usepackage[colorlinks=true, pdfstartview=FitV, linkcolor=dkblue, citecolor=dkblue, urlcolor=dkblue]{hyperref}

\emergencystretch=\maxdimen
\newtheorem{thm}{Theorem}[chapter]
\newtheorem{prop}[thm]{Proposition}

\newtheorem{lemma}[thm]{Lemma}

\theoremstyle{definition}
\newtheorem{defn}[thm]{Definition}

\theoremstyle{remark}
\newtheorem{remark}[thm]{Remark}

\numberwithin{section}{chapter}
\numberwithin{equation}{chapter}

\newcommand{\pv}{\text{pv}}

\newcommand{\R}{\mathbb R}
\newcommand{\T}{\mathbb T}
\newcommand{\md}{\mathcal{D}}
\newcommand{\mR}{\mathcal{R}}
\newcommand{\mS}{\mathcal{S}}
\newcommand{\mH}{\mathcal{H} }

\newcommand{\real}{\text{Re}\hspace{0.05cm}}
\newcommand{\imag}{\text{Im}\hspace{0.05cm}}

\newcommand{\bfcn}{b}
\newcommand{\sonesplit}{\mathcal{A}}
\newcommand{\stwosplit}{\mathcal{B}}

\newcommand{\fzerone}{\mathcal{F}^{0,1}}
\newcommand{\fzeronenu}{\mathcal{F}^{0,1}_\nu}
\newcommand{\fsone}{\dot{\mathcal{F}}^{s,1}}
\newcommand{\fsonenu}{\dot{\mathcal{F}}^{s,1}_\nu}
\newcommand{\fonesonenu}{\dot{\mathcal{F}}^{1+s,1}_\nu}

\newcommand{\fhone}{\dot{\mathcal{F}}^{\frac12,1}}
\newcommand{\fhonenu}{\dot{\mathcal{F}}^{\frac12,1}_\nu}
\newcommand{\fthreehonenu}{\dot{\mathcal{F}}^{\frac32,1}_\nu}
\newcommand{\fthreehone}{\dot{\mathcal{F}}^{\frac32,1}}
\newcommand{\fsevenhone}{\dot{\mathcal{F}}^{\frac72,1}}

\newcommand{\foneone}{\dot{\mathcal{F}}^{1,1}}

\newcommand{\thetazerone}{\|\theta\|_{\mathcal{F}^{0,1}}}
\newcommand{\thetazeronenu}{\|\theta\|_{\mathcal{F}^{0,1}_\nu}}
\newcommand{\thetathreesonenu}{\|\theta\|_{\dot{\mathcal{F}}^{3+s,1}_\nu}}
\newcommand{\thetaonesonenu}{\|\theta\|_{\dot{\mathcal{F}}^{1+s,1}_\nu}}

\newcommand{\thetasonenu}{\|\theta\|_{\dot{\mathcal{F}}^{s,1}_\nu}}
\newcommand{\thetaonenu}{\|\theta\|_{\dot{\mathcal{F}}^{1,1}_\nu}}

\newcommand{\thetatwosonenu}{\|\theta\|_{\dot{\mathcal{F}}^{2+s,1}_\nu}}
\newcommand{\thetahonenu}{\|\theta\|_{\dot{\mathcal{F}}^{\frac12,1}_\nu}}

\newcommand{\thetaoneone}{\|\theta\|_{\dot{\mathcal{F}}^{1,1}}}
\newcommand{\thetatwoone}{\|\theta\|_{\dot{\mathcal{F}}^{2,1}}}
\newcommand{\thetatwoonenu}{\|\theta\|_{\dot{\mathcal{F}}^{2,1}_\nu}}

\newcommand{\thetasevenhonenu}{\|\theta\|_{\dot{\mathcal{F}}^{\frac72,1}_\nu}}

\newcommand{\tildefzero}{\tilde{\mathcal{F}}^{0,1}}

\newcommand{\tildefs}{\tilde{\mathcal{F}}^{s,1}}

\newcommand{\sign}{\text{sgn}}
\newcommand{\eqdef}{\overset{\mbox{\tiny{def}}}{=}}
\newcommand{\dissconstant}{\mathcal{D}}
\newcommand{\fsoneplusnu}{\dot{\mathcal{F}}^{(s-1)^+,1}_\nu}
\newcommand{\crconstant}{C_{\mR}}
\newcommand{\timeintE}{\mathcal{E}}
\newcommand{\diffusint}{\mathcal{C}}

\begin{document}

\frontmatter

\title[Global regularity for gravity unstable Muskat bubbles]{Global regularity for gravity unstable Muskat bubbles}



\author[F. Gancedo]{Francisco Gancedo$^\dagger$}
\address{$^\dagger$Departamento de An\'{a}lisis Matem\'{a}tico $\&$ IMUS, Universidad de Sevilla, C/ Tarfia s/n, Campus Reina Mercedes, 41012 Sevilla, Spain.  
(\href{https://orcid.org/0000-0002-9006-8804}{https://orcid.org/0000-0002-9006-8804})}
\email{\href{mailto:fgancedo@us.es}{fgancedo@us.es}}


\author[E. Garc\'ia-Ju\'arez]{Eduardo Garc\'ia-Ju\'arez$^{\ddagger}$}
\address{$^\ddagger$Departament de Matemàtiques i Informàtica, Universitat de Barcelona, 
Gran Via de les Corts Catalanes, 08007 Barcelona, Spain.
(\href{https://orcid.org/0000-0003-1022-0835}{https://orcid.org/0000-0003-1022-0835})}
\email{\href{mailto:egarciajuarez@ub.edu}{egarciajuarez@ub.edu}}
\address{\textit{Former address:} Department of Mathematics, University of Pennsylvania, Philadelphia, PA 19104, USA.
}


\author[N. Patel]{Neel Patel$^*$}
\address{$^*$Department of Mathematics, University of Michigan, Ann Arbor, East Hall, 530 Church Street 48109, Ann Arbor, Michigan, USA 48109, USA. }
\email{\href{mailto:neeljp@umich.edu}{neeljp@umich.edu} }


\author[R. M. Strain]{Robert M. Strain$^{\mathsection}$}
\address{$^\mathsection$Department of Mathematics, University of Pennsylvania, Philadelphia, PA 19104, USA. 
(\href{https://orcid.org/0000-0002-1107-8570}{https://orcid.org/0000-0002-1107-8570})
}
\email{\href{mailto:strain@math.upenn.edu}{strain@math.upenn.edu}}


\date{16 Sep 2020}




\dedicatory{Acknowledgements: FG and EGJ were partially supported by the grant MTM2017-89976-P (Spain). FG, EGJ and NP were partially supported by the ERC through the Starting Grant project H2020-EU.1.1.-639227. EGJ was partially supported by the ERC Starting Grant ERC-StG-CAPA-852741. NP was partially supported by AMS-Simons Travel Grants, which are administered by the American Mathematical Society with support from the Simons Foundation.  RMS was partially supported by the NSF grants DMS-1764177 and DMS-2055271 (USA).  The authors thank the referee for their careful reading of the manuscript.}

\begin{abstract}
    In this paper, we study the dynamics of fluids in porous media governed by Darcy's law: the Muskat problem. We consider the setting of two immiscible fluids of different densities and viscosities under the influence of gravity in which one fluid is completely surrounded by the other. This setting is gravity unstable because along a portion of the interface, the denser fluid must be above the other. Surprisingly, even without capillarity, the circle-shaped bubble is a steady state solution moving with vertical constant velocity determined by the density jump between the fluids. Taking advantage of our discovery of this steady state, we are able to prove global in time existence and uniqueness of dynamic bubbles of nearly circular shapes under the influence of surface tension. We prove this global existence result for low regularity initial data.  Moreover, we prove that these solutions are instantly analytic and decay exponentially fast in time to the circle.
\end{abstract}

\maketitle

\tableofcontents


\mainmatter

\chapter{Introduction}

This paper studies the dynamics of fluid drops or bubbles immersed in another fluid filling a porous media under the action of gravity. This process is governed by the classical Darcy's law 
\begin{equation}\label{Darcy}
\frac{\mu(x,t)}{\kappa(x,t)} u(x,t)=-\nabla p(x,t)-g(0,\rho(x,t)), 
\end{equation} 
where  $u$ is the velocity of the fluid, $p$ is the pressure, $\rho$ is the density and $\mu$ is the viscosity of the fluid. Above $x\in\mathbb{R}^2$ and $t\geq 0$.  Here the medium is assumed to be homogeneous so  that the permeability $\kappa(x,t)=\kappa \ge0$ is constant, as is the gravitational acceleration $g>0$. While Darcy's law was first derived experimentally \cite{Darcy56}, it can be rigorously obtained through homogenization \cite{Tartar89,Hornung97}. This physical scenario is mathematically analogous to the evolution of an incompressible flow in a Hele-Shaw cell \cite{Gancedo2017} where the fluid is set inside two parallel plates that are close enough together so that the resulting dynamics are two dimensional. In particular, the results in this paper can be applied to the Hele-Shaw problem.

The presence of two immiscible fluids is modeled by taking the viscosity $\mu$ and the density $\rho$ as piece-wise constant functions:
\begin{equation}\label{patchsolution}
\mu(x,t)=\left\{\begin{array}{rl}
\mu^1,& x\in D(t),\\
\mu^2,& x\in \mathbb{R}^2\smallsetminus \overline{D(t)},
\end{array}\right.
\quad 
\rho(x,t)=\left\{\begin{array}{rl}
\rho^1,& x\in D(t),\\
\rho^2,& x\in \mathbb{R}^2\smallsetminus \overline{D(t)},
\end{array}\right.
\end{equation}
where $D(t)$ is a simply connected bounded domain, namely, the \textit{bubble}. Thus, there is a sharp interface between the fluids, moving with the flow, which we assume to be incompressible:
\begin{equation}\label{incom}
\nabla \cdot u(x,t)=0.
\end{equation}
We consider the physically relevant case where surface tension at the free boundary is taken into consideration. The Laplace-Young's formula then states that \cite{Hou94}:
\begin{equation}\label{surfacetension}
p^1(x,t)-p^2(x,t)=\sigma K(x,t),\hspace{1cm}x\in\partial D(t),
\end{equation}
where $K(x,t)$ denotes the curvature of the curve $\partial D(t)$, $\sigma>0$ is the constant surface tension coefficient and $p^1(x,t)$, $p^2(x,t)$ are the limits of the pressure at $x$ from inside and outside, respectively.
We are then dealing with the Muskat problem, where the main mathematical interest is to study the dynamics of the free boundary $\partial D(t)$, especially between water and oil \cite{Muskat34}. 
It is remarkable that the evolution equation for the free boundary is well-defined even though the velocity is not continuous. The discontinuity in the velocity is due to the density, viscosity and pressure jumps. But the interface evolution is dictated only by the normal velocity, which is continuous by the incompressibility condition. 

In this sense it is indeed possible to obtain a self-evolution equation for the interface $\partial D(t)$ which is called the contour evolution system.   This system is equivalent to the Eulerian-Lagrangian formulation \eqref{Darcy}, \eqref{patchsolution}, \eqref{incom}, \eqref{surfacetension} understood in a weak sense.
Due to the irrotationality of the velocity in each domain $D(t)$, the vorticity is concentrated on the interface $\partial D(t)$. That is, the vorticity is given by a delta distribution as follows
$$
\nabla^\perp\cdot u(x,t)=\omega(\alpha,t)\delta(x=z(\alpha,t)),
$$
where $\omega(\alpha,t)$ is the amplitude of the vorticity and $z(\alpha,t)$ is a parameterization of $\partial D(t)$ with
$$
\partial D(t)=\{z(\alpha,t)=(z_1(\alpha,t), z_2(\alpha,t)):\, \alpha\in[-\pi,\pi]\}.
$$
The Biot-Savart law then yields that
\begin{equation*}
u(x,t)=\frac1{2\pi} \pv\int_{-\pi}^\pi \frac{(x-z(\beta,t))^\perp}{|x-z(\beta,t)|^2}\omega(\beta,t)d\beta,\hspace{1cm}x\neq z(\beta,t),
\end{equation*}
and taking limits in the normal direction to $z(\alpha,t)$ one finds
\begin{equation*}
\begin{aligned}
u_1(z(\alpha,t),t)&=BR(z,\omega)(\alpha,t)-\frac12\frac{\omega(\alpha,t)}{|\partial_{\alpha}z(\alpha,t)|^2}\partial_{\alpha}z(\alpha,t),\\
u_2(z(\alpha,t),t)&=BR(z,\omega)(\alpha,t)+\frac12\frac{\omega(\alpha,t)}{|\partial_{\alpha}z(\alpha,t)|^2}\partial_{\alpha}z(\alpha,t),
\end{aligned}
\end{equation*}
where $BR$ is the Birkhoff-Rott integral that is given by
\begin{equation}\label{eqBR}
BR(z,\omega)(\alpha,t)=\frac1{2\pi} \pv\int_{-\pi}^\pi \frac{(z(\alpha,t)-z(\beta,t))^\perp}{|z(\alpha,t)-z(\beta,t)|^2}\omega(\beta,t)d\beta.
\end{equation}
Taking the dot product with $\partial_\alpha z$ in the above equations for $u_1$ and $u_2$ and subtracting one from the other, one then finds that the vorticity strength is given by the jump in the tangential velocity
\begin{equation*}
\omega(\alpha,t)=\left(u_2(z(\alpha,t),t)-u_1(z(\alpha,t),t)\right)\cdot\partial_{\alpha}z(\alpha,t).
\end{equation*}
Then using Darcy's law \eqref{Darcy} yields the non-local implicit identity
\begin{multline}\label{eqOmega}
\omega(\alpha,t)=2 A_\mu |\partial_{\alpha}z(\alpha,t)| \md(z,\omega)(\alpha,t) +2A_\sigma\partial_{\alpha}K(z(\alpha,t))-2A_\rho\partial_{\alpha}z_2(\alpha,t),
\end{multline}
where
\begin{equation}\label{Dz}
\md(z,\omega)(\alpha,t)=-BR(z,\omega)(\alpha,t)\cdot\frac{\partial_{\alpha}z(\alpha,t)}{|\partial_{\alpha}z(\alpha,t)|},
\end{equation}
and
\begin{equation}\label{AmuAsigmaArho}
A_\mu=\frac{\mu_2-\mu_1}{\mu_2+\mu_1}, \quad 
A_{\sigma}=\frac{\kappa \sigma}{\mu_2+\mu_1}, \quad 
A_\rho = g\kappa \frac{\rho_2-\rho_1}{\mu_2+\mu_1}.
\end{equation}
Further, in \eqref{eqOmega} the curvature is given by
\begin{equation}\label{curvature}
	K(\alpha,t)=\frac{\partial_{\alpha}z(\alpha,t)^\perp\cdot\partial_{\alpha}^2z(\alpha,t)}{|\partial_{\alpha}z(\alpha,t)|^3}.
\end{equation}
Since the fluids are immiscible, the interface is just advected by the normal velocity of the fluid flow:
\begin{equation*}
z_t(\alpha,t)\cdot \partial_\alpha z(\alpha,t)^\perp=BR(z,\omega)(\alpha,t)\cdot \partial_\alpha z(\alpha,t)^\perp.
\end{equation*}
Therefore a tangential velocity $T(z(\alpha,t))$ can be introduced to change the parametrization of the interface, without altering its shape.  Let $\partial_\alpha z(\alpha,t)=z_\alpha(\alpha,t)$. Then we denote the unit tangent and normal vectors by
	\begin{equation}\label{vectors}
	\tau(\alpha,t)=\frac{z_\alpha(\alpha,t)}{|z_\alpha(\alpha,t)|},\quad n(\alpha,t)=\frac{z_\alpha(\alpha,t)^\perp}{|z_\alpha(\alpha,t)|}.
	\end{equation}
Without changing the shape of the interface we can replace the above equation by 	
\begin{equation}\label{eqcurve}
z_t(\alpha,t)=\big(BR(z,\omega)(\alpha,t)\cdot n(\alpha,t)\big)n(\alpha,t)+T(z(\alpha,t))\tau(\alpha,t).
\end{equation}
Therefore we have a closed system of equations for the the contour evolution system with (\ref{eqcurve}), (\ref{eqBR}), (\ref{eqOmega}), and  (\ref{Dz}).

Given its origins in petrochemical engineering and its mathematical equivalence with Hele-Shaw flows \cite{SaffmanTaylor58}, the Muskat problem has long attracted a lot of attention from physics \cite{Bear72,Saffman86}. Mathematically, the Muskat problem poses many challenges, since even the well-posedness of the problem is not always guaranteed. Indeed, when one neglects surface tension, the well-posedness depends on the  Rayleigh-Taylor condition (which is also called the Saffman-Taylor condition for the Muskat problem). If the fluids have different densities, this condition requires the denser fluid to be below the less dense fluid. When this condition is satisfied, i.e., in the stable setting \cite{Ambrose04}, local-in-time existence for large initial data is known for both density and viscosity jump cases in 2d and 3d \cite{CG07,CCG11, CCG13,CGS16}
for subcritical spaces
\cite{CGSV17,Mat16,AlazardOneFluid2019,AlazardLazar2019,NguyenPausader2019}.
However, finite time singularities can arise even from these stable configurations. As a matter of fact, the Muskat problem was the first incompressible model where blow-up was proved starting with well-posed initial data \cite{CCFGL12,CCFG13,GG14,CCFG16}.

From the previous considerations, it is an important question to determine under which conditions the solution exists and remains regular globally in time.  For the non-surface tension case, the global existence in the stable setting was first obtained  for small enough initial data in subcritical norms, allowing both density and viscosity jumps \cite{SCH04,CG07,EM11, CGS16} and later for some critical norms \cite{BSW14,CGSV17}. Very recently, global well-posedness results appeared that allow initial data of \textit{medium} size in critical spaces, meaning initial data explicitly bounded independent of any parameter: first only for the density jump case \cite{CCGS13,CCGRS16,Cam17,Cam20}, and later extended to the density-viscosity jump case \cite{GGPS19}.  In particular in \cite{GGPS19} there is a medium-size bound for the initial data that is independent of any parameter of the system when $|A_\mu|=1$, and that value of the bound for the initial data is improved when $|A_\mu|<1$. In all these results, the magnitude of the slope of the first derivative appears as a crucial quantity. However, this restriction is removed in \cite{CordobaLazar2018,GancedoLazar2020} by assuming smallness in the critical $L^2$ based Sobolev norm.

 On the other hand, in the unstable scenario, the problem is ill-posed in all Sobolev spaces $H^s$, $s>0$ \cite{GGPS19}, unless surface tension is taken into account. In that case, surface tension controls the instabilities at large scales, giving well-posedness. Classical results for this scenario can be found in \cite{DuchonRobert1984,Chen93,ES97}. See the recent work \cite{NguyenST2019} for low
  regularity initial data, and  \cite{1905.05370}
 for weak solutions
constructed by interpreting the Muskat problem as a gradient flow in a product Wasserstein space. Unstable scenarios are known \cite{Ott97} which exhibit exponential growth locally in time of low order norms \cite{GHS07}, and finger shaped unstable stationary solutions were also studied \cite{EM11}. In particular, Rayleigh-Taylor stable solutions with surface tension converge to the solution without surface tension \cite{Ambrose2014} with optimal decay rate or low regularity \cite{FlynnNguyen2020}. Here, this is not the case as the scenario we deal with is Rayleigh-Taylor unstable. 
Recently, while writing this paper, unstable fluid layers have been proved to exist globally in time for initial near flat configurations \cite{GG-BS2019}.
 
In this paper, we aim to improve the understanding of the effects produced by the surface tension for bubble-shaped interfaces. In particular, we consider the movement of fluid bubbles under the effect of gravity in another fluid with both different densities and viscosities. This is a highly unstable situation, as the Rayleigh-Taylor condition cannot hold for a closed curve. The function that provides the Rayleigh-Taylor condition has mean zero in this scenario. Moreover, as one expects, we will show that a less dense bubble moves upwards. But this means that on the top part of the interface, the less viscous fluid may push the more viscous one and the denser one is on top of the lighter one: both classic conditions in the linear Rayleigh-Taylor analysis are violated here. So that in our scenario, gravity effects make hard to find global-in-time control.
Previous results dealing with this setting \cite{CP93,YT11,YT12} assumed no gravity force (i.e., $g=0$ or no density jump) and required small initial data in high regularity spaces (such as $H^r$ for $r\geq 4$).

 We show here that even without surface tension, circle shaped curves are steady state solutions evolving vertically due to gravity. Furthermore, this surprising state in this unstable configuration allows to find global-in-time existence for capillarity bubbles. We will show that if the initial interface of a bubble is close to a circle with respect to a  constant depending on the dimensionless constants
$$|A_\mu|\quad\mbox{and}\quad  \frac{R^2|A_\rho|}{A_\sigma},\quad\mbox{with}\quad \pi R^2=|D(0)|,$$
 then the solution exists globally in time and, moreover, it becomes instantly analytic. In particular, in our proof it is possible to compute the explicit numerical condition that the initial data must satisfy. It is interesting to notice that only two quantities are involved, where the second represents the ratio between gravity force per length and surface tension, 
$$\frac{|A_{\rho}|R^2}{A_\sigma}=\frac{gR^2|\rho_2-\rho_1|}{\sigma}.$$
We will also show that these bubbles converge exponentially fast in time to a circle that moves vertically with constant velocity equal to $A_\rho$ (upwards if $A_\rho>0$). Due to the incompressibility condition, the area of the bubble is preserved during the process.
We give precise statements of these results in Chapter \ref{MainResults}. In next section, we provide the contour equations we use throughout the paper.

Note that the parameterization that is used in for instance \cite{CG07,CCGS13,CCGRS16,Cam17,GGPS19} is difficult to use in our scenario (those results are close to a horizontal line while in contrast the results in this paper are close to a circle) because our system in general sends the solution to a nearby circular steady state, and not to the one that we linearize around. The steady state that the solution converges to is determined by the dynamics and, without another conservation law that may not exist, then the limit can not be predicted by the initial data alone. In particular, the standard parametrization for star-shape bubbles given by
\begin{equation*}
    z(\alpha,t)=R(1+f(\alpha,t))(\cos{(\alpha)},\sin{(\alpha)})
\end{equation*}
does not do a good job of describing the nearby circular steady states, and therefore it is hard to use in this context. In particular, unless the center is the origin, which one cannot know \textit{a priori}, circles parametrized in this form do not have a simple expression.
From the analytical point of view, 
looking at the decay on the Fourier side, it is easy to find that there is no dissipation at the linear level for the $\pm 1$ Fourier coefficients, which corresponds to the fact that the center of a circle parametrized in this way is given by the $\pm 1$ Fourier coefficients of $f$. At the nonlinear level, these Fourier coefficients are present and mixed in the evolution together with the rest of them, making it difficult to control globally in time. In order to handle this issue we reparametrize the interface getting a tangent vector to the curve with length independent of the parameter $\alpha$  so that $\partial_\alpha |z_\alpha|=0$  \cite{Hou94}. Therefore, the system can be reformulated in terms of the angle formed between the tangent and the horizontal, $\alpha+\vartheta(\alpha,t)$, and the length of the curve as follows
$$
z_\alpha(\alpha,t)=\frac{L(t)}{2\pi}(\cos(\alpha+\vartheta(\alpha,t)),\sin(\alpha+\vartheta(\alpha,t))).
$$
The main unknown to control in this setting is $\vartheta$. 

In this parametrization circles correspond to a constant value of $\vartheta$. 
The evolution  of the zero frequency of $\vartheta$ is decoupled from the rest.
While the $\pm 1$ are also neutral in this formulation, the simple compatibility condition
$$
\int_{-\pi}^\pi z_\alpha(\alpha,t)d\alpha=0
$$
used in \cite{YT11,YT12}, allows us to control the $\pm 1$ Fourier coefficients of $\vartheta$ in terms of the higher modes.   For the higher Fourier modes we can use the dissipation due to surface tension. All the frequencies, together with the initial condition, determine the evolution of the center of the bubble.

On the other hand, in the analysis done around circles, it is possible to check that the Fourier coefficients of different frequencies interact together in the evolution, even at the linear level. If the ratio $|A_\rho|R^2/A_\sigma$  between gravity and surface tension forces is large, it is not straightforward how to take advantage of the dissipation. Thus it is not clear how to obtain the global-in-time result in terms only of the size of the initial data and not upon the size of the parameters. In order to obtain a global result that does not rely on the size of the physical parameters of the problem, we preform a transformation in Fourier space of the infinite-dimensional nonlinear system and we prove that this transformation \textit{diagonalizes} the linear system so that our result holds for any size of the physical parameters.  In particular, we show that it is possible to obtain explicitly the size of the smallness constant. Finally, the analysis we perform is for low regularity initial data, ($z_0(\alpha)\in C^{1,\frac12}(\T)$), allowing unbounded  initial curvature and providing instant (analytic) smoothing.

\section{Outline} The rest of the paper is structured as follows.  In Chapter \ref{formulation}, we explain the contour dynamics formulation of the system of equations for the interface and we derive the full linearization.  Chapter \ref{MainResults} records the notation that will be used in the rest of the paper and explains the main theorem proving global existence, uniqueness and exponential large time decay.  Then Chapter \ref{IFTSection} gives the proof of the implicit function theorem to obtain the implicit relation between $\hat{\theta}(\pm 1)$ and the higher Fourier modes.
In Chapter \ref{sec:FourierRandS} we prove the Fourier multiplier estimates for the operators $\mathcal{R}$ and $\mathcal{S}$.
Chapter \ref{secw} proves the a priori estimates on the vorticity strength $\omega$.
In Chapter \ref{secanalytic} we use all the previous estimates to prove the global existence and instant analyticity of solutions.
Lastly in Chapter \ref{sec:uniqueness} we explain the proof of uniqueness.

\chapter{Contour dynamics formulation}\label{formulation}

In this chapter we introduce the contour evolution equations that will be used throughout the paper. We suppress the dependence in $t$ for clarity of notation. We note that in the introduction it was convenient to introduce the system using vector notation.  However in the rest of the paper, we will study the equation using complex notation.    In Section \ref{sec:ComplexVector} we explain some complex notation used in the rest of this paper.  In Section \ref{subsec:parametrization}, we rewrite the equations (\ref{eqcurve}), (\ref{eqBR}), (\ref{eqOmega}), (\ref{Dz}) in terms of the length of the curve and the angle of the tangent vector \cite{Hou94}.  Then in Section \ref{subsec:evolutionSystem} we derive an equivalent expression for the evolution of the length of the curve.  
In Section \ref{sec:FourierTransCalc}, we explain the calculations that we will use involving the Fourier transform in our further decompositions.
Lastly in Section \ref{sec:linearization}  we decompose the equations into linear and nonlinear parts.  In particular, the calculation of the expression for the linearized operator is given in Proposition \ref{linearfourier}.

\section{Complex notation and vector notation}\label{sec:ComplexVector}
In particular given $x=(x_1, x_2)$ and $y=(y_1, y_2)$ in vector notation and given $z=x_1+ix_2$ and $w=y_1+iy_2$ in complex notation, then the inner product is expressed as
\begin{equation}\notag
    x\cdot y = x_1 y_1 + x_2 y_2 = \real(\overline{z} w).    
\end{equation}
Here $\overline{z} = x_1 - i x_2$ is the complex conjugate.  Similarly in two dimensions in vector notation we can write  $x \wedge y \eqdef x_1 y_2 - x_2 y_1$ in vector notation and this is equal to $\imag{(\overline{x}y)}$ in complex notation.  Then for a vector the perpendicular is $x^\perp = (-x_2, x_1)$, and in complex notation the perpendicular is $iz= -x_2 + i x_1$.  We will use the complex notation in most of the rest of the paper.

\section{Parametrization}\label{subsec:parametrization}

Now we define $\vartheta(\alpha)$ so
that $\alpha+\vartheta(\alpha)$ is the angle formed between the tangent to the curve and the horizontal. In complex notation, this means that
\begin{equation}\label{param}
    z_\alpha(\alpha)=|z_\alpha(\alpha)|e^{i(\alpha+\vartheta(\alpha))}.
\end{equation}
In this formulation the normal and tangential vectors from \eqref{vectors} are 
\begin{equation*}
    n(\alpha)=ie^{i(\alpha+\vartheta(\alpha))}, \quad \tau(\alpha) = e^{i(\alpha+\vartheta(\alpha))}.
\end{equation*}
We will then denote the normal velocity by $U(\alpha)$ with
\begin{equation}\label{U}
    U(\alpha)=\real(\overline{BR}(z,\omega)(\alpha) n(\alpha))=\real(\overline{BR}(\omega)(\alpha) ie^{i(\alpha+\vartheta(\alpha))}),  
\end{equation}
with the Birkhoff-Rott integral \eqref{eqBR} given by
\begin{equation*}
    \overline{BR}(z,\omega)(\alpha)=\frac{1}{2\pi i}\pv\int_{-\pi}^{\pi}\frac{\omega(\beta)}{z(\alpha)-z(\beta)}d\beta.
\end{equation*}
Note that $\overline{BR}(z,\omega)$ is the complex conjugate of \eqref{eqBR} written in complex notation, this holds in particular since $\omega(\beta)$ is seen to be real as in \eqref{omega} below.
These expressions can be written in terms of $z_\alpha(\alpha)$ by noticing that
\begin{equation*}
    z(\alpha)-z(\alpha-\beta)=\int_0^\alpha z_\alpha(\eta)d\eta-\int_0^{\alpha-\beta} z_\alpha(\eta)d\eta =\beta\int_0^1 z_\alpha(\alpha+(s-1)\beta)ds.
\end{equation*}
The equation \eqref{eqcurve} then reads as follows
\begin{equation}\label{eqcurve.theta}
    z_t(\alpha)=U(\alpha)n(\alpha)+T(\alpha)\tau(\alpha).
\end{equation}
Taking a derivative in $\alpha$ and projecting into normal and tangential components, we obtain the evolution equations for $\vartheta(\alpha)$ and $|z_\alpha(\alpha)|$:
\begin{equation}\label{thetaeqaux}
\begin{aligned}
    \vartheta_t(\alpha)&=\frac{1}{|z_\alpha(\alpha)|}\Big(U_\alpha(\alpha)+T(\alpha)(1+\vartheta_\alpha(\alpha))\Big),\\
    |z_\alpha(\alpha)|_t&=T_\alpha(\alpha)-(1+\vartheta_\alpha(\alpha))U(\alpha).
\end{aligned}
\end{equation}
Now, we can choose a tangential velocity $T(\alpha)$ so that the parametrization of $z(\alpha)$ has a tangent vector whose modulus does not depend on $\alpha$. Indeed, we impose
\begin{equation}\label{modul}
    |z_\alpha(\alpha)|=\frac{1}{2\pi}\int_{-\pi}^{\pi}|z_\alpha(\eta)|d\eta=\frac{L(t)}{2\pi},
\end{equation}
where $L(t)$ is the length of the curve at time $t$. We then differentiate in time the equation above and use equation \eqref{thetaeqaux}$_2$ twice to obtain that
\begin{equation}\label{T}
    T(\alpha)=\int_0^{\alpha}(1+\vartheta_\alpha(\eta))U(\eta)d\eta-\frac{\alpha}{2\pi}\int_{-\pi}^{\pi}(1+\vartheta_\alpha(\eta))U(\eta)d\eta+T(0),
\end{equation}
where $T(0)$ simply provides a change of frame in the parametrization. 
Therefore, after substitution of this expression of $T(\alpha)$ into \eqref{thetaeqaux} and using the relation $|z_\alpha(\alpha)|=\frac{L(t)}{2\pi}$, the evolution system in terms of $\vartheta(\alpha)$ and $L(t)$ is the following
\begin{equation}\label{eqcurve3}
\begin{aligned}
    \vartheta_t(\alpha)&=\frac{2\pi}{L(t)} U_\alpha(\alpha)+\frac{2\pi}{L(t)}T(\alpha)(1+\vartheta_\alpha(\alpha)),\\
    L_t(t)&=-\int_{-\pi}^{\pi} (1+\vartheta_\alpha(\alpha))U(\alpha)d\alpha,
\end{aligned}
\end{equation}
where $T(\alpha)$ is defined in \eqref{T}, with $T(0)$ free to choose, 
and $U(\alpha)$ is given by \eqref{U} with
\begin{equation}\label{BR}
    \overline{BR}(\omega)(\alpha)=\frac{1}{ i L(t) }\pv\int_{-\pi}^{\pi}\frac{\omega(\alpha-\beta)}{\int_0^1 e^{i(\alpha+(s-1)\beta)}e^{i\vartheta(\alpha+(s-1)\beta)}ds}\frac{d\beta}{\beta}.
\end{equation}
Recalling the expression of the curvature in terms of the angle using \eqref{curvature}, 
\begin{equation*}
    K(z)(\alpha)=\frac{2\pi}{L(t)}(1+\vartheta_\alpha(\alpha)),
\end{equation*}
the equation for the vorticity strength $\omega(\alpha)$ in \eqref{eqOmega} reads as follows
 \begin{equation}\label{omega}
    \omega(\alpha)=2A_\mu \frac{L(t)}{2\pi} \md(\omega)(\alpha)+2A_\sigma\frac{2\pi}{L(t)} \vartheta_{\alpha\alpha}(\alpha)-2A_\rho\frac{L(t)}{2\pi}\sin{(\alpha+\vartheta(\alpha))},
\end{equation}
 with $\md(z,\omega)(\alpha)$ in \eqref{Dz} given by
\begin{equation}\label{md}
    \md(\omega)(\alpha)=-\real(\overline{BR}(\omega)(\alpha) e^{i(\alpha+\vartheta(\alpha))}).
\end{equation}
In addition, 
we notice that because $|z_\alpha(\alpha)|$ is constant in $\alpha$ and $z(\alpha)$ is a closed curve, then the following constraint must hold
\begin{equation}\label{constraint}
    0=\int_{-\pi}^{\pi} \frac{z_\alpha(\alpha)}{|z_\alpha(\alpha)|} d\alpha=\int_{-\pi}^{\pi} e^{i(\alpha+\vartheta(\alpha))}d\alpha.
\end{equation}
Finally, once the system \eqref{T}-\eqref{md} is solved, one can track the evolution of a single point, say that with $\alpha=0$, by integrating in time \eqref{eqcurve} (notice that the right hand side of \eqref{eqcurve} has been shown to depend only on $z_\alpha$, given by \eqref{param}, \eqref{modul}, and \eqref{eqcurve3}).

\section{Evolution System}\label{subsec:evolutionSystem}
For our purposes, equation \eqref{eqcurve3}$_2$ is not  convenient  to study $L(t)$. Instead, we will make use of the fact that the fluid is incompressible, and thus the volume is preserved.
The volume is given in terms of the curve $z(\alpha)$ by
\begin{equation*}
    V=\frac12\int_{-\pi}^\pi z(\alpha)\wedge  z_\alpha(\alpha)d\alpha,
\end{equation*}
which in complex notation reads as
\begin{equation*}
    V=\frac12\imag\int_{-\pi}^\pi  \overline{z(\alpha)} z_\alpha(\alpha) d\alpha.
\end{equation*}
Since $U(\alpha)$ in \eqref{U} is a total derivative in $\alpha$, the conservation of volume is obtained by simply taking a time derivative in the equation above.
Now, from \eqref{param} and \eqref{modul} we have that
$$
z_\alpha(\alpha)=\frac{L(t)}{2\pi}e^{i(\alpha+\vartheta(\alpha))},
$$
and we can write
\begin{equation*}
    z(\alpha)=z(0)+\int_0^\alpha z_\alpha(\eta)d\eta.
\end{equation*}
Then the conservation of volume  writes as follows
\begin{equation}\label{volumelength}
\begin{aligned}
    V_0&=\pi R^2
    \\
    &=\frac{1}{2}\Big(\frac{L(t)}{2\pi}\Big)^2\imag \Big(\int_{-\pi}^\pi\int_0^\alpha e^{i(\alpha-\eta)}e^{i(\vartheta(\alpha)-\vartheta(\eta))}d\eta d\alpha\Big)\\
    &=\frac{1}{2}\Big(\frac{L(t)}{2\pi}\Big)^2\imag \Big(2\pi i
    +\int_{-\pi}^\pi\int_0^\alpha e^{i(\alpha-\eta)} \sum_{n\geq1}\frac{i^n}{n!}(\vartheta(\alpha)-\vartheta(\eta))^n d\eta d\alpha\Big).
\end{aligned}
\end{equation}
This yields the following equation for $L(t)$:
\begin{equation}\label{Lequation}
\Big(\frac{L(t)}{2\pi}\Big)^2=R^2\Big(1+\frac{1}{2\pi}\imag \int_{-\pi}^\pi\int_0^\alpha e^{i(\alpha-\eta)} \sum_{n\geq1}\frac{i^n}{n!}(\vartheta(\alpha)-\vartheta(\eta))^n d\eta d\alpha\Big)^{-1}.
\end{equation}
This is the equation for $L(t)$ that will be use later in the paper,

Conversely, it is not hard to check that \eqref{Lequation} implies \eqref{eqcurve3}$_2$. In fact, taking a time derivative of \eqref{volumelength}, and assuming $L(t)\neq0$, gives that
\begin{equation*}
    \begin{aligned}
    L'(t)&=\frac{1}{4\pi R^2}\Big(\frac{L(t)}{2\pi}\Big)^3\imag \int_{-\pi}^\pi\int_0^\alpha i e^{i(\alpha-\eta)} e^{i(\vartheta(\alpha)\!-\!\vartheta(\eta))}(\vartheta_t(\alpha)\!-\!\vartheta_t(\eta)) d\eta d\alpha\\
    &=\frac{1}{4\pi R^2}\Big(\frac{L(t)}{2\pi}\Big)^3\imag \int_{-\pi}^\pi i e^{i\alpha} e^{i\vartheta(\alpha)}\vartheta_t(\alpha)\int_0^\alpha  e^{-i\eta} e^{-i\vartheta(\eta)}d\eta d\alpha\\
    &\quad-\frac{1}{4\pi R^2}\Big(\frac{L(t)}{2\pi}\Big)^3\imag \int_{-\pi}^\pi i e^{i\alpha} e^{i\vartheta(\alpha)}\int_0^\alpha  e^{-i\eta} e^{-i\vartheta(\eta)}\vartheta_t(\eta)d\eta d\alpha.
    \end{aligned}
\end{equation*}
Thus writing in the first term $e^{i\alpha} e^{i\vartheta(\alpha)}\vartheta_t(\alpha)=\partial_\alpha\int_0^\alpha e^{i\eta} e^{i\vartheta(\eta)}\vartheta_t(\eta)$ and integrating by parts we obtain that
\begin{equation}\label{auxL}
    \begin{aligned}
    L'(t)=\frac{-1}{2\pi R^2}\Big(\frac{L(t)}{2\pi}\Big)^3\imag \int_{-\pi}^\pi i e^{i\alpha} e^{i\vartheta(\alpha)}\int_0^\alpha  e^{-i\eta} e^{-i\vartheta(\eta)}\vartheta_t(\eta)d\eta d\alpha.
    \end{aligned}
\end{equation}
Using the equation for $\vartheta$ in \eqref{eqcurve3}$_1$, we have
\begin{equation*}
    \begin{aligned}
     \frac{L(t)}{2\pi}&\int_0^\alpha  e^{-i\eta} e^{-i\vartheta(\eta)}\vartheta_t(\eta)d\eta=
    \int_0^\alpha  e^{-i\eta-i\vartheta(\eta)}\big(U_\alpha(\eta)+ T(\eta)(1+\vartheta_\alpha(\eta))\big)
     d\eta\\
     &=e^{-i\alpha-i\vartheta(\alpha)}U(\alpha)-e^{-i\vartheta(0)}U(0)+i\int_0^\alpha  e^{-i\eta-i\vartheta(\eta)}(1+\vartheta_\alpha(\eta))U(\eta)d\eta\\
     &\quad+i\int_0^\alpha  \partial_\eta\big(e^{-i\eta-i\vartheta(\eta)}\big)T(\eta) d\eta.
    \end{aligned}
\end{equation*}
We then integrate by parts once more in the last term, also using \eqref{T}, to obtain
\begin{equation*}
    \begin{aligned}
    i&\int_0^\alpha  \partial_\eta\big(e^{-i\eta-i\vartheta(\eta)}\big)T(\eta) d\eta=i e^{-i\alpha-i\vartheta(\alpha)}T(\alpha)-i e^{-i\vartheta(0)}T(0)\\
    &-i\int_0^\alpha e^{-i\eta-i\vartheta(\eta)}(1+\vartheta_\alpha(\eta))U(\eta)d\eta +\frac{i}{2\pi}\int_{-\pi}^\pi(1+\theta_\alpha(\eta))U(\eta)d\eta \int_0^\alpha e^{-i\eta-i\vartheta(\eta)}d\eta.
    \end{aligned}
\end{equation*}
Substituting into the previous equation we find that
\begin{equation*}
    \begin{aligned}
     \frac{L(t)}{2\pi}&\int_0^\alpha  e^{-i\eta} e^{-i\vartheta(\eta)}\vartheta_t(\eta)d\eta=
    e^{-i\alpha-i\vartheta(\alpha)}U(\alpha)-e^{-i\vartheta(0)}U(0)+i e^{-i\alpha-i\vartheta(\alpha)}T(\alpha)\\
    &-i e^{-i\vartheta(0)}T(0)
     +\frac{i}{2\pi}\int_{-\pi}^\pi(1+\theta_\alpha(\eta))U(\eta)d\eta \int_0^\alpha e^{-i\eta-i\vartheta(\eta)}d\eta.
    \end{aligned}
\end{equation*}
Thus, plugging this back into \eqref{auxL} and using relation \eqref{constraint} gives
\begin{equation*}
    \begin{aligned}
    L'(t)=\frac{-1}{2\pi R^2}\Big(\frac{L(t)}{2\pi}\Big)^2 \Big(\imag\! \int_{-\pi}^\pi \! \!\!e^{i\alpha+i\vartheta(\alpha)}\! \int_0^\alpha\! \! e^{-i\eta-i\vartheta(\eta)} d\eta d\alpha\Big)\!\int_{-\pi}^\pi\!\! (1\!+\!\vartheta_\alpha(\eta))U(\eta)d\eta,
    \end{aligned}
\end{equation*}
which recalling \eqref{volumelength} implies \eqref{eqcurve3}$_2$. 

We will later show (see Section \ref{sec:Lestimate}) that the condition on the initial data will guarantee that $L(t)>0$ for all time, and thus the formulations using \eqref{eqcurve3}$_2$ and \eqref{Lequation} are equivalent.

In summary, the closed system of equations that define the evolution of the Muskat bubble can be expressed by \eqref{eqcurve3}$_1$ and \eqref{Lequation}, together with \eqref{constraint}.  We will study the evolution of this system to prove our main results.

\section{Calculations Involving the Fourier Transform}\label{sec:FourierTransCalc}

In this section we recall basic calculations for the periodic Fourier transform that will be used in the next section.  In particular we define the Fourier transform of a periodic function $g$ with domain $\mathbb{T}=[-\pi, \pi]$ as: 
\begin{equation}\notag 
    \mathcal{F}(g)(k)\eqdef\widehat{g}(k)=\frac{1}{2\pi}\int_{-\pi}^\pi g(\alpha)e^{-i k\alpha}d\alpha,
\end{equation}
and the corresponding Fourier series
\begin{equation}\notag
    g(\alpha)=\sum_{k\in\mathbb{Z}}\widehat{g}(k)e^{i k\alpha}. 
\end{equation}
For later use, we also define the periodic Hilbert transform as 
\begin{equation}\label{defHilbertTransform}
    \mH(g)(\alpha)\eqdef\frac{1}{2\pi}\pv\int_{-\pi}^{\pi}\frac{g(\alpha-\beta)}{\tan{(\beta/2)}}.
\end{equation}
We notice that $\mH(g)(\alpha)=-\frac{1}{2\pi}\pv\int_{-\pi}^{\pi}\frac{g(\alpha+\beta)}{\tan{(\beta/2)}}$.  Adding half of these together one can calculate that $\mH(c)=0$ if $c\in \mathbb{C}$ is a constant and 
$
\mH(g)(\alpha)=-i\sum_{k \ne 0} \sign(k) \widehat{g}(k)e^{i k\alpha}.
$
And further $\mathcal{F}(\mH(g))(k) = -i \sign(k) \widehat{g}(k)$.  Then we define the operator $\Lambda$ using the Fourier transform as 
$
\mathcal{F}(\Lambda g)(k)\eqdef |k| \widehat{g}(k).
$
And we observe that $\mH(g_\alpha )(\alpha)=\sum_{k\in \mathbb{Z}} |k| \widehat{g}(k)e^{i k\alpha} = \Lambda g$.    And furthermore $$\partial_\alpha \mH(g_{\alpha\alpha} )(\alpha)=-\sum_{k\in \mathbb{Z}} |k|^3 \widehat{g}(k)e^{i k\alpha} = -\Lambda^3 g.$$  
Also one can compute by plugging in the Fourier series that
\begin{equation}\label{fourierCalcAlpha}
        \mathcal{F}\Big(\int_0^\alpha g(\eta)d\eta-\frac{\alpha}{2\pi}\int_{-\pi}^\pi g(\eta)d\eta \Big)(k)=\left\{
    \begin{aligned}-\frac{i}{k} \widehat{g}(k),\quad k\neq0,\\
        \sum_{j\neq0}\frac{i}{j}\widehat{g}(j),\quad k=0,
    \end{aligned}\right.
\end{equation}
These calculations will be used in the next section when we take the Fourier transform of the linearization.

\section{Linearization and Nonlinear Expansion}\label{sec:linearization}

We proceed next to decompose the equation for $\vartheta$ in the system \eqref{T}-\eqref{md} into linear and nonlinear parts. We will Taylor expand the nonlinear terms around the zero frequency of $\vartheta(\alpha)$. Define
\begin{equation}\label{theta}
    \theta(\alpha)=\vartheta(\alpha)-\hat{\vartheta}(0).
\end{equation}
Taking into account that
\begin{equation*}
    \int_0^1 e^{i(\alpha+(s-1)\beta)}ds=e^{i\alpha}\frac{1-e^{-i\beta}}{i\beta},
\end{equation*}
we write the denominator of \eqref{BR} as follows
\begin{multline*}
    \int_0^1 e^{i(\alpha+(s-1)\beta)}e^{i\hat{\vartheta}(0)}e^{i\theta(\alpha+(s-1)\beta)}ds=\\
    e^{i\hat{\vartheta}(0)}e^{i\alpha}\frac{1-e^{-i\beta}}{i\beta}\left(e^{-i\alpha}\frac{i\beta}{1-e^{-i\beta}}\int_0^1 e^{i(\alpha+(s-1)\beta)}e^{i\theta(\alpha+(s-1)\beta)}ds-1+1\right).
\end{multline*}
Then after performing a Taylor expansion, \eqref{BR} is given by 
\begin{multline*}
    \overline{BR}(\omega)(\alpha)=\frac{e^{-i\hat{\vartheta}(0)}}{iL(t)} \pv\int_{-\pi}^{\pi}\frac{\omega(\alpha-\beta)}{\beta e^{i\alpha}\frac{1-e^{-i\beta}}{i\beta}}
    \\\cdot\sum_{n\geq 0}\left(1-\frac{i\beta}{1-e^{-i\beta}}\int_0^1 e^{i(s-1)\beta}e^{i\theta(\alpha+(s-1)\beta)}ds\right)^nd\beta.
\end{multline*}
By further Taylor expanding the exponential term, we find that
\begin{multline*}
    \overline{BR}(\omega)(\alpha)=\frac{e^{-i\hat{\vartheta}(0)}e^{-i\alpha}}{iL(t)}\sum_{n\geq 0}\pv\int_{-\pi}^{\pi}\frac{\omega(\alpha-\beta)}{\beta }\frac{(-1)^n(i\beta)^{n+1}}{(1-e^{-i\beta})^{n+1}}\\
    \cdot\Big(\sum_{m\geq 1}\frac{i^m}{m!}\int_0^1 e^{i(s-1)\beta}(\theta(\alpha+(s-1)\beta))^m ds\Big)^nd\beta.
\end{multline*}
We further Taylor expand $e^{i\theta(\alpha)}$.  Then plugging these expansions into \eqref{U} provides the series for $U(\alpha)$
\begin{multline*}
    U(\alpha)=\real\left(\frac{1}{L(t)}\sum_{n,l\geq0}\frac{(i\theta(\alpha))^l}{l!}\pv\int_{-\pi}^{\pi}\frac{\omega(\alpha-\beta)(-1)^n(i\beta)^{n+1}}{\beta(1-e^{-i\beta})^{n+1}}
    \right.\\
    \left.\cdot\left(\sum_{m\geq 1}\frac{i^m}{m!}\int_0^1 e^{i(s-1)\beta}(\theta(\alpha+(s-1)\beta))^m ds\right)^nd\beta\right).
\end{multline*}
For convenience, we introduce the following notation for the operators $\mathcal{R}$ and $\mathcal{S}$.  We first define $\mathcal{R}$:  
\begin{equation}\label{R}
    \mathcal{R}(f)(\alpha)\!=\!\frac{i}{\pi}\pv\!\!\int_{-\pi}^{\pi}\!\!\!\frac{f(\alpha\!-\!\beta)}{\beta}\frac{\beta^2}{(1\!-\!e^{-i\beta})^2}\!\!\int_0^1\!\!e^{i(s-1)\beta}\theta(\alpha\!+\!(s-1)\beta) ds d\beta.
\end{equation}
Above $\mathcal{R}$ is chosen to be a linear function in $\theta$, it corresponds to $l=0$, $n=1$ and $m=1$ in $U(\alpha)$ above. Then, we further define the operator $\mathcal{S}$ to be the nonlinear in $\theta$ terms inside $U(\alpha)$ above:
\begin{multline}\label{S}
    \mathcal{S}(f)(\alpha)=
    \frac{1}{\pi}\sum_{\substack{n,l\geq0 \\ n+l\geq 2}}\frac{(-1)^n i^{l+n+1}(\theta(\alpha))^l}{l!}\pv\int_{-\pi}^{\pi}\frac{f(\alpha-\beta)\beta^{n+1}}{\beta(1-e^{-i\beta})^{n+1}}\\
    \cdot
    \left(\sum_{m\geq 1}\frac{i^m}{m!}\int_0^1 e^{i(s-1)\beta}(\theta(\alpha+(s-1)\beta))^m ds\right)^n d\beta\\
    +\frac{1}{\pi}\pv\int_{-\pi}^{\pi}\frac{f(\alpha-\beta)\beta^{2}}{\beta(1-e^{-i\beta})^{2}}\sum_{m\geq 2}\frac{i^m}{m!}\int_0^1 e^{i(s-1)\beta}(\theta(\alpha+(s-1)\beta))^m ds d\beta.
\end{multline}
The $\mathcal{S}$ operator corresponds to the terms in $U(\alpha)$ above where $n,l \ge 0$ and $n+l\ge 2$ plus the case where $l=0$, $n=1$ and $m\ge 2$. 

For the cases in $U(\alpha)$ where $n=l=0$ and $n=0$, $l=1$ we further notice that
\begin{equation}\label{hilbert}
    \frac{1}{\pi}\pv\int_{-\pi}^{\pi}\frac{f(\alpha-\beta)}{1-e^{-i\beta}}d\beta=-i\mH f(\alpha)+\hat{f}(0),
\end{equation}
where $\mH f$ denotes the periodic Hilbert transform of $f$ as given in \eqref{defHilbertTransform}.  The previous identity \eqref{hilbert} is obtained multiplying above and below by $1-e^{i\beta}$ and using the trigonometric identities
$1-\cos{(\beta)}=2\sin^2{(\beta/2)}$ and  $\sin{(\beta)}=2\sin{(\beta/2)}\cos{(\beta/2)}.
$
Further $\omega(\alpha)$ in \eqref{omega} can be written as an exact derivative, its mean value is zero and therefore $\hat{\omega}(0)=0$. 
These calculations show that we can write the expression inside \eqref{U} as
\begin{equation}\notag
    i\overline{BR}(\omega)(\alpha) e^{i(\alpha+\vartheta(\alpha))}
    =
    \frac{\pi}{L(t)} \left(i \theta(\alpha) \mathcal{H}(\omega) + \mathcal{H}(\omega) +\mathcal{R}(\omega)(\alpha)+\mathcal{S}(\omega)(\alpha)\right).
\end{equation}
Thus, noticing that the term with $n=0$, $l=1$ vanishes in $U(\alpha)$ in \eqref{U} because it is purely imaginary, using the notation above we can write $U(\alpha)$ in the following manner
\begin{equation}\label{Udecomp}
    U(\alpha)=\frac{\pi}{L(t)}\Big(\mH \omega(\alpha)+ \real\hspace{0.05cm}\mathcal{R}(\omega)(\alpha)+ \real\hspace{0.05cm}\mathcal{S}(\omega)(\alpha)\Big).
\end{equation}
Proceeding similarly, in \eqref{md}, one finds that
\begin{equation}\label{Ddecomp}
    \md(\omega)(\alpha)=\frac{-\pi}{L(t)}\Big(\theta(\alpha)\mH \omega(\alpha)+ \imag\hspace{0.05cm}\mathcal{R}(\omega)(\alpha)+ \imag\hspace{0.05cm}\mathcal{S}(\omega)(\alpha)\Big).
\end{equation}
We shall now split all the terms into zero, first or higher order polynomials of $\theta(\alpha)$. First, the vorticity strength \eqref{omega} is split as follows
\begin{equation*}
    \omega(\alpha)=\omega_0(\alpha)+\omega_1(\alpha)+\omega_{\geq2}(\alpha),
\end{equation*}
where
\begin{equation}\label{omegasplit}
\left\{\begin{aligned}
    \omega_0(\alpha)&\!=\!-A_\rho\frac{L(t)}{\pi}\sin{(\alpha+\hat{\vartheta}(0))},\\
    \omega_1(\alpha)&\!=\!A_\mu\frac{L(t)}{\pi}\md_1(\omega_0)(\alpha)\!+\!2A_\sigma\frac{2\pi}{L(t)}\theta_{\alpha\alpha}\!
    \\
    &\hspace{1.6cm}-\!A_\rho\frac{L(t)}{\pi}\cos{(\alpha\!+\!\hat{\vartheta}(0))}\theta(\alpha),\\
    \omega_{\geq2}(\alpha)&\!=\!A_\mu\frac{L(t)}{\pi}\md_{\geq2}(\omega)(\alpha)\!
    \\ &\hspace{1.6cm} -\!A_\rho\frac{L(t)}{\pi}\sin{(\alpha\!+\!\hat{\vartheta}(0))}\sum_{j\geq1}\!\!\frac{(-1)^j(\theta(\alpha))^{2j}}{(2j)!}\\
    &\hspace{1.6cm}-A_\rho\frac{L(t)}{\pi}\cos{(\alpha+\hat{\vartheta}(0))}\sum_{j\geq1}\frac{(-1)^j(\theta(\alpha))^{1+2j}}{(1+2j)!},\\
    \omega_{\geq1} &\!=\! \omega_{1}+\omega_{\geq2}.
\end{aligned}\right.
\end{equation}
Above we used the trigonometric identity
$\sin(a+b) = \sin(a)\cos(b)+\cos(a)\sin(b)$,
as well as the Taylor expansions for sine and cosine. 
Then $\md_1(\omega_0)(\alpha)$ and $\md_{\geq2}(\omega)(\alpha)$ are obtained, in turn, by introducing \eqref{omegasplit} into \eqref{Ddecomp} as follows
\begin{equation*}
    \md(\omega)(\alpha)=\md_1(\omega_0)(\alpha)+\md_{\geq2}(\omega)(\alpha),
\end{equation*}
where
\begin{equation}\label{mdsplit}
\left\{\begin{aligned}
    \md_1(\omega_0)(\alpha)&=\frac{-\pi}{L(t)}\Big(\theta(\alpha)\mH \omega_0(\alpha)+\imag\hspace{0.05cm}\mR(\omega_0)(\alpha)\Big),\\
    \md_{\geq2}(\omega)(\alpha)&=\frac{-\pi}{L(t)}\Big(\theta(\alpha)\mH \omega_{\geq1}(\alpha)+\imag\hspace{0.05cm} \mR(\omega_{\geq1})(\alpha)+\imag\hspace{0.05cm}\mS(\omega)(\alpha)\Big).
\end{aligned}\right.
\end{equation}
Analogously, the splitting for $U(\alpha)$ from \eqref{Udecomp} is
\begin{equation*}
    U(\alpha)=U_0(\alpha)+U_1(\alpha)+U_{\geq2}(\alpha),
\end{equation*}\vspace{-0.5cm}
with
\begin{equation}\label{Usplit}
\left\{\begin{aligned}
    U_0(\alpha)&=\frac{\pi}{L(t)}\mH \omega_0(\alpha),\\
    U_1(\alpha)&=\frac{\pi}{L(t)}\Big(\mH \omega_1(\alpha)+\real\hspace{0.05cm}\mR(\omega_0)(\alpha)\Big),\\
    U_{\geq2}(\alpha)&=\frac{\pi}{L(t)}\Big(\mH \omega_{\geq2}(\alpha)+\real\hspace{0.05cm} \mR(\omega_{\geq1})(\alpha)+\real\hspace{0.05cm}\mS(\omega)(\alpha)\Big).
\end{aligned}\right.
\end{equation}
Recalling the expression for $T(\alpha)$ in \eqref{T}, we find that
\begin{equation*}
    T(\alpha)=T_0(\alpha)+T_1(\alpha)+T_{\geq2}(\alpha),
\end{equation*}
where, using $\int_{-\pi}^{\pi} U_0(\eta) d\eta =0$, we have
\begin{equation}\label{Tsplit}
\left\{\begin{aligned}
    T_0(\alpha)=& T(0)+\int_0^\alpha U_0(\eta)d\eta,\\
    T_1(\alpha)=& \int_0^\alpha  U_1(\eta)d\eta-\frac{\alpha}{2\pi}\int_{-\pi}^{\pi}U_1(\eta)d\eta \\
    &+\int_0^\alpha\theta_\alpha(\eta)U_0(\eta)d\eta-\frac{\alpha}{2\pi}\int_{-\pi}^{\pi}\theta_\alpha(\eta)U_0(\eta)d\eta,\\
    T_{\geq2}(\alpha)=&\int_0^\alpha (1+\theta_\alpha(\eta))U_{\geq2}(\eta)d\eta-\frac{\alpha}{2\pi}\int_{-\pi}^{\pi}(1+\theta_\alpha(\eta))U_{\geq2}(\eta)d\eta\\
    &+\int_0^\alpha\theta_\alpha(\eta)U_1(\eta)d\eta-\frac{\alpha}{2\pi}\int_{-\pi}^{\pi}\theta_\alpha(\eta)U_1(\eta)d\eta.
\end{aligned}\right.
\end{equation}
These are all the splittings that we will use in the following.

We first examine the zero order terms from \eqref{thetaeqaux}$_1$.   The zero order terms on the right side of the equality \eqref{thetaeqaux}$_1$ would be
\begin{equation}\notag
        \Theta(\alpha)=(U_0)_\alpha(\alpha)+T_0(\alpha).
\end{equation}
Now a direct calculation from \eqref{omegasplit} shows that 
\begin{equation}\label{hilbertTcalc}
    \mathcal{H}(\omega_0)(\alpha) 
    = A_\rho\frac{L(t)}{\pi}\cos{(\alpha+\hat{\vartheta}(0))}.
\end{equation}
Then we plug this into \eqref{Usplit}$_1$ and \eqref{Tsplit}$_1$ to obtain
\begin{equation}\label{u0t0}
\left\{
\begin{aligned}
    U_0(\alpha)&=A_\rho \cos{(\alpha+\hat{\vartheta}(0))},\\
     T_0(\alpha)&=A_\rho\sin{(\alpha+\hat{\vartheta}(0))}-A_\rho\sin{\hat{\vartheta}(0)}+T(0).
\end{aligned}\right.
\end{equation}
In particular then the zero order term $\Theta(\alpha)$ does not depend on $\alpha$,
\begin{equation*}
    \Theta(\alpha)=-A_\rho \sin{\hat{\vartheta}(0)}+T(0).
\end{equation*}
Now we choose 
\begin{equation}\label{T0}
    T(0)=A_\rho \sin{\hat{\vartheta}(0)}.
\end{equation}
Thus the parametrization of the circle solution is independent of time (see Proposition \ref{circles} below).  Further $\Theta(\alpha)=0.$

Now we introducing the splittings \eqref{Usplit} and \eqref{Tsplit} into the equation for $\vartheta$ in \eqref{eqcurve3}, we find that
\begin{equation}\label{system}
\left\{
\begin{aligned}
    \vartheta_t(\alpha)&=\frac{2\pi}{L(t)}\Big(\mathcal{L}(\alpha)+N(\alpha)\Big),\\
    \mathcal{L}(\alpha)&=(U_1)_\alpha(\alpha)+T_1(\alpha)+T_0(\alpha)\theta_\alpha(\alpha),\\
    N(\alpha)&=(U_{\geq 2})_\alpha(\alpha)+T_{\geq 2}(\alpha)(1+\theta_\alpha(\alpha))+T_1(\alpha)\theta_\alpha(\alpha).
\end{aligned}\right.
\end{equation}
Now we will expand the linear terms in $\mathcal{L}(\alpha)$ in \eqref{system}.  To do this we first split $U_1(\alpha)$ in \eqref{Usplit} into parts corresponding to the parameters $A_\rho$, $A_\sigma$ and $A_\mu$ respectively as  
\begin{equation*}
    U_1(\alpha)=A_\rho U_{1\rho}(\alpha)+4A_\sigma\frac{\pi^2}{(L(t))^2} U_{1\sigma}(\alpha)+A_\mu A_\rho U_{1\mu}(\alpha),
\end{equation*}
To calculate these terms we plug $\omega_0(\alpha)$ and $\omega_1(\alpha)$ from \eqref{omegasplit} into $U_1(\alpha)$ in \eqref{Usplit} using also $\md_1(\omega_0)(\alpha)$ from \eqref{mdsplit} and \eqref{hilbertTcalc}.  
We obtain
\begin{equation}\label{uterms}
\left\{
\begin{aligned}
    U_{1\rho}(\alpha)&=-\mH\big(\theta(\alpha)\cos{(\alpha+\hat{\vartheta}(0))}\big)-\real\hspace{0.05cm}\mR(\sin{(\alpha+\hat{\vartheta}(0))}),\\
    U_{1\sigma}(\alpha)&=\mH \theta_{\alpha\alpha}(\alpha),\\
    U_{1\mu}(\alpha)&= - \mH \big(\theta(\alpha)\cos{(\alpha+\hat{\vartheta}(0))}\big)+\mH\hspace{0.03cm}\imag\hspace{0.05cm}\mR(\sin{(\alpha+\hat{\vartheta}(0))}).
\end{aligned}\right.
\end{equation}
We analogously write the linear part, $\mathcal{L}(\alpha)$ in \eqref{system}, 
as follows
\begin{equation*}
    \mathcal{L}(\alpha)=A_\rho \mathcal{L}_\rho(\alpha)+4A_\sigma\frac{\pi^2}{(L(t))^2}\mathcal{L}_\sigma(\alpha)+A_\mu A_\rho\mathcal{L}_\mu(\alpha),
\end{equation*}
where
\begin{equation}\label{linearterms}
\left\{
\begin{aligned}
    \mathcal{L}_{\rho}(\alpha)&=(U_{1\rho})_\alpha(\alpha)+\int_0^\alpha U_{1\rho}(\eta)d\eta-\frac{\alpha}{2\pi}\int_{-\pi}^{\pi}U_{1\rho}(\eta)d\eta
    \\
    &\quad+\int_0^{\alpha}\theta_\alpha(\eta)\cos{(\eta+\hat{\vartheta}(0))}d\eta-\frac{\alpha}{2\pi}\int_{-\pi}^{\pi}\theta_\alpha(\eta)\cos{(\eta+\hat{\vartheta}(0))}d\eta\\
    &\quad+\theta_\alpha(\alpha)\sin{(\alpha+\hat{\vartheta}(0))},\\
    \mathcal{L}_\sigma(\alpha)&=-\Lambda^3 \theta(\alpha)+\int_0^\alpha \mH \theta_{\alpha\alpha}(\eta)d\eta,\\
    \mathcal{L}_\mu(\alpha)&=(U_{1\mu})_\alpha(\alpha)+\int_0^\alpha U_{1\mu}(\eta)d\eta.
\end{aligned}\right.
\end{equation}
Here we used that $\partial_\alpha \mathcal{H}\theta_{\alpha\alpha} = -\Lambda^3 \theta(\alpha)$.  We also used that 
$$\int_{-\pi}^{\pi}\mH \theta_{\alpha\alpha}(\eta)d\eta = \int_{-\pi}^{\pi}U_{1\mu}(\eta)d\eta =0$$ since both integrals are of a Hilbert transform and thus have zero value for the zero Fourier frequency.
This completes our decomposition of the equation \eqref{eqcurve} into \eqref{system}.  

In the following, we explain the steady states circles for equation \eqref{eqcurve} using the reformulation of the equations given above.  

\begin{prop}\label{circles}
	A circle of radius $R$, defined by \eqref{circles} and \eqref{modul} with $\vartheta(\alpha)=\widehat{\vartheta}(0)$ constant in time and $L(t)=2\pi R$, is a time-independent solution of \eqref{T}-\eqref{md} with $T(0)$ given by \eqref{T0}.
	It corresponds to the solution of \eqref{eqcurve} given by a circle of radius $R$ moving vertically with velocity $A_\rho$.	
\end{prop}

\begin{proof}
For $\vartheta(\alpha)=\widehat{\vartheta}(0)$, all the linear and nonlinear terms in the decompositions \eqref{omegasplit}-\eqref{Tsplit} are zero. Thus, with $L(t)=2\pi R$,  as in \eqref{u0t0} with \eqref{T0} we have
\begin{align*}
    U(\alpha)&=U_0(\alpha)=A_\rho\cos{(\alpha+\widehat{\vartheta}(0))},\\
    T(\alpha)&=T_0(\alpha) = A_\rho\sin{(\alpha+\widehat{\vartheta}(0))}.
\end{align*} 
Both equations in \eqref{eqcurve3} are then trivially satisfied; equation \eqref{eqcurve3}$_1$ is decomposed as  \eqref{system} with $\mathcal{L}(\alpha)=N(\alpha)=0$.	
	
Then we integrate \eqref{param} to obtain
\begin{equation*}
	z(\alpha,t)=z(0,t)+R\int_0^\alpha e^{i(\eta+\widehat{\vartheta}(0))} d\eta.
\end{equation*}
We differentiate the above in time, and then use \eqref{eqcurve.theta} to obtain
\begin{equation*}
\begin{aligned}
  	z_t(\alpha,t)&=z_t(0,t)=U(0,t)n(0,t)+T(0,t)\tau(0,t)\\
	&=A_\rho\cos{(\widehat{\vartheta}(0))}ie^{i\widehat{\vartheta}(0)}+A_\rho\sin{(\widehat{\vartheta}(0))}e^{i\widehat{\vartheta}(0)}=0+i A_\rho.
\end{aligned}	
\end{equation*}	
This completes the proof.
\end{proof}

Next, we compute the Fourier transform of the linearized system. Because the function $\theta(\alpha)$ is real and has zero average, we only need to compute the positive frequencies.

\begin{prop}\label{linearfourier}
	(Linear system in Fourier variables.)
	For $k\geq 1$, $k\neq2$, the Fourier transform of the linear terms \eqref{linearterms} are given by  
	\begin{equation*}
	\begin{aligned}
	    \widehat{\mathcal{L}}(k)&=-A_\sigma \frac{4\pi^2}{L(t)^2}k(k^2\!-\!1)\hat{\theta}(k)\!-\!(1\!+\!A_\mu)A_\rho\frac{(k^2\!-\!1)(k\!+\!1)}{k(k\!+\!2)}e^{-i\hat{\vartheta}(0)}\hat{\theta}(k+1),	
	\end{aligned}
	\end{equation*}
	and for $k=2$,
	\begin{equation*}
	\begin{aligned}
	    \widehat{\mathcal{L}}(2)&=-A_\sigma \frac{4\pi^2}{L(t)^2}6\hat{\theta}(2)+(1-A_\mu)A_\rho\frac{3}{2}\left(\frac34-\log{2}\right)e^{i\hat{\vartheta}(0)}\hat{\theta}(1)\\
	    &\quad-(1+A_\mu)A_\rho\frac{9}{8}e^{-i\hat{\vartheta}(0)}\hat{\theta}(3).
	\end{aligned}
	\end{equation*}
\end{prop}

\begin{proof}
First, we note that, for a general function $f(\alpha)$ and $k\neq0$ we have \eqref{fourierCalcAlpha}.
Therefore, for $k\geq 1$, the Fourier coefficients of
\eqref{linearterms} are given by
\begin{equation}\label{Lfourier}
\begin{aligned}
    \widehat{\mathcal{L}}_\sigma(k)&=-k(k^2-1)\hat{\theta}(k),\\
    \widehat{\mathcal{L}}_\mu(k)&=i\left(k-\frac1{k}\right)\widehat{U}_{1\mu}(k),\\
    \widehat{\mathcal{L}}_\rho(k)
    &=
    \left(k-\frac1{k}\right)
    \left(i\widehat{U}_{1\rho}(k)+\frac{e^{i\hat{\vartheta}(0)}}{2}\hat{\theta}(k-1)-\frac{e^{-i\hat{\vartheta}(0)}}{2}\hat{\theta}(k+1)\right),
\end{aligned}
\end{equation}
so it remains to compute $\widehat{U}_{1\mu}(k)$ and $\widehat{U}_{1\rho}(k)$.
From \eqref{uterms}, we can write 
\begin{align*}
    \widehat{U}_{1\mu}(k)
    =&
    \frac{i}{2}\Big(e^{i\hat{\vartheta}(0)}\hat{\theta}(k\!-\!1)\!
    +\!e^{-i\hat{\vartheta}(0)}\hat{\theta}(k\!+\!1)\Big)\!
    \\
    &-\!i\mathcal{F}\Big(\imag\hspace{0.05cm} \mathcal{R}(\sin{(\alpha\!+\!\hat{\vartheta}(0))})\Big)(k),
\end{align*}
and
\begin{align*}
    \widehat{U}_{1\rho}(k)
    =&\frac{i}2 \Big(e^{i\hat{\vartheta}(0)}\hat{\theta}(k\!-\!1)\!+\!e^{-i\hat{\vartheta}(0)}\hat{\theta}(k\!+\!1)\Big)\!
    \\
    &-\!\mathcal{F}\Big(\real\hspace{0.05cm} \mathcal{R}(\sin{(\alpha\!+\!\hat{\vartheta}(0))})\Big)(k).
\end{align*}
Recalling the expression of $\mathcal{R}$ in \eqref{R}, we have that
\begin{multline*}
    \mathcal{F}\Big(\imag\hspace{0.05cm} \mathcal{R}(\sin{(\alpha+\hat{\vartheta}(0))})\Big)(k)=\\
    \frac1{\pi}\pv\!\int_{-\pi}^{\pi}\int_0^1\!\! \imag\left(\!\frac{i\beta e^{i(s-1)\beta}}{(1-e^{-i\beta})^2}\!\right)\!\mathcal{F}\Big(\theta(\alpha+(s-1)\beta)\sin{(\alpha\!-\!\beta\!+\!\hat{\vartheta}(0))}\Big)(k)ds d\beta.
\end{multline*}
Using that
\begin{equation*}
    \imag\left(\frac{i\beta e^{i(s-1)\beta}}{(1-e^{-i\beta})^2}\right)=-\frac{\beta\cos{(\beta s)}}{4\sin^2{(\beta/2)}},
\end{equation*}
and computing the Fourier transform inside the integral, 
we obtain that
\begin{multline}\notag
    \mathcal{F}\Big(\imag\hspace{0.05cm} \mathcal{R}(\sin\!{(\alpha\!+\!\hat{\vartheta}(0))})\!\Big)\!(k)=\\
    -\frac{ e^{i\hat{\vartheta}(0)}}{2\pi}\hat{\theta}(k\!-\!1)\pv\!\!\int_{-\pi}^{\pi}\!\int_0^1 \!\!\!\frac{\beta\cos{(\beta s)}}{4\sin^2{(\beta/2)}}\sin{((k\!-\!1)(s\!-\!1)\beta\!-\!\beta)}ds d\beta\\
    +\frac{ e^{-i\hat{\vartheta}(0)}}{2\pi}\hat{\theta}(k\!+\!1)\pv\!\!\int_{-\pi}^{\pi}\!\int_0^1 \!\!\!\frac{\beta\cos{(\beta s)}}{4\sin^2{(\beta/2)}}\sin{((k\!+\!1)(s\!-\!1)\beta\!+\!\beta)}ds d\beta.	
\end{multline}
Taking into account that
\begin{equation*}
    \real\left(\frac{i\beta e^{i(s-1)\beta}}{(1-e^{-i\beta})^2}\right)=\frac{\beta\sin{(\beta s)}}{4\sin^2{(\beta/2)}},
\end{equation*}
and proceeding analogously, the following expression is found for the real part: 
\begin{multline}\notag 
    \mathcal{F}\Big(\real\hspace{0.05cm} \mathcal{R}(\sin\!{(\alpha\!+\!\hat{\vartheta}(0))})\!\Big)\!(k)=\\
    -\frac{i e^{i\hat{\vartheta}(0)}}{2\pi}\hat{\theta}(k\!-\!1)
    \pv\!\!\int_{-\pi}^{\pi}\!\int_0^1 \!\!\!\frac{\beta\sin{(\beta s)}}{4\sin^2{(\beta/2)}}\cos{((k\!-\!1)(s\!-\!1)\beta\!-\!\beta)}ds d\beta
    \\
    +\frac{i e^{-i\hat{\vartheta}(0)}}{2\pi}\hat{\theta}(k\!+\!1)\pv\!\!\int_{-\pi}^{\pi}\!\int_0^1 \!\!\!\frac{\beta\sin{(\beta s)}}{4\sin^2{(\beta/2)}}\cos{((k\!+\!1)(s\!-\!1)\beta\!+\!\beta)}ds d\beta.	
\end{multline}
The above integrals are calculated in Lemma \ref{lemmaI1I2} below. Plugging in their values, we have 
\begin{multline}\label{fourierimag}
    \mathcal{F}\Big(\imag\hspace{0.05cm} \mathcal{R}(\sin\!{(\alpha\!+\!\hat{\vartheta}(0))})\!\Big)\!(k)=
    \frac{ e^{-i\hat{\vartheta}(0)}}{2\pi}\hat{\theta}(k\!+\!1)\frac{-k\pi}{2+k} 1_{k \ge 1}\\
    -\frac{ e^{i\hat{\vartheta}(0)}}{2\pi}\hat{\theta}(k\!-\!1)
        \left(-\pi 1_{k \ge 1, k\ne 2}+ \pi \left(\frac{1}{2}-\log 4 \right)1_{k =2} \right),	
\end{multline}
and
\begin{multline}\label{fourierreal}
    \mathcal{F}\Big(\real\hspace{0.05cm} \mathcal{R}(\sin\!{(\alpha\!+\!\hat{\vartheta}(0))})\!\Big)\!(k)=
    -\frac{i e^{i\hat{\vartheta}(0)}}{2\pi}\hat{\theta}(k\!-\!1)\pi \left(\log 4-\frac{3}{2} \right)1_{k =2} 
    \\
    +\frac{i e^{-i\hat{\vartheta}(0)}}{2\pi}\hat{\theta}(k\!+\!1)
    \frac{2\pi}{2+k} 1_{k \ge 1}.	
\end{multline}
We conclude the Fourier transform of $U_{1\mu}$ and $U_{1\rho}$ are given by
\begin{equation*}
    \widehat{U}_{1\mu}(k)=
    \left\{
    \begin{aligned}
        &\frac{i}2e^{-i\hat{\vartheta}(0)}\Big(1+\frac{k}{2+k}\Big)\hat{\theta}(k+1),\qquad\hspace{1.8cm} k=1,3,4,\dots,\\
        &\frac{i}2e^{i\hat{\vartheta}(0)}\Big(\frac32-\log{(4)}\Big)\hat{\theta}(1)+\frac{3i}4e^{-i\hat{\vartheta}(0)}\hat{\theta}(3),\qquad k=2,
    \end{aligned}\right.
\end{equation*}	
and 
\begin{equation*}
    \widehat{U}_{1\rho}(k)=
    \left\{
    \begin{aligned}
        &\frac{i}2e^{i\hat{\vartheta}(0)}\hat{\theta}(k\!-\!1)\!+\!\frac{i}2e^{-i\hat{\vartheta}(0)}\Big(1\!-\!\frac{2}{2\!+\!k}\Big)\hat{\theta}(k\!+\!1),\hspace{0.3cm} k=1,3,4,\dots,\\
        &\frac{i}2e^{i\hat{\vartheta}(0)}\Big(-\frac12+\log{4}\Big)\hat{\theta}(1)+\frac{i}4e^{-i\hat{\vartheta}(0)}\hat{\theta}(3),\qquad\hspace{0.4cm} k=2.
    \end{aligned}\right.
\end{equation*}	
Substituting these expressions into \eqref{Lfourier} gives that
\begin{equation*}
    \widehat{\mathcal{L}}_{\mu}(k)=
    \left\{
    \begin{aligned}
        &-e^{-i\widehat{\vartheta}(0)}\frac{(k^2-1)(k+1)}{k(k+2)}\widehat{\theta}(k+1),\hspace{0.4cm} k=1,3,4,\dots,\\
        &-\frac34\Big(\frac32-\log{4}\Big)e^{i\widehat{\vartheta}(0)}\widehat{\theta}(1)-\frac98e^{-i\widehat{\vartheta}(0)}\widehat{\theta}(3),\qquad\hspace{0.4cm} k=2,
    \end{aligned}\right.
\end{equation*}
and
\begin{equation*}
    \widehat{\mathcal{L}}_{\rho}(k)=
    \left\{
    \begin{aligned}
        &-e^{i\widehat{\vartheta}(0)}\frac{(k^2-1)(k+1)}{k(k+2)}\widehat{\theta}(k-1),\hspace{0.4cm} k=1,3,4,\dots,\\
        &\frac34\Big(\frac32-\log{4}\Big)e^{i\widehat{\vartheta}(0)}\widehat{\theta}(1)-\frac98e^{-i\widehat{\vartheta}(0)}\widehat{\theta}(3),\qquad\hspace{0.4cm} k=2,
    \end{aligned}\right.
\end{equation*}
Finally, adding them according to \eqref{linearterms}, the result follows.
\end{proof}

\begin{lemma}\label{lemmaI1I2}
For $k\in\mathbb{Z}\setminus\{0\}$, define the integrals 
\begin{equation*} 
	I_1(k)=\int_{-\pi}^{\pi}\!\int_0^1\!\frac{\beta\cos{(\beta s)}}{4\sin^2{(\beta/2)}}\sin{((k\!-\!1)(s\!-\!1)\beta\!-\!\beta)}dsd\beta,
\end{equation*}
\begin{equation*}
	I_2(k)=\int_{-\pi}^{\pi}\!\int_0^1\!\frac{\beta\sin{(\beta s)}}{4\sin^2{(\beta/2)}}\cos{((k\!-\!1)(s\!-\!1)\beta\!-\!\beta)}dsd\beta.
\end{equation*}
For $k\geq 1$ and $k\neq 2$,
\begin{equation*}
	I_1(k)=-\pi,\qquad I_2(k)=0,
\end{equation*}
while for $k\leq -1$. 
\begin{equation*}
	I_1(k)=\frac{-k\pi}{2-k},\qquad I_2(k)=\frac{2\pi}{2-k}.
\end{equation*}	
The value $k=2$ is given by
\begin{equation*}
	I_1(2)=\pi\Big(\frac12-\log{4}\Big),\qquad I_2(2)=\pi\Big(\log{4}-\frac32\Big).
\end{equation*}
\end{lemma}

\begin{remark}
Notice that
\begin{equation*} 
	I_1(-k)=-\int_{-\pi}^{\pi}\!\int_0^1\!\frac{\beta\cos{(\beta s)}}{4\sin^2{(\beta/2)}}\sin{((k\!+\!1)(s\!-\!1)\beta\!+\!\beta)}dsd\beta,
\end{equation*}
and 
\begin{equation*}
	I_2(-k)=\int_{-\pi}^{\pi}\!\int_0^1\!\frac{\beta\sin{(\beta s)}}{4\sin^2{(\beta/2)}}\cos{((k\!+\!1)(s\!-\!1)\beta\!+\!\beta)}dsd\beta.
\end{equation*}
Thus Lemma \ref{lemmaI1I2} covers all the integrals in \eqref{fourierimag} and \eqref{fourierreal}.
\end{remark}

\begin{proof}
	Both integrals are computed similarly. We only show the details for $I_1(k)$. 
	First, consider the case $k\neq2$.
    Using complex exponentials, we write the numerator as follows
    \begin{equation*}
        \begin{aligned}
            \cos{(\beta s)} \sin{((k\!-\!1)(s\!-\!1)\beta\!-\!\beta)}&=   \frac{1}{4i}\Big(e^{ik(s-1)\beta}-e^{-i(k(s-1)\beta-2s\beta)}\\
            &\quad+e^{i(k(s-1)\beta-2s\beta)}-e^{-i(k(s-1)\beta)}\Big).   
        \end{aligned}
    \end{equation*}
    Thus integration in $s$ gives that
    \begin{equation*}
        \begin{aligned}
            I_1(k)&\!=\!\frac{-1}{16}\!\int_{-\pi}^{\pi}\!\Big(\frac{1\!-\!e^{-i k\beta}}{k}\!+\!\frac{e^{2i\beta}\!-\!e^{i k\beta}}{k-2}\!+\!\frac{e^{-2i\beta}\!-\!e^{-i k\beta}}{k-2}\!+\!\frac{1\!-\!e^{i k\beta}}{k}\Big)\frac{d\beta}{\sin^2{(\beta/2)}}\\
            &\!= \!\frac{-1}{16}\!\int_{-\pi}^{\pi}\!\Big(\frac{2\!-\!e^{i k\beta}\!-\!e^{-i k\beta}}{k}\!+\!\frac{e^{2i\beta}\!+\!e^{-2i \beta}}{k-2}\!-\!\frac{e^{i k\beta}\!+\!e^{-i k\beta}}{k-2}\Big)\frac{d\beta}{\sin^2{(\beta/2)}}.
        \end{aligned}
    \end{equation*}
We then write the denominator in complex form too
\begin{equation*}
    \begin{aligned}
        \sin^2{(\beta/2)}=\frac{-1}4\big(e^{i\beta/2}-e^{-i\beta/2}\big)^2,
    \end{aligned}
\end{equation*}
and formally expand it
\begin{equation*}
    \begin{aligned}
        \big(\sin{\beta/2)}\big)^{-2}\!=-4e^{-i\beta}\big(1\!-\!e^{-i\beta}\big)^{-2}\!=-4e^{-i\beta}\sum_{l\geq1} l e^{-i(l-1)\beta}\!=\!-4\sum_{l\geq1}l e^{-i l\beta},
    \end{aligned}
\end{equation*}
where we have used that\begin{equation*}
\frac{1}{(1-x)^2}=\sum_{l\geq1} l x^{l-1}.    
\end{equation*}
Therefore, for $k\neq0,2$, we have
 \begin{equation*}
        \begin{aligned}
            I_1(k)\!= \!\frac{1}{4}\sum_{l\geq1}l\!\int_{-\pi}^{\pi}\!\Big(&\frac{2e^{-i l\beta}\!-\!e^{i (k-l)\beta}\!-\!e^{-i (k+l)\beta}}{k}\!+\!\frac{e^{i(2-l)\beta}\!+\!e^{-i(2+l) \beta}}{k-2}\\
            &\quad-\frac{e^{i (k-l)\beta}\!+\!e^{-i( k+l)\beta}}{k-2}\Big)d\beta.
        \end{aligned}
    \end{equation*}
Thus performing the integral in $\beta$ we obtain that
\begin{equation*}
        \begin{aligned}
            I_1(k)&\!= \!\frac{1}{4}\sum_{l\geq1}l\Big(\frac{2\pi}{k}\big(-\delta(k-l)-\delta(k+l)\big)\\
            &\hspace{1.7cm}+\!\frac{2\pi}{k\!-\!2}\big(\delta(2\!-\!l)\!+\!\delta(2\!+\!l)\!-\!\delta(k\!-\!l)\!-\!\delta(k\!+\!l)\big)\Big)\\
            &=\frac{1}{4}\Big(-2\pi \sign(k)+\frac{2\pi}{k-2}2-\frac{2\pi}{k-2}|k|\Big),
        \end{aligned}
    \end{equation*}
	which gives
	\begin{equation*}
        \begin{aligned}
            I_1(k)&\!=\frac{\pi}{2}\Big(\!- \sign(k)\!+\!\frac{2\!-\!|k|}{k\!-\!2}\Big)=\frac{\pi}{2}\frac{-2|k|\!+\!2(\sign(k)\!+\!1)}{k\!-\!2}.
        \end{aligned}
    \end{equation*}
	It follows then that for $k\geq1$ ($k\neq2$),
	\begin{equation*}
	    I_1(k)=-\pi,
	\end{equation*}
    and for $k\leq-1$,
    \begin{equation*}
        I_1(k)=\frac{\pi}2\frac{2k}{k-2}=\pi\frac{|k|}{2+|k|}.
    \end{equation*}
    The above computations can be justified by writing $x=\lambda e^{-i\beta}$ with $0<\lambda<1$. Then, $\big(1-\lambda e^{-i\beta}\big)^{-2}=\sum_{l\geq1} l\lambda^{l-1}e^{-i(l-1)\beta}$ converges uniformly and one can repeat the steps above and take the limit $\lambda\to1$.
    
    Lastly, we deal with the case $k=2$. We first rewrite it as follows
	\begin{equation*}
	\begin{aligned}
	    I_1(2)=\int_{-\pi}^{\pi}\int_0^1\frac{\beta}{8\sin^2{(\beta/2)}}\left(\sin{(2\beta s-2\beta)}-\sin{(2\beta)}\right)ds d\beta,
	\end{aligned}
	\end{equation*}
	so after integration in $s$ we obtain
	\begin{equation*}
	\begin{aligned}
	    I_1(2)=\int_{-\pi}^{\pi}\frac{\beta}{8\sin^2{(\beta/2)}}\left(\frac{-1+\cos{(2\beta)}}{2\beta}-\sin{(2\beta)}\right)ds d\beta.
	\end{aligned}
	\end{equation*}
	Using repeatedly the double angle formula, we find that
	\begin{equation*}
	\begin{aligned}
	    I_1(2)&=-\frac12\int_{-\pi}^{\pi}\cos^2{(\beta/2)}d\beta-\frac12\int_{-\pi}^{\pi}\frac{\beta\cos{(\beta)}\cos{(\beta/2)}}{\sin{(\beta/2)}}d\beta	\\
	    &=-\frac{\pi}2-\frac12\int_{-\pi}^{\pi}\frac{\beta\cos{(\beta)}\cos{(\beta/2)}}{\sin{(\beta/2)}}d\beta,
	\end{aligned}
	\end{equation*}
	which can be further simplified
	\begin{equation*}
	\begin{aligned}
	    I_1(2)&=-\frac{\pi}2+\int_{-\pi}^{\pi}\beta\sin{(\beta/2)}\cos{(\beta/2)}d\beta -\frac12\int_{-\pi}^{\pi}\frac{\beta\cos{(\beta/2)}}{\sin{(\beta/2)}}d\beta\\
	    &=\frac{\pi}2-\frac12\int_{-\pi}^{\pi}\frac{\beta\cos{(\beta/2)}}{\sin{(\beta/2)}}d\beta.
	\end{aligned}
	\end{equation*}
	Integration by parts gives that
	\begin{equation*}
	    I_1(2)=\frac{\pi}2+2\int_0^\pi \log{|\sin{(\beta/2)}|}d\beta.
	\end{equation*}
	This integral above is related to the known \text{Clausen function} \cite{Cl32} of order two (whose value at $\pi$ is zero):
		\begin{equation*}
		0=Cl_2(\pi) =
	    \int_0^\pi \log{|2\sin{(\beta/2)}|}d\beta=
	    \int_0^\pi \log{|\sin{(\beta/2)}|}d\beta+\pi\log{2},
	\end{equation*}
	We thus conclude 
		\begin{equation*}
	    I_1(2)=\frac{\pi}2+2\int_0^\pi \log{|\sin{(\beta/2)}|}d\beta=\frac{\pi}2-2\pi\log{2}=\pi\Big(\frac12-\log{4}\Big).
	\end{equation*}
This completes the proof.
\end{proof}

We notice that in Proposition \ref{linearfourier} the first frequency mode is neutral at the linear level. However, the restriction \eqref{constraint} is an equation that relates this frequency with all the higher ones. Thus, the rough idea is to proceed as follows: apply an implicit function theorem to \eqref{constraint} to solve $\widehat{\theta}(-1)$ and $\widehat{\theta}(1)$ in terms of $\widehat{\theta}(k)$ for $|k|\geq2$; use \eqref{system} to compute $\widehat{\theta}(k)$ for $|k|\geq2$, for which the linear operator provides dissipation when we are able to control the nonlinear terms.  Then we will use \eqref{Lequation} to control $L(t)$ in terms of $\theta(t)$.  Finally we can compute the evolution of the zero frequency $\widehat{\vartheta}(0)$ from \eqref{system}.

\chapter{Notation and main results}\label{MainResults}

We introduce the notation that will be used in the rest of the paper in Section \ref{sec:Notation}.   We will then state the main results in Section \ref{sec:globalTHM}.

\section{Notation}\label{sec:Notation}
We recall the complex vector notation introduced in Section \ref{sec:ComplexVector} and the Fourier transform notation introduced in Section \ref{sec:FourierTransCalc}.  We use $\mathbb{T}=[-\pi,\pi]$ as our domain with periodic boundary conditions. 

We introduce the space $\mathcal{F}^{0,1}$ to denote the Wiener algebra, i.e., the space of absolutely convergent Fourier series. The norm in this space is
\begin{equation}\notag \label{fzerone.def}
\|f\|_{\fzerone}\eqdef\sum_{k\in\mathbb{Z}} |\hat{f}(k)|.
\end{equation}
We analogously introduce the homogeneous spaces $\dot{\mathcal{F}}^{s,1}$ with norm
\begin{equation}\notag \label{fsone.def}
\|f\|_{\fsone}\eqdef\sum_{k\ne 0} |k|^s|\hat{f}(k)|, \quad s\ge 0.
\end{equation}
Further in Chapter \ref{sec:uniqueness} we will use the notation
\begin{equation}\label{max.fcns}
    \|f_{1}, f_{2}, \ldots , f_{k}\|_{\fsone} \eqdef 
    \sum_{j=1}^{k}\|f_{j}\|_{\fsone}.
\end{equation}
Moreover, we will use the spaces of analytic functions $\dot{\mathcal{F}}^{s,1}_\nu$ where these norms include exponential weights as follows:
\begin{equation}\label{fzeroonenunorm}
\|f\|_{\fzeronenu}\eqdef\sum_{k\in\mathbb{Z}}e^{\nu(t)|k|}|\hat{f}(k)|,
\end{equation}
and
\begin{equation}\label{fsonenorm}
\|f\|_{\fsonenu}\eqdef\sum_{k\ne 0}e^{\nu(t)|k|}|k|^s|\hat{f}(k)|, \quad s\ge 0,
\end{equation}
with  
\begin{equation}\label{nu}
    \nu(t)\eqdef\nu_0 \frac{t}{1+t}>0,
\end{equation}
where $0<\nu'(t)\leq \nu_0$ is bounded and small enough for all time when $\nu_0>0$ is chosen small enough. 


We recall the embeddings $\dot{\mathcal{F}}^{s_2,1}_\nu\hookrightarrow \dot{\mathcal{F}}^{s_1,1}_\nu$ for $0<s_1\leq s_2$, with norm inequality
\begin{equation}\label{embed}
\|f\|_{\dot{\mathcal{F}}^{s_1,1}_\nu}\leq \|f\|_{\dot{\mathcal{F}}^{s_2,1}_\nu}, \qquad  0<s_1\leq s_2.
\end{equation}
This inequality also holds if $s_1=0$ with $\fzerone$ in the lower bound provided that $\hat{f}(0)=0$.
We also recall the interpolation inequality 
\begin{equation}\label{interpolation}
\|f\|_{\dot{\mathcal{F}}^{s,1}_\nu}\leq \|f\|_{\dot{\mathcal{F}}^{s_1,1}_{\nu}}^{1-\sigma}\|f\|_{\dot{\mathcal{F}}^{s_2,1}_{\nu}}^\sigma,\qquad s=(1-\sigma)s_1+\sigma s_2,
\end{equation}
for $0 \le \sigma \le 1$, $s, s_1, s_2 \ge 0$.

We further will use the discrete delta function $\delta(k)$ for $k\in \mathbb{Z}$ which is defined  as $\delta(0)=1$ and $\delta(k)=0$ for $k\ne 0$.  We also define $1_A$ to be the standard indicator function of the set $A$ so that $1_A(x)=1$ if $x\in A$ and $1_A(x)=0$ if $x\notin A$.  

We define the $\ell^p=\ell^p(\mathbb{Z})$ norms for $1\le p < \infty$ of a sequence $a=\{a_k\}_{k\in \mathbb{Z}}$ as
\begin{equation}\notag
    \| a\|_{\ell^p} \eqdef \left( \sum_{k\in\mathbb{Z}} |a_k|^p \right)^{1/p},
\end{equation}
and for $p=\infty$ we use
\begin{equation}\notag
    \| a\|_{\ell^\infty} \eqdef \sup_{k\in\mathbb{Z}} |a_k|.
\end{equation}
We also define the following notation for $s\ge 0$:
\begin{equation}\label{bfcn.def}
    \bfcn(n,s) 
    \eqdef \left\{
    \begin{aligned} 
    1,\quad 0 \le s \le 1,\\
        n^{s-1}, \quad s >1,
    \end{aligned}\right.
\end{equation}
Further define the high frequency cut-off operator $\mathcal{J}_N$ for $N\ge 0$ by
\begin{equation}\label{CutOffHigh}
    \widehat{\mathcal{J}_N f}(k) \eqdef 1_{|k|\leq N}\widehat{f}(k).
\end{equation}
We will use these norms and notations in the rest of the paper.


Throughout the paper, we will denote 
\begin{equation*}
    C_i=C_i(x)=C_i\left(x; A_\mu,\frac{|A_\rho|R^2}{A_\sigma}\right)>0
\end{equation*}
as functions that are increasing in $x\geq0$ and might depend on the physical parameters such as $A_\mu,\frac{|A_\rho|R^2}{A_\sigma}$,  with the property that $C_i(x)\approx 1$  for all $-1\leq A_\mu\leq1$, $\frac{|A_\rho|R^2}{A_\sigma}\geq0$. Typically, $x$ will be the norm $\|\theta\|_{\fsonenu}$ with $s=0$ or $s=1/2$.

We denote $f*g$ as the standard convolution of $f$ and $g$.
We use the iterated convolution notation
\begin{equation}\label{iteratedConvlution}
*^kf = \underbrace{f *  \cdots *f}_\text{$k-1$ convolutions of $k$ copies of $f$}
\end{equation}
Thus for instance $*^2f = f*f$.  We then sometimes also use the notation $g*^kf = g * \underbrace{f *  \cdots *f}_\text{$k-1$ convolutions}$ to avoid an additional $*$ in the notation.

\section{Main Results}\label{sec:globalTHM} 

The main result of this paper states that, for any value of the physical parameters $A_\rho$, $A_\mu$, $A_\sigma$, a bubble in a porous medium, with arbitrary volume $\pi R^2$ and shape that is not too far from a circle, converges exponentially fast to a circle that rises (or falls) with constant velocity proportional to the difference between the inner and outer density. The initial interface needs to have barely more than a  continuous tangent vector, allowing in particular for unbounded curvature.  In particular since we suppose $\theta_0\in \dot{\mathcal{F}}^{\frac12,1}$ then the initial interface has regularity $W^{\frac32,\infty}$,  in terms of the tangent vector the initial regularity is $W^{\frac12,\infty}$.  In particular the initial curvature doesn't need to be bounded.
Moreover, the interface becomes instantaneously analytic.

We summarize here the system that models our problem. For clarity, we write the zero frequency of $\vartheta$ apart because it is decoupled from the rest, and the equation for $\theta$ with the linear and nonlinear terms separated:
\begin{equation}\label{finalsystem}
\left\{\begin{aligned}
 \widehat{\vartheta}_t(0)&=\frac{2\pi}{L(t)}\widehat{T}*\widehat{(1+\theta_\alpha)}(0),\\
     \theta_t(\alpha)&=\frac{2\pi}{L(t)}\Big(\mathcal{L}(\alpha)+N(\alpha)\Big),\\
    L(t)&=2\pi R\Big(1+\frac{1}{2\pi}\imag \int_{-\pi}^\pi\int_0^\alpha e^{i(\alpha-\eta)} \sum_{n\geq1}\frac{i^n}{n!}(\theta(\alpha)-\theta(\eta))^n d\eta d\alpha\Big)^{-\frac12},\\
    0&=\int_{-\pi}^{\pi} e^{i(\alpha+\theta(\alpha))}d\alpha,
    \end{aligned}\right.
\end{equation}
where the linear and nonlinear terms $\mathcal{L}(\alpha)$, $N(\alpha)$ are given by \eqref{system}
with $T(\alpha)$ defined in \eqref{T}, \eqref{Tsplit}, and \eqref{T0}, $U(\alpha)$ in \eqref{U} and \eqref{Usplit}.

\begin{thm}\label{thm:global}
	Fix $A_\mu\in[-1,1]$, $A_\rho\in\mathbb{R}$, $A_\sigma>0$, and $R>0$.
	Assume that the initial data  $\vartheta_0(\alpha)=\widehat{\vartheta}_0(0)+\theta_0(\alpha)\in \dot{\mathcal{F}}^{\frac12,1}$ satisfies the medium-size condition 
    \begin{equation}\label{condition}
     \|\theta_0\|_{\dot{\mathcal{F}}^{\frac12,1}}<K\Big(\frac{|A_\rho|R^2}{A_\sigma}, A_\mu\Big),
    \end{equation}
    where $K$ is defined in \eqref{Kdef}.
	Then, for any $T>0$, there exists a unique global strong solution   $\vartheta(\alpha, t)=\widehat{\vartheta}(0,t)+\theta(\alpha,t)$ to the system  \eqref{finalsystem}, which lies in the space
	\begin{equation*}	    \vartheta\in C([0,T];\dot{\mathcal{F}}^{\frac12,1}_\nu)\cap L^1([0,T];\dot{\mathcal{F}}^{\frac72,1}_\nu),\quad 0<T<\infty,
	\end{equation*}
	with $\nu(t)$ given by \eqref{nu}.
	In particular the solution becomes instantaneously analytic. Moreover, the following energy inequality is satisfied for $0\leq t\leq T$:
	\begin{equation}\label{estimatef12}
    	\thetahonenu(t)+ \frac{A_\sigma}{R^3}  \dissconstant \int_0^t \thetasevenhonenu(\tau) d\tau \leq C_S^2\|\theta_0\|_{\fhone},
	\end{equation}
	with $$\dissconstant=\dissconstant\Big(C_S\|\theta_0\|_{\fhonenu},\frac{|A_\rho|R^2}{A_\sigma},
       A_\mu,\frac{A_\sigma}{R^3},\nu\Big)>0$$
       defined in \eqref{dissipation}, and $C_S=C_S\big(A_\mu,\frac{|A_\rho| R^2}{A_\sigma}\big)$ defined in \eqref{CSbounds}. 
       In addition,  we have the uniform in time estimate
	\begin{equation}\label{decay}
	    \begin{aligned}
	        \thetahonenu(t)\leq C_S^2 \|\theta_0\|_{\fhone}e^{-\frac{A_\sigma}{R^3}\dissconstant t}.
	    \end{aligned}
	\end{equation}
	Furthermore, the zero frequency $\widehat{\vartheta}(0)$ remains bounded for all times
	\begin{equation}\label{zeroBoundUniform}
	     |\widehat{\vartheta}(0,t)|\leq |\widehat{\vartheta}_0(0)|+C_{42}\|\theta_0\|_{\fhone},
	\end{equation}
	where $C_{42}$ is defined in \eqref{C42}.
\end{thm}

We remark that none of the uniform constants in Theorem \ref{thm:global} depend upon our choice of $T>0$, and $T$ can be taken arbitrarily large.  We also remark that for $\nu(t)$ given by \eqref{nu} then in  $\nu_0>0$ is chosen sufficiently small as in \eqref{dissipationcond}.

\begin{remark}
From Proposition \ref{circles}, the large time decay in \eqref{decay} implies the exponential convergence to rising or falling circles. Moreover, as part of the proof, it will be proven in \eqref{Lboundaux} that the length satisfies for all times $t\geq0$ that
\begin{equation}\label{Lfinalbound}
     \begin{aligned}
        \frac{R}{\sqrt{1+\frac{\pi}{2} \big(e^{2\|\theta\|_{\fhone}(t)}-1\big)}}\leq\frac{L(t)}{2\pi}\leq \frac{R}{\sqrt{1-\frac{\pi}{2}\big(e^{2\|\theta\|_{\fhone}(t)}-1\big)}},
    \end{aligned}
    \end{equation}
    which also shows that $L(t)\to 2\pi R$ as $t\to \infty$.
\end{remark}

\begin{remark}
The size condition \eqref{condition} of the theorem above is explicit: for any value of the physical parameters, it gives a bound for the norm of the initial data that can be computed. We also notice that, thanks to the diagonalization performed in Chapter \ref{secanalytic}, the dissipative character of the equation is shown in \eqref{estimatef12} for any $A_\sigma>0$, no matter how large the gravity effects are.
\end{remark}

\begin{remark}
We further point out that all of the estimates in Chapter \ref{IFTSection}, Chapter \ref{sec:FourierRandS}, Chapter \ref{secw},  Chapter \ref{secanalytic}, and Chapter \ref{sec:uniqueness} carefully track the constants in each estimate.    If one only wanted to replace \eqref{condition} and \eqref{Kdef} with a non-explicit smallness condition then the proof presented in this paper could be substantially shortened. 
\end{remark}

\begin{remark}
The gain of analyticity property for any $t>0$ of the solution from our main Theorem \ref{thm:global} in $\fhonenu$ makes equivalent any of the formulations used in this paper, including the equivalence between \eqref{eqcurve.theta} and \eqref{finalsystem}.
\end{remark}

\chapter{Implicit function theorem}\label{IFTSection}

In this chapter, we will prove an explicit uniform upper bound for the $\pm 1$ frequencies of $\theta$ in terms of the higher Fourier frequencies of $\theta$.  The main result of this chapter is the implicit function theorem in Proposition \ref{IFTprop}.

For ease of notation, only in this chapter we will use the following space.  For $s \ge 0$, we define the normed space
\begin{equation}\label{tildefsone}
    \tildefs \eqdef \left\{ u: \mathbb{T} \rightarrow \mathbb{T} \ | \ \hat{u}(0) = \hat{u}(\pm 1) = 0 \text{ and } \|u\|_{\tildefs}  < \infty \right\}.
\end{equation}
Here we use the norm
\begin{equation*}
    \|u\|_{\tildefs} \eqdef \sum_{|k|\geq 2} |k|^{s}|\hat{u}(k)|.
\end{equation*}
In view of \eqref{tildefsone}, recalling \eqref{CutOffHigh}, we consider 
$
\tilde{\theta} \eqdef \left(I-\mathcal{J}_2 \right)\theta.
$
We then remark that 
$$
\tilde{\theta}(\alpha) = 
\left(I-\mathcal{J}_2 \right)\theta(\alpha)=\sum_{|k|\ge 2} \hat{\theta}(k) e^{ik\alpha}.
$$
An implicit relation between the frequencies $\hat{\theta}(\pm 1)$ and the function $\tilde{\theta}$ can be derived from \eqref{constraint} and is given by
\begin{equation}\label{implicitrelation}
	\int_{-\pi}^{\pi} e^{i\alpha + i\hat{\theta}(-1)e^{-i\alpha} + i\hat{\theta}(1)e^{i\alpha} + i\tilde{\theta}(\alpha)} d\alpha = 0.
\end{equation}
Note that $\hat{\theta}(-1) = \overline{\hat{\theta}(1)}=\real \hat{\theta}(1)- i \imag \hat{\theta}(1)$ since $\theta$ is real.  Further
\begin{equation}\notag
\hat{\theta}(-1)e^{-i\alpha} + \hat{\theta}(1)e^{i\alpha}
=
2\real \hat{\theta}(1) \cos(\alpha) - \imag \hat{\theta}(1)\sin(\alpha).
\end{equation}
We will also use the following notation
\begin{equation}\label{psi}
    \psi(\alpha, x, u) = \alpha + 2(x_{1}\cos(\alpha) - x_{2}\sin(\alpha)) + u(\alpha),
    \quad x\in\mathbb{R}^{2}.
\end{equation}
Then we express the integral in \eqref{implicitrelation} as a vector as follows
\begin{equation}\label{g}
	g(u, x) = 
	\begin{bmatrix}
		\int_{-\pi}^{\pi} \cos(\psi(\alpha, x, u))d\alpha
		\\
		\int_{-\pi}^{\pi} \sin(\psi(\alpha, x, u))d\alpha
	\end{bmatrix}
	=	\begin{bmatrix}
		g_1(u, x)
		\\
		g_2(u, x)
	\end{bmatrix}.
\end{equation}
Here $g: \tildefzero \times \mathbb{R}^{2} \rightarrow \mathbb{R}^{2}$.
Now we can rewrite the relation \eqref{implicitrelation} as
$$
g(\tilde{\theta}, (\real\hat{\theta}(1), \imag \hat{\theta}(1))) = 0.
$$
The main result in this chapter is the following proposition.

\begin{prop}\label{IFTprop}
Fix $0<r<\frac{1}{2}\log(\frac{5}{4})$.
Consider $\|\theta\|_{\fzerone} < r$.  Let
$$
B_{\tildefzero}(0,r) \eqdef \{ u \in \tildefzero: \|u\|_{\tildefzero} < r\}.
$$
Then there exists a unique function $F:B_{\tildefzero}(0,r)\to \R^2$ such that 
$F(\tilde{\theta})=(\real\hat{\theta}(1), \imag \hat{\theta}(1))$ and \eqref{implicitrelation} is satisfied.
We further have
\begin{equation}\label{frequencyrelation}
			|\hat{\theta}(+1)|+|\hat{\theta}(-1)|  \leq   C_{I}(r)r \sum_{|k|\ge 2} |\hat{\theta}(k)|.
\end{equation}
Here the constant $C_I(r)>0$ is given by 
\begin{equation}\label{c1.implicit}
C_I(r) \eqdef \frac{1}{r}\frac{2\exp(r)(\exp(r)-1)}{1- 4(\exp(2r)-1)}.
\end{equation}
We note that $C_I(r)$ is an increasing function of $r$ and that $C_I(r)\to 2$ as $r\to 0$ and $C_I(r)\to \infty$ as $r\to \frac{1}{2}\log(\frac{5}{4})$. 
\end{prop}

We point out that the explicit constant $K>0$ in our main Theorem \ref{thm:global} in \eqref{condition} and \eqref{Kdef} is smaller than $\frac{1}{2}\log(\frac{5}{4})$.  Therefore the smallness condition in Proposition \ref{IFTprop} is consistent with Theorem \ref{thm:global}.

This proposition is shown by an implicit function theorem argument on the function described in \eqref{g} around the value $g(0,0) = 0$. The remainder of this chapter is dedicated to proving Proposition \ref{IFTprop}.

\section{Calculation of the Fr{\'e}chet derivatives}
First, we compute the Fr{\'e}chet derivatives with respect to $u \in \tildefzero$ and $x\in \mathbb{R}^{2}$.  Below we use the notation $D_u$ to denote the one component Fr{\'e}chet derivative of $g(u,x)$ so that $D_u g(u,x)$ is a two dimensional vector, and we denote $D_x$ to denote the two component Fr{\'e}chet derivative of $g(u,x)$ so that $D_xg(u,x)$ is a $2\times 2$ matrix.  We will also use the notation $D=(D_u, D_x)$ to denote the derivative in the variables $(u,x)$ and then $Dg(u,x)$ can be represented as a $3\times 3$ matrix.

\begin{lemma}\label{Frechet:Lemma}
	The Fr\'echet derivatives of $g$, recalling \eqref{psi}, are given by
	\begin{equation}\label{Frechetu}
		D_{u}g(u,x)h 
		=
		\int_{-\pi}^{\pi}d\alpha ~ h(\alpha)	\begin{bmatrix}
		 -\sin(\psi(\alpha, x, u))
		\\
		\cos(\psi(\alpha, x, u))
	\end{bmatrix},
	\end{equation}
	for $h\in \tildefzero$ and
	\begin{equation}\label{Frechetx}
		D_{x}g(u,x)y 
				=
		\begin{bmatrix}
		 -2\int_{-\pi}^{\pi}d\alpha \sin(\psi) \cos\alpha & 2\int_{-\pi}^{\pi}d\alpha \sin(\psi) \sin\alpha
		\\
		 2\int_{-\pi}^{\pi}d\alpha \cos(\psi) \cos\alpha 
		 & -2\int_{-\pi}^{\pi}d\alpha \cos(\psi) \sin\alpha
	\end{bmatrix}
	\begin{bmatrix}
		 y_{1}
		\\
		y_{2}
	\end{bmatrix},
	\end{equation}
	for $y = \begin{bmatrix}
		 y_{1}
		\\
		y_{2}
	\end{bmatrix} \in \mathbb{R}^{2}$.
\end{lemma}

For simplicity, we write \eqref{g} in complex notation as
\begin{equation}\notag
	g(u, x) = \int_{-\pi}^{\pi} e^{i\alpha + 2i(x_{1}\cos(\alpha) -x_{2}\sin(\alpha)) + iu(\alpha)}d\alpha.
\end{equation}
In particular then in complex notation \eqref{Frechetu} takes the form
	\begin{equation}\notag
		D_{u}g(u,x)h = \int_{-\pi}^{\pi} i h(\alpha) e^{i\psi(\alpha, x, u)} d\alpha.
	\end{equation}
We will prove Lemma \ref{Frechet:Lemma} using this expression for $g(u,x)$.

\begin{proof}[Proof of Lemma \ref{Frechet:Lemma}]
	Computing the derivative with respect to $u$, we have
	\begin{multline*}
		|g(u+h, x) - g(u,x) - D_{u}g(u,x)h| \\= \Big|\int_{-\pi}^{\pi}  e^{i\psi(\alpha, x, u)}(e^{i h(\alpha)} - 1 - ih(\alpha)) d\alpha \Big|
		\\
		\leq 2\pi \|e^{ih(\alpha)} - 1 - ih(\alpha)\|_{L^{\infty}} \leq 2\pi \sum_{n=2}^{\infty} \frac{\|h\|_{L^{\infty}}^{n}}{n!} 
		\\
		\leq   \Big(2\pi \sum_{n=1}^{\infty} \frac{\|h\|_{\tildefzero}^{n}}{(n+1)!} \Big) \|h\|_{\tildefzero}
				\leq   \Big(2\pi e^{\|h\|_{\tildefzero}} \Big) \|h\|_{\tildefzero},
	\end{multline*}
	since for $h\in\tildefzero$ we have that $\hat{h}(0) = \hat{h}(\pm 1) = 0$ and so
	$$ \|h\|_{L^{\infty}} \leq \sum_{k\in\mathbb{Z}} |\hat{h}(k)| = \sum_{|k| \geq 2}  |\hat{h}(k)|.$$
	Thus, as $h \rightarrow 0$, we obtain validation of \eqref{Frechetu}. Then \eqref{Frechetx} is proven in a similar way.  \end{proof}

\section{Proof of the Implicit function theorem}
We will now prove Proposition \ref{IFTprop}.

\begin{proof}[Proof of Proposition \ref{IFTprop}]
First notice that from \eqref{Frechetx} we have 
\begin{equation}\label{dxg00}
    \begin{aligned}
	D_{x} g(0,0)y
	& = 		\begin{bmatrix}
		 -2\int_{-\pi}^{\pi}d\alpha \sin(\alpha) \cos\alpha & 2\int_{-\pi}^{\pi}d\alpha \sin(\alpha) \sin\alpha
		\\
		 2\int_{-\pi}^{\pi}d\alpha \cos(\alpha) \cos\alpha 
		 & -2\int_{-\pi}^{\pi}d\alpha \cos(\alpha) \sin\alpha
	\end{bmatrix}
	\begin{bmatrix}
		y_{1}\\
		y_{2}
	\end{bmatrix} 
	\\
	&= 2\pi \begin{bmatrix}
		0 &  1\\
		1 & 0
	\end{bmatrix} \begin{bmatrix}
		y_{1}\\
		y_{2}
	\end{bmatrix}
	=2\pi\begin{bmatrix}
		y_{2}\\
		y_{1}
	\end{bmatrix},
	\end{aligned}
\end{equation}
Therefore $D_{x} g(0,0)^{-1} = \frac{1}{2\pi}\begin{bmatrix}
		0 &  1\\
		1 & 0
	\end{bmatrix}$.  
For simplicity, we normalize the function $g$ around $(0,0)$ by defining
\begin{equation}\label{tildeg}
	\tilde{g}(u,x) = \begin{bmatrix}
		\tilde{g}_{1}(u,x)\\
		\tilde{g}_{2}(u,x)
	\end{bmatrix} 
	=  D_{x}g(0,0)^{-1}\begin{bmatrix}
		g_{1}(u,x)\\
		g_{2}(u,x)
	\end{bmatrix}
	=\frac{1}{2\pi}\begin{bmatrix}
		g_{2}(u,x)\\
		g_{1}(u,x)
	\end{bmatrix},
\end{equation}
and thus, $D_{x}\tilde{g}(0,0) = \mathbb{I}_{\mathbb{R}^{2}}$ which is the identity matrix on $\mathbb{R}^{2}$. Next, define the function $\phi: \tildefzero \times \mathbb{R}^{2} \rightarrow \tildefzero \times \mathbb{R}^{2}$ by
\begin{equation}\label{phi.def}
\phi(u,x) = [D\tilde{\phi}(0,0)]^{-1}\tilde{\phi}(u,x), \quad\text{  where  } \quad \tilde{\phi}(u,x) 
=
\begin{bmatrix}
u   \\
	\tilde{g}_{1}(u,x)\\
		\tilde{g}_{2}(u,x)
\end{bmatrix}.
\end{equation}
Then also $D\tilde{\phi}(u,x)$ is a $3 \times 3$ matrix.  We will obtain the implicit function $F$ given in Proposition \ref{IFTprop} by inverting $\phi$ in a neighborhood of the point $(0, 0) \in \tildefzero \times \mathbb{R}^{2}$. We will also calculate $[D\tilde{\phi}(0,0)]^{-1}$,  notice that 
\begin{equation}\notag
D\tilde{\phi}(0,0) 
=
    \begin{bmatrix}
		\mathbb{I}_{\tildefzero} & 0\\
		D_{u}\tilde{g}(0,0) & D_{x}\tilde{g}(0,0)
	\end{bmatrix}
	=
		\begin{bmatrix}
		\mathbb{I}_{\tildefzero} &  0\\
		  0  & \mathbb{I}_{\mathbb{R}^{2}}
	\end{bmatrix}.
\end{equation}
Here $\mathbb{I}_{\tildefzero}$ is the identity map on $\tildefzero$.
The last equality holds since
\begin{equation*}
\begin{aligned}
D_{u}\tilde{g}(0,0)h 
&= D_{x}g(0,0)^{-1}D_{u}g(0,0)h 
\\
&=\frac{1}{2\pi}
		\int_{-\pi}^{\pi}d\alpha ~ h(\alpha)	\begin{bmatrix}
		0 &  1\\
		1 & 0
	\end{bmatrix}	\begin{bmatrix}
		 -\sin(\psi(\alpha, x, u))
		\\
		\cos(\psi(\alpha, x, u))
	\end{bmatrix}=  0,
\end{aligned}
\end{equation*}
since $h\in\tildefzero$ so that $\hat{h}(\pm 1) =0$.  Therefore $[D\tilde{\phi}(0,0)]^{-1} = 		\begin{bmatrix}
		\mathbb{I}_{\tildefzero} &  0\\
		  0  & \mathbb{I}_{\mathbb{R}^{2}}
	\end{bmatrix}$.

For two norms $\| \cdot \|$ and $| \cdot |$, we will use the notation that 
\begin{equation}\label{littleo}
   f=o(|h|), \quad \text{if} \quad 
   \|f\|\to 0 \quad \text{and} \quad \frac{\|f\|}{|h|}\to 0 \quad \text{as}\quad  |h|\to 0.
\end{equation}
Now, to invert $\phi$, we first define the function 
\begin{equation}\label{tau.def}
    \tau(u,x) \eqdef (u,x) - \phi(u,x).
\end{equation}
We will show that $\tau(u,x)$ is a contraction map by computing $D\tau$. We will calculate that
$$
D\tau(u,x) = \mathbb{I} - [D\tilde{\phi}(0,0)]^{-1} D\tilde{\phi}(u,x),
$$
where $\mathbb{I}$ is the identity map on $\tildefzero \times \mathbb{R}^{2}$. To this end, we compute 
\begin{multline}\notag
	\tilde{\phi}(u+h,x+y) - \tilde{\phi}(u,x)
	\\ = \tilde{\phi}(u+h,x+y) - \tilde{\phi}(u,x+y) + \tilde{\phi}(u,x+y) - \tilde{\phi}(u,x)
	\\
	= \begin{bmatrix}
		\mathbb{I}_{\tildefzero} & 0\\
		D_{u}\tilde{g}(u,x) & D_{x}\tilde{g}(u,x)
	\end{bmatrix}
	\begin{bmatrix}
		h\\
		y
	\end{bmatrix} + o(\|h\|_{\tildefzero} + |y|).
\end{multline}
Then using, 
\eqref{littleo},  the above holds as $\|h\|_{\tildefzero} + |y| \to 0$.
 Hence,
\begin{equation*}
\begin{aligned}
	D\tau(u,x) &=  \mathbb{I} - 
	\begin{bmatrix}
		\mathbb{I}_{\tildefzero} &  0\\
		  0  & \mathbb{I}_{\mathbb{R}^{2}}
	\end{bmatrix} \!\!
	\begin{bmatrix}
		\mathbb{I}_{\tildefzero} & 0\\
		D_{u}\tilde{g}(u,x) & D_{x}\tilde{g}(u,x)
	\end{bmatrix} \\
	&=
	\begin{bmatrix}
		0  & 0\\
		-D_{u}\tilde{g}(u,x) & \mathbb{I}_{\mathbb{R}^{2}}- D_{x}\tilde{g}(u,x)
	\end{bmatrix}.
	\end{aligned}
\end{equation*}
To obtain $\tau$ as a contraction on some ball $B_{R}(0) \subset \tildefzero \times \mathbb{R}^{2}$, it is sufficient to show that the following holds
\begin{equation}\label{contraction}
	\Big\|D\tau(u,x) \begin{bmatrix}
		h\\
		y
	\end{bmatrix}\Big\|_{\tildefzero\times\mathbb{R}^{2}} < r \|(h,y)\|_{\tildefzero\times\mathbb{R}^{2}},
\end{equation}
for some $0< r< 1$ and
for any $(h,y)\in B_{R}(0) \subset \tildefzero \times \mathbb{R}^{2}$  where 
\begin{equation}\notag
    B_{R}(0) \eqdef \{(u,x) \ | \ |x|+\|u\|_{\tildefzero} < R \}, \quad R>0.
\end{equation}
Here we also denote $\|(h,y)\|_{\tildefzero\times\mathbb{R}^{2}} = \|h\|_{\tildefzero}+ |y|$.

Since we have \eqref{dxg00} then, using \eqref{Frechetu} and \eqref{Frechetx}, condition \eqref{contraction} becomes the condition that the inequality
\begin{multline}\label{contractioncondition}
\Big\|D\tau(u,x) \begin{bmatrix}
		h\\
		y
	\end{bmatrix}\!\Big\|_{\tildefzero\times\mathbb{R}^{2}}\!\!
	=
	\\
	\Bigg|\frac{1}{2\pi} 		\int_{-\pi}^{\pi}d\alpha ~ h(\alpha)	\begin{bmatrix}
		 -\cos(\psi(\alpha, x, u))
		\\
		\sin(\psi(\alpha, x, u))
	\end{bmatrix}
	\\
	+\! \frac{1}{2\pi} 
			\begin{bmatrix}
					 2\int_{-\pi}^{\pi}d\alpha \cos(\psi) \left(y_1 \cos\alpha -y_2\sin\alpha\right) - 2\pi y_1
		 \\
		 -2\int_{-\pi}^{\pi}d\alpha \sin(\psi) \left(y_1 \cos\alpha -y_2\sin\alpha\right)- 2\pi y_2
	\end{bmatrix}
	\Bigg|
	\\
	< r\|h\|_{\tildefzero} + r|y|,
\end{multline}
holds for a fixed $0< r< 1$, for all $(u,x)\in B_{R}(0)$, and for any $(h,y)\in \tildefzero\times \mathbb{R}^{2}$.  We also use the definition \eqref{psi} in the integral above.

We further {\it claim} that condition \eqref{contractioncondition} is satisfied for any ball 
$B_{R}(0)$ such that $2|x|+\|u\|_{\tildefzero} < \log(2)$ holds for all $(u,x)\in B_{R}(0)$.

 {\it Proof of the claim:}	Let
  	\begin{equation}\label{A1}
 	A_{1} = 
 	\begin{bmatrix}
		A_{11}
		\\
		A_{12}
	\end{bmatrix}
 	=
 	 \frac{1}{2\pi} 		\int_{-\pi}^{\pi}d\alpha ~ h(\alpha)	\begin{bmatrix}
		 -\cos(\psi(\alpha, x, u))
		\\
		\sin(\psi(\alpha, x, u))
	\end{bmatrix}.\end{equation}
We will use the complex exponential representation of $\cos(\psi(\alpha, x, u))$ and $\sin(\psi(\alpha, x, u))$ with \eqref{psi}.  We will further Taylor expand $\exp\left(i 2(x_{1}\cos(\alpha)-x_{2}\sin(\alpha))+i u(\alpha)) \right)$ 
	and use that
	$\frac{1}{2\pi}\int_{-\pi}^{\pi} h(\alpha) e^{\pm i\alpha} d\alpha = \hat{h}(\mp 1) = 0$ since $h\in \tildefzero$.   We obtain
		\begin{multline*}
		|A_{11}| 
		\le \frac{1}{4\pi}\Big| \int_{-\pi}^{\pi} h(\alpha) e^{i\alpha}\sum_{n=1}^{\infty} \frac{i^{n}(2(x_{1}\cos(\alpha)-x_{2}\sin(\alpha))+u(\alpha))^{n}}{n!} d\alpha \Big|\\
+ \frac{1}{4\pi}\Big| \int_{-\pi}^{\pi} h(\alpha) e^{-i\alpha}\sum_{n=1}^{\infty} \frac{i^{n}(2(x_{1}\cos(\alpha)-x_{2}\sin(\alpha))+u(\alpha))^{n}}{n!} d\alpha \Big|.
	\end{multline*}
	For simplicity we estimate the term with $e^{i\alpha}$ below:
\begin{multline}\notag
\frac{1}{2\pi}\Big| \int_{-\pi}^{\pi} h(\alpha) e^{i\alpha}\sum_{n=1}^{\infty} \frac{i^{n}(2(x_{1}\cos(\alpha)-x_{2}\sin(\alpha))+u(\alpha))^{n}}{n!} d\alpha \Big|
\\
		= \frac{1}{2\pi}\Big| \int_{-\pi}^{\pi} h(\alpha) e^{i\alpha}\sum_{n=1}^{\infty}\sum_{k=0}^{n} {n \choose k} \frac{2^{k}(x_{1}\cos(\alpha)-x_{2}\sin(\alpha))^{k}u(\alpha)^{n-k}}{n!} d\alpha   \Big|
\\
		\leq \sum_{n=1}^{\infty}\sum_{k=0}^{n} {n \choose k}\frac{1}{n!}\Big|\Big(\hat{h} \ast^{k} \mathcal{F}\{2(x_{1}\cos(\alpha)-x_{2}\sin(\alpha))\}\ast^{n-k}\hat{u}\Big)(-1)\Big|
\\
\leq \sum_{n=1}^{\infty}\sum_{k=0}^{n} {n \choose k}\frac{1}{n!} \|\hat{h}\|_{\ell^{1}}\|\hat{u}\|_{\ell^{1}}^{n-k}2^{k-1}|x|^{k}.
\end{multline}
The last line is obtained using Young's inequality for convolutions and 
	$$\widehat{\cos}(k) = \frac{1}{2}(\delta(1-k) + \delta(-1-k)), \quad   \widehat{\sin}(k) = \frac{-i}{2}(\delta(1-k) - \delta(-1-k)),$$
 and thus the $\ell^{1}$ norm of the Fourier transform is
 \begin{equation}\label{ell1ft.def}
    \|(2(x_{1}\cos(\alpha)-x_{2}\sin(\alpha)))^{\wedge}\|_{\ell^{1}} = 2|x|,
\end{equation}
	and the $\ell^{\infty}$ norm of the Fourier transform is
	$$\|(2(x_{1}\cos(\alpha)-x_{2}\sin(\alpha)))^{\wedge}\|_{\ell^{\infty}} = |x|.$$
Then, rewriting the series in function form using Taylor's theorem, we get
\begin{equation}\notag
    	|A_{11}| \leq \frac{1}{2}\|h\|_{\tilde{\mathcal{F}}^{0,1}}\Big( e^{2|x| + \|u\|_{\tilde{\mathcal{F}}^{0,1}}} - 1\Big).
\end{equation}
Doing the same for the second term in \eqref{A1} we obtain
\begin{equation}\label{A1estimate}
    	|A_{1}| \leq \|h\|_{\tilde{\mathcal{F}}^{0,1}}\Big( e^{2|x| + \|u\|_{\tilde{\mathcal{F}}^{0,1}}} - 1\Big).
\end{equation}
	Since the latter term in \eqref{contractioncondition} does not involve $h$, for this term we need 
$$
 e^{2|x| + \|u\|_{\tilde{\mathcal{F}}^{0,1}}} - 1 < r < 1.
$$
These estimates give the upper bound in terms of $\|h\|_{\tildefzero}$ in \eqref{contractioncondition}.

	The other term in \eqref{contractioncondition} will give us the upper bound in terms of $|y|$. Now let
	\begin{equation*}
	    A_{2} = \begin{bmatrix} A_{21} \\ A_{22} \end{bmatrix}=  \frac{1}{2\pi} 
			\begin{bmatrix}
					 2\int_{-\pi}^{\pi}d\alpha \cos(\psi) \left(y_1 \cos\alpha -y_2\sin\alpha\right) - 2\pi y_1
		 \\
		 -2\int_{-\pi}^{\pi}d\alpha \sin(\psi) \left(y_1 \cos\alpha -y_2\sin\alpha\right)- 2\pi y_2
	\end{bmatrix}.
	\end{equation*}
	Then using \eqref{psi}, again we use the complex exponential form of $\cos(\psi)$ and by Taylor expansion we have
	\begin{multline}\notag
		A_{21} = 
		\frac{1}{2\pi} \int_{-\pi}^{\pi} d\alpha (y_{1}\cos(\alpha)-y_{2}\sin(\alpha))e^{i\alpha}\sum_{n=1}^{\infty}\frac{i^{n}(\psi(\alpha,x,u)-\alpha)^{n}}{n!} 
		\\
+	\frac{1}{2\pi} \int_{-\pi}^{\pi} d\alpha (y_{1}\cos(\alpha)-y_{2}\sin(\alpha))e^{-i\alpha}\sum_{n=1}^{\infty}\frac{i^{n}(\psi(\alpha,x,u)-\alpha)^{n}}{n!}.
	\end{multline}
We estimate the term with $e^{i\alpha}$ as 
	\begin{multline}\notag
		\frac{1}{2\pi}\Big| \int_{-\pi}^{\pi} d\alpha (y_{1}\cos\alpha-y_{2}\sin\alpha)e^{i\alpha}\sum_{n=1}^{\infty}\frac{i^{n}(\psi(\alpha,x,u)-\alpha)^{n}}{n!} \Big|
		\\
 =\Big| \sum_{n=1}^{\infty}\frac{1}{2\pi}\int_{-\pi}^{\pi} (y_{1}\cos\alpha-y_{2}\sin\alpha)e^{i\alpha}\frac{(2(x_{1}\cos\alpha-x_{2}\sin\alpha)+u(\alpha))^{n}}{n!} d\alpha \Big|
		\\
		\leq \frac{1}{2}\sum_{n=1}^{\infty} \frac{1}{n!}\|2(y_{1}\cos\alpha-y_{2}\sin\alpha)^{\wedge}\|_{\ell^{\infty}}\|(2(x_{1}\cos\alpha-x_{2}\sin\alpha)+u(\alpha))^{\wedge}\|_{\ell^{1}}^{n}
		\\
		\leq \frac{1}{2}\Big( e^{2|x| + \|u\|_{\tilde{\mathcal{F}}^{0,1}}} - 1\Big) |y|.
	\end{multline}
	All the other terms are estimated in the same way.
	Adding all the estimates together we obtain the condition 
\begin{equation}\notag
    |A_2| \le 2\Big( e^{2|x| + \|u\|_{\tilde{\mathcal{F}}^{0,1}}} - 1\Big) |y|.
\end{equation}
This is a bigger coefficient in front of $|y|$ than the coefficient in the bound for $A_{1}$. Thus, the following condition
$$
2\Big( e^{2|x| + \|u\|_{\tilde{\mathcal{F}}^{0,1}}} - 1\Big) < 1,
$$ 
is sufficient. This yields the claim that condition \eqref{contractioncondition} is satisfied on any ball contained in the set of $(x,u)$ such that $$2|x|+\|u\|_{\tildefzero} < \log(3/2).$$  This completes the proof of the {\it claim}.

Now  $\tau$ satisfies the contraction \eqref{contraction} for $(u,x)\in B_{R}(0)$ for $R$ chosen to satisfy \eqref{contractioncondition}.  Recalling \eqref{phi.def} and \eqref{tau.def}, for $v,w\in\tildefzero\times\mathbb{R}^{2}$ and for $0<r<1$ fixed in \eqref{contractioncondition}, we obtain
\begin{align*}
	\|v-w\|_{\tildefzero\times\mathbb{R}^{2}} &\leq \|\phi(v)-\phi(w)\|_{\tildefzero\times\mathbb{R}^{2}} + \|\tau(v)-\tau(w)\|_{\tildefzero\times\mathbb{R}^{2}}\\ &\leq \|\phi(v)-\phi(w)\|_{\tildefzero\times\mathbb{R}^{2}} + r\|v-w\|_{\tildefzero\times\mathbb{R}^{2}}.
\end{align*} 
We conclude that
\begin{equation*}
	\|v-w\|_{\tildefzero\times\mathbb{R}^{2}} 
	\leq (1-r)^{-1}\|\phi(v)-\phi(w)\|_{\tildefzero\times\mathbb{R}^{2}}.
\end{equation*}
Hence, $\phi$ is an injection on $B=B_R(0)$ and so there exists an inverse $\phi^{-1} : \phi(B)\rightarrow B$. Now we define the map $F(u)$ by
$$
F(u) \eqdef \pi_{2}\circ \phi^{-1}(u,(0,0)) \text{ \ \ where \ \ } \pi_{2}(a,(b_{1},b_{2})) = (b_{1},b_{2}).
$$
Further define the map $\pi_{1}$ by $\pi_{1}(a,(b_{1},b_{2})) = a$.

Given $u\in \pi_{1}(B)$, by the definition of $F(u)$ there exists $u'\in \pi_{1}(B)$ such that $(u,(0,0)) = \phi(u',F(u))$.  Then $u=u'$ by \eqref{phi.def}. Further $\phi(u,F(u)) = (u,\tilde{g}(u,F(u)))$ and therefore $\tilde{g}(u,F(u)) = 0$ using \eqref{phi.def} and \eqref{tildeg}.  We have thus shown that $g(u,F(u)) = 0$ from \eqref{g}.  We conclude that $(u,(0,0)) \in \phi(B)$  for all $u\in \pi_{1}(B)$ and so $(u,F(u)) = \phi^{-1}(u,(0,0)) \in B$.  Therefore $g(u,F(u)) = 0$ for any $u\in \pi_{1}(B)$. This concludes the proof of the existence of the function $F$ described in Proposition \ref{IFTprop}.

To obtain the estimate \eqref{frequencyrelation} on $B$, first note that for $x = (x_{1},x_{2})$ we have
\begin{equation*}
\begin{aligned}
|x| &= |\pi_{2}\circ \phi^{-1}(u,(0,0))| = |\pi_{2}\circ \phi^{-1}(u,(0,0)) - \pi_{2}\circ \phi^{-1}(0,(0,0))|\\
&\leq \|u\|_{\tildefzero}\|D_{u}(\pi_{2}\circ\phi^{-1})\|.
\end{aligned}
\end{equation*}
Above note that $g(0,F(0))=0$ and so $F(0)=0$.  
Also the norm on $D_{u}(\pi_{2}\circ\phi^{-1})$ on is the operator norm.   We calculate $D_{u}(\pi_{2}\circ\phi^{-1})$  as follows:
$$\pi_{2}\circ \phi^{-1}(u+h,y)-\pi_{2}\circ \phi^{-1}(u,y) = \pi_{2}(u+h,x_{h}) - \pi_{2}(u,x) = x_{h}-x,$$
where we suppose $(u,y) \in \phi(B)$ and $(u+h,y) \in \phi(B)$ for small $h$.  Therefore
\begin{equation*}
\tilde{g}(u+h,x_{h}) = y \text{ \ \ and \ \ }  \tilde{g}(u,x) = y.    
\end{equation*}
We conclude that $[\tilde{g}(u+h,x_{h})- \tilde{g}(u,x)] = 0$.
Thus
$$0 = \tilde{g}(u+h,x_{h}) - \tilde{g}(u,x_{h}) +\tilde{g}(u,x_{h})- \tilde{g}(u,x).$$
Then, using \eqref{littleo}, we have
$$
0 = D_{u}\tilde{g}(u,x_{h})h +o(\|h\|_{\tildefzero}) + D_{x}\tilde{g}(u,x)(x_{h}-x) +o(|x_{h}-x|).
$$
We can now conclude that
\begin{multline}\notag
    x_{h}-x + D_{x}\tilde{g}(u,x)^{-1}o(|x_{h}-x|) \\ = -D_{x}\tilde{g}(u,x)^{-1}D_{u}\tilde{g}(u,x)h -D_{x}\tilde{g}(u,x)^{-1}o(\|h\|_{\tildefzero}).
\end{multline}
We will show that $D_{x}\tilde{g}(u,x)^{-1}$ is a bounded operator with bound \eqref{detandTbound} for $(x,u)$ satisfying \eqref{xucriteria}. Thus, for $(x,u)$ satisfying \eqref{xucriteria}, we have
$$D_{u}(\pi_{2}\circ \phi^{-1}) = -D_{x}\tilde{g}(u,x)^{-1}D_{u}\tilde{g}(u,x).$$
Then we conclude that
\begin{equation}\label{xinitialbound}
|x| \leq \|D_{x}\tilde{g}(u,x)^{-1}\| \|D_{u}\tilde{g}(u,x)\| \|u\|_{\tildefzero}.
\end{equation}
The term $D_{u}\tilde{g}(u,x)$, using \eqref{Frechetu} with \eqref{dxg00}, is given by
$$D_{u}\tilde{g}(u,x) = D_{x}g(0,0)^{-1}D_{u}g(u,x) = -A_{1},$$
where  $A_{1}$ is given by \eqref{A1} with \eqref{psi}. Hence, using \eqref{A1estimate}, we have 
\begin{equation}\label{Dugbound}
\|D_{u}\tilde{g}(u,x)\| \leq 	\Big(e^{2|x| + \|u\|_{\tilde{\mathcal{F}}^{0,1}}} - 1\Big).
\end{equation}
Consider the operator $D_{x}\tilde{g}(u,x)^{-1}$.  
Using \eqref{Frechetx} and \eqref{dxg00} with \eqref{psi}, we calculate that
\begin{equation*}
    D_{x}\tilde{g}(u,x) =\! \frac{1}{\pi}\! \begin{bmatrix}
    \!\int_{-\pi}^{\pi}\cos(\alpha)\cos(\psi(\alpha)) d\alpha & -\int_{-\pi}^{\pi}\sin(\alpha)\cos(\psi(\alpha)) d\alpha\\\\
    -\int_{-\pi}^{\pi}\cos(\alpha)\sin(\psi(\alpha)) d\alpha &  \int_{-\pi}^{\pi}\sin(\alpha)\sin(\psi(\alpha)) d\alpha
    \end{bmatrix}\!.    
\end{equation*}
Inverting this matrix, we obtain
\begin{equation}\label{aux1}
    \begin{aligned}
    D_{x}\tilde{g}(u,x)^{-1} = \frac{1}{\det(D_{x}\tilde{g}(u,x))} T(\psi),
    \end{aligned}
\end{equation}
where
\begin{equation}\label{T.psi.def}
    T(\psi)=\frac{1}{\pi} \begin{bmatrix}
\int_{-\pi}^{\pi}\sin(\alpha)\sin(\psi(\alpha)) d\alpha\ \ \ \ \int_{-\pi}^{\pi}\sin(\alpha)\cos(\psi(\alpha)) d\alpha
\\
\\
\int_{-\pi}^{\pi}\cos(\alpha)\sin(\psi(\alpha)) d\alpha \ \ \ \  \int_{-\pi}^{\pi}\cos(\alpha)\cos(\psi(\alpha)) d\alpha
\end{bmatrix}.
\end{equation}
We will calculate a lower bound for $d\eqdef\det(D_{x}\tilde{g}(u,x))$:  
\begin{align}
d
&\!=\! \frac1{\pi^{2}}\!\int_{-\pi}^{\pi}\!\!\cos(\alpha)\cos(\psi(\alpha)) d\alpha\! \int_{-\pi}^{\pi}\!\!\sin(\beta)\sin(\psi(\beta)) d\beta \nonumber\\ \label{det}
&\quad- \frac1{\pi^{2}}\! \int_{-\pi}^{\pi}\!\!\!\!\cos(\alpha)\sin(\psi(\alpha)) d\alpha \!\int_{-\pi}^{\pi}\!\!\!\sin(\beta)\cos(\psi(\beta)) d\beta\\
&=\frac1{\pi^{2}}\int_{-\pi}^{\pi}\!\!\int_{-\pi}^{\pi}\!\!\cos(\alpha)\sin(\beta)\sin(\psi(\beta)-\psi(\alpha))d\alpha d\beta.\nonumber
\end{align}
Next, recalling \eqref{psi}, define 
$$
R(\alpha,\beta) \eqdef 2x_{1}(\cos\beta-\cos\alpha)-2x_{2}(\sin\beta-\sin\alpha) + u(\beta)-u(\alpha).
$$
Then we have
\begin{align}
|\sin(\psi(\beta)-\psi(\alpha))& -\sin(\beta - \alpha)|  = |\sin(\beta-\alpha + R(\alpha,\beta))-\sin(\beta - \alpha)| \nonumber\\\label{sinAalphbeta}
&=  |\sin(\beta-\alpha)(\cos(R(\alpha,\beta)) - 1)+\cos(\beta-\alpha)\sin(R(\alpha,\beta))| \\\nonumber
&\leq |\cos(R(\alpha,\beta)) - 1| + |\sin(R(\alpha,\beta))| .
\end{align}
Since $|R(\alpha,\beta)| \leq 4|x| + 2\|u\|_{\fzerone}$, we have that
\begin{multline}\label{sinAalphbetabound}
|\cos(R(\alpha,\beta)) - 1| + |\sin(R(\alpha,\beta))| 
\\
\leq \sum_{n=1}^{\infty} \frac{(4|x|\! +\! 2\|u\|_{\fzerone})^{2n}}{(2n)!}\! +\! \sum_{n=0}^{\infty}\frac{(4|x| \!+\! 2\|u\|_{\fzerone})^{2n+1}}{(2n+1)!} = e^{4|x| + 2\|u\|_{\fzerone}} \!-\! 1,
\end{multline}
which tends to zero as $(u,x)$ go to zero. Hence, by \eqref{sinAalphbeta} and \eqref{sinAalphbetabound} we have
\begin{multline*}
\int_{-\pi}^{\pi}\int_{-\pi}^{\pi}\cos(\alpha)\sin(\beta)(\sin(\psi(\beta)-\psi(\alpha)) -\sin(\beta - \alpha))d\alpha d\beta \\ \geq -4\pi^{2}\|\sin(\psi(\beta)-\psi(\alpha)) -\sin(\beta - \alpha)\|_{L^{\infty}} \\
\geq -4\pi^{2} (e^{4|x| + 2\|u\|_{\fzerone}} - 1)
\end{multline*}
and
$$\int_{-\pi}^{\pi}\int_{-\pi}^{\pi}\cos(\alpha)\sin(\beta) \sin(\beta - \alpha)d\alpha d\beta = \pi^{2}.$$
We thus conclude 
\begin{multline*}
    \Big|\int_{-\pi}^{\pi}\int_{-\pi}^{\pi}\cos(\alpha)\sin(\beta)\sin(\psi(\alpha)-\psi(\beta))d\alpha d\beta\Big| \\
    = \Big|\int_{-\pi}^{\pi}\int_{-\pi}^{\pi}\cos(\alpha)\sin(\beta)(\sin(\psi(\alpha)-\psi(\beta)) -\sin(\alpha - \beta))d\alpha d\beta\\ 
    + \int_{-\pi}^{\pi}\int_{-\pi}^{\pi}\cos(\alpha)\sin(\beta) \sin(\alpha - \beta)d\alpha d\beta\Big|\\
    \geq -4\pi^{2} (e^{4|x| + 2\|u\|_{\fzerone}} - 1) + \pi^{2}.
\end{multline*}
Thus, the determinant given by \eqref{det} is bounded from below by
\begin{equation}\label{detbound}
|\det(D_{x}\tilde{g}(u,x))| \geq 1+4\big(1- e^{4|x| + 2\|u\|_{\fzerone}}\big).
\end{equation}
Therefore, it only remains to estimate the norm $\|T(\psi)\|$ to obtain a bound for the norm of $D_{x}\tilde{g}(u,x)^{-1}$ in \eqref{aux1}.

Recalling \eqref{T.psi.def}, we calculate that 
\begin{equation}\notag
    T(\psi)\begin{bmatrix}
y_{1}\\
y_{2}
\end{bmatrix}
=
    \frac{1}{\pi} \begin{bmatrix}
\int_{-\pi}^{\pi}\sin(\alpha)\left( y_1 \sin(\psi(\alpha))+y_2\cos(\psi(\alpha)) \right)d\alpha
\\
\int_{-\pi}^{\pi}\cos(\alpha)\left( y_1 \sin(\psi(\alpha))+y_2\cos(\psi(\alpha))\right) d\alpha
\end{bmatrix}.
\end{equation}
We have for $\phi(\alpha) = \psi(\alpha) - \alpha = 2(x_{1}\cos(\alpha) - x_{2}\sin(\alpha)) + u(\alpha)$ that
\begin{multline}\notag
\sin\alpha\left( y_1 \sin(\psi(\alpha))+y_2\cos(\psi(\alpha)) \right)
\\
    =
    -\frac{1}{4}y_1
    \left(
    e^{i\phi}e^{2i\alpha}-e^{-i\phi}-e^{i\phi}
    +e^{-i\phi}e^{-2i\alpha}
    \right)
    \\
        +\frac{1}{4i}y_2
    \left(
    e^{i\phi}e^{2i\alpha}+e^{-i\phi}-e^{i\phi}
    -e^{-i\phi}e^{-2i\alpha}
    \right).
\end{multline}
And similarly
\begin{multline}\notag
\cos\alpha\left( y_1 \sin(\psi(\alpha))+y_2\cos(\psi(\alpha)) \right)
\\
    =
    \frac{1}{4i}y_1
    \left(
    e^{i\phi}e^{2i\alpha}-e^{-i\phi}+e^{i\phi}
    -e^{-i\phi}e^{-2i\alpha}
    \right)
    \\
        +\frac{1}{4}y_2
    \left(
    e^{i\phi}e^{2i\alpha}+e^{-i\phi}+e^{i\phi}
    +e^{-i\phi}e^{-2i\alpha}
    \right).
\end{multline}
We conclude that
\begin{multline}\notag
    T(\psi)
=
\\
\frac{1}{2}
\begin{bmatrix}
\left(\real{\mathcal{F}\{e^{i\phi}\}(0)}-\real{\mathcal{F}\{e^{i\phi}\}(-2)} \right)
&
\left(\imag{\mathcal{F}\{e^{i\phi}\}(-2)}-\imag{\mathcal{F}\{e^{i\phi}\}(0)} \right)
\\
\left(\imag{\mathcal{F}\{e^{i\phi}\}(-2)}-\imag{\mathcal{F}\{e^{i\phi}\}(0)} \right)
& \left(\real{\mathcal{F}\{e^{i\phi}\}(-2)}+\real{\mathcal{F}\{e^{i\phi}\}(0)} \right)
\end{bmatrix}.
\end{multline}
Hence 
\begin{equation*}
\Big|T(\psi) \begin{bmatrix}
y_{1}\\
y_{2}
\end{bmatrix} \Big| \leq 2 |y|\|\widehat{e^{i\phi(\alpha)}}\|_{\ell^{1}}.
\end{equation*}
Hence, using \eqref{ell1ft.def}, we obtain the bound
\begin{align*}
\|T(\psi)\| &\leq \sum_{n=0}^{\infty} \frac{\|\widehat{\phi^{n}}\|_{\ell^{1}}}{n!}\\ 
&\leq  \sum_{n=1}^{\infty}\sum_{k=0}^{n} {n \choose k}\frac{1}{n!}\Big\| \ast^{k} |\mathcal{F}\{2(x_{1}\cos\alpha-x_{2}\sin\alpha)\}|\ast^{n-k}|\hat{u}|\Big\|_{\ell^{1}}\\
&\leq  \sum_{n=1}^{\infty}\sum_{k=0}^{n} {n \choose k}\frac{1}{n!} 2^{k}|x|^{k}\|u\|_{\tildefzero}^{n-k}.
\end{align*}
This finally yields the result 
\begin{equation}\label{Tbound}
\|T(\psi)\|\leq 2e^{2|x|+\|u\|_{\tildefzero}}.
\end{equation}
From \eqref{aux1}, using \eqref{detbound} and \eqref{Tbound}, we obtain 
\begin{equation}\label{detandTbound}
\|D_{x}\tilde{g}(u,x)^{-1}\| \leq \frac{ 2e^{2|x|+\|u\|_{\tildefzero}}}{1+ 4(1- e^{4|x|+2\|u\|_{\tildefzero}})}.
\end{equation}
In conclusion, we have from \eqref{xinitialbound}, \eqref{Dugbound} and \eqref{detandTbound} that
\begin{equation}\label{xbound}
2|x| \leq  \frac{ 2e^{2|x|+\|u\|_{\tildefzero}}(e^{2|x| + \|u\|_{\tildefzero}} - 1)}{1- 4( e^{4|x|+2\|u\|_{\tildefzero}}-1)} \|u\|_{\tildefzero}.
\end{equation}
From the denominator term, we need $$ e^{4|x|+2\|u\|_{\tildefzero}} - 1 < \frac{1}{4}.$$
This yields the condition
\begin{equation}\label{xucriteria}
2|x| + \|u\|_{\tildefzero} < \frac{1}{2}\log\big(5/4\big),
\end{equation}
which ensures that the condition 
$2|x| + \|u\|_{\tildefzero} < \log(3/2)$ is also satisfied.   Using $|x| = |\hat{\theta}(\pm 1)|$ and $u = \tilde{\theta}$, we obtain 
$$2|x| + \|u\|_{\tildefzero} = |\hat{\theta}(1)| +  |\hat{\theta}(-1)| + \|\tilde{\theta}\|_{\tildefzero} =  \|\theta\|_{\fzerone}.$$
Thus, for $\|\theta\|_{\fzerone} <\frac{1}{2}\log\big(\frac{5}{4}\big)$ and using \eqref{xbound}, we obtain Proposition \ref{IFTprop}.
\end{proof}

The result in Proposition \ref{IFTprop} will be needed to close the \textit{a priori} estimates in Section \ref{subsecGlobal}.  In particular, Proposition \ref{linearfourier} show that the modes $\pm 1$ are neutral at the linear level (see also \eqref{auxx2}).

\chapter{Fourier multiplier estimates}\label{sec:FourierRandS}

In this chapter, we will prove the crucial estimates for the operators $\mathcal{R}$ from \eqref{R}  (in Proposition \ref{Restimates.prop}) and $\mathcal{S}$ from \eqref{S} (in Proposition \ref{Smultiplierprop}) in the norms ${\fzeronenu}$ and  ${\fsonenu}$ for $s> 0$ from \eqref{fzeroonenunorm} and \eqref{fsonenorm} respectively.  

We will use the following facts throughout the section:

\begin{lemma}\label{triangleprop}
We have the estimate
	\begin{equation}\label{zeroproduct}
		\|g_{1}g_{2}\cdots g_{n}\|_{\fzeronenu} \leq \prod_{k=1}^{n} \|g_{k}\|_{\fzeronenu}.
	\end{equation}
	For $s>0$, recalling \eqref{bfcn.def}, we have
	\begin{equation}\label{sproduct}
		\|g_{1}g_{2}\cdots g_{n}\|_{\fsonenu} \leq \bfcn(n,s)\sum_{j=1}^{n} \|g_{j}\|_{\fsonenu}
		\prod_{\substack{k=1 \\ k\ne j}}^{n}  \|g_{k}\|_{\fzeronenu}.
	\end{equation}
\end{lemma}

\begin{remark}\label{nuremark}
We note that these results and all the apriori estimates in this paper further hold with ${\fzeronenu}$ and ${\fsonenu}$ replaced by ${\fzerone}$ and ${\fsone}$ respectively since we can simply take $\nu=0$.
\end{remark}

\begin{proof}
First consider the case $0< s \leq 1$ and $n=2$ in \eqref{sproduct}. Then we use the inequalities
$$|k|^{s} \leq |k-j|^{s} + |j|^{s}, $$
and
$$e^{\nu(t)|k|} \leq e^{\nu(t)|k-j|}e^{\nu(t)|j|}, $$
to see that
\begin{align*}
	\|g_1 g_2\|_{\fsonenu} &= \sum_{k\in\mathbb{Z}} e^{\nu(t)|k|} e^{\nu(t)|k|} |k|^{s} 
	\left| \widehat{g_1 g_2}(k) \right|
	\\
 	&\leq \sum_{j,k\in\mathbb{Z}} e^{\nu(t)|k-j|}e^{\nu(t)|j|}(|k-j|^{s} + |j|^{s})  \left| \hat{g_1}(k-j)\hat{g_2}(j) \right| \\
	&\leq \|g_1\|_{\fsonenu}\|g_2\|_{\fzeronenu} + \|g_1\|_{\fzeronenu}\|g_2\|_{\fsonenu}.
\end{align*}
For general $n$ and $s>0$, recalling \eqref{bfcn.def}, we use the following inequality
\begin{equation}\label{s.inequality}
    |k|^{s} \leq \bfcn(n,s) (|k-k_{1}|^{s} + |k_{1}-k_{2}|^{s} + \ldots + |k_{n-2} - k_{n-1}|^{s} + |k_{n-1}|^{s}).
\end{equation}
Then the proofs of all the other cases in \eqref{zeroproduct} and \eqref{sproduct} follow similarly.  
\end{proof}

We will also use the following repeatedly in this chapter:
\begin{prop}
We have the useful bound for $\beta \in \mathbb{T}=[-\pi, \pi]$ and $l \ge 1$:
\begin{equation}\label{usefuloperatorbound}
\Big|\Big( \frac{i\beta}{(1-e^{-i\beta})}\Big)^{l}-1 \Big| \leq |\beta|\frac{l}{2}\Big(\frac{\pi}{2}\big)^{l-1}\sqrt{1 + \frac{\pi^{2}}{4}}.
\end{equation}
\end{prop}

\begin{proof}
By the mean value theorem we have
\begin{equation*}
		\Big|\Big( \frac{i\beta}{(1-e^{-i\beta})}\Big)^{l}-1 \Big| \leq |\beta| \Big\|\frac{d}{d\beta}\Big( \frac{i\beta}{(1-e^{-i\beta})}\Big)^{l}  \Big\|_{L^{\infty}({\mathbb{T}})}.
	\end{equation*}
	We compute the above derivative as:
	\begin{equation*}
		\frac{d}{d\beta}\Big( \frac{i\beta}{1-e^{-i\beta}}\Big)^{l}  = l\Big(\frac{i\beta}{1-e^{-i\beta}}\Big)^{l-1}\frac{d}{d\beta}\Big(\frac{i\beta}{1-e^{-i\beta}}\Big).
	\end{equation*}
	First, for $\beta\in\mathbb{T}$, we have 
	\begin{align*}
		\Big|\frac{\beta}{(1-e^{-i\beta})}\Big|&= \frac{|\beta|}{\sqrt{(1-\cos(\beta))^{2}+\sin^{2}(\beta)}}
		= \frac{|\beta|}{\sqrt{2-2\cos(\beta)}}\\
		&= \frac{|\beta|}{2\sin(\beta/2)}
		\leq \frac{\pi}{2}.
	\end{align*}
		Next, we have
	\begin{align*}
		\Big|\frac{d}{d\beta}\Big(\frac{\beta}{1-e^{-i\beta}}\Big)\Big| & = \frac{|1-e^{-i\beta}-i\beta e^{-i\beta}|}{|(1-e^{-i\beta})^{2}|}\\
		&=\frac{|(1-\cos(\beta)-\beta\sin(\beta)) + i (\sin(\beta)-\beta\cos(\beta))|}{4\sin^{2}(\beta/2)}.
	\end{align*}
	In absolute value, the real term in the numerator is
	$$|1-\cos(\beta)-\beta\sin(\beta)| = |2\sin^{2}(\beta/2) -2\beta\sin(\beta/2)\cos(\beta/2)|.$$
	Hence for $\beta\in\mathbb{T}$ we have  
	$$\frac{|1-\cos(\beta)-\beta\sin(\beta)| }{4\sin^{2}(\beta/2)} \leq \Big|\frac{1}{2} - \frac{\beta}{2\tan(\beta/2)}\Big|\leq \frac{1}{2}.$$
	Next, the imaginary part can be checked to have a maximum of 
	$$ \frac{|\sin(\beta)-\beta\cos(\beta)|}{4\sin^{2}(\beta/2)} \leq \frac{\pi}{4}.$$
	Thus we have 
	\begin{equation}\notag
	\Big|\frac{d}{d\beta}\Big(\frac{\beta}{1-e^{-i\beta}}\Big)\Big|\leq \frac{1}{2}\sqrt{1 + \frac{\pi^{2}}{4}}.
	\end{equation}
	We conclude that
	$$\Big|\frac{d}{d\beta}\Big[\Big( \frac{i\beta}{1-e^{-i\beta}}\Big)^{l}-1\Big] \Big| \leq \frac{1}{2}l\Big(\frac{\pi}{2}\Big)^{l-1}\sqrt{1 + \frac{\pi^{2}}{4}}.$$
	This yields \eqref{usefuloperatorbound}.
\end{proof}

\section{Estimates on the operator $\mathcal{R}$}
We will now estimate the operator $\mathcal{R}$ from \eqref{R} as follows. 

\begin{prop}\label{Restimates.prop}
	The operator $\mathcal{R}$ from \eqref{R} satisfies the estimates
	\begin{equation}\label{Restimates}
\begin{aligned}
		\|\mathcal{R}(f)\|_{\fzeronenu} &\leq {\crconstant}\|f\|_{\fzeronenu}\|\theta\|_{\fzeronenu},
		\\
		\|\mathcal{R}(f)\|_{\fsonenu} &\leq {\bfcn(2,s)} {\crconstant}(\|f\|_{\fsonenu}\|\theta\|_{\fzeronenu} + \|f\|_{\fzeronenu}\|\theta\|_{\fsonenu}),\qquad s>0.
\end{aligned}
\end{equation}
where $\bfcn(2,s)$ is from \eqref{bfcn.def} and the constant
\begin{equation}\label{CR}
{\crconstant}	\eqdef 1+\tilde{\mathcal{C}}>0,    
\end{equation}
and $\tilde{\mathcal{C}}>0$ is defined in \eqref{catalanconstant}.
\end{prop}

In the rest of this chapter we will adopt the convention that 
	\begin{equation}\label{convention}
	\frac{1-e^{-i\beta(1+k_1)}}{1+k_1}  = i\beta \quad \text{for} ~k_1=-1. 
	\end{equation}
This convention will allow us to write many formula's succinctly.

\begin{proof}
Taking the Fourier transform of the operator $\mathcal{R}$ from \eqref{R} and using the convention \eqref{convention} we obtain 
	\begin{equation*}
	   \begin{aligned}
	   	\widehat{\mathcal{R}(f)}(k) &=\!\frac{i}{\pi}\pv\!\!\int_{-\pi}^{\pi}\!\frac{\hat{f}(k)e^{-ik\beta}\beta}{(1\!-\!e^{-i\beta})^2}\ast \!\!\int_0^1\!\!e^{i(s-1)\beta(1+k)}\hat{\theta}(k) dsd\beta\\
		&=\!\frac{i}{\pi}\sum_{k_1\in\mathbb{Z}}\pv\!\!\int_{-\pi}^{\pi}\!\frac{\hat{f}(k-k_1)e^{-i(k-k_{1})\beta}\beta}{(1\!-\!e^{-i\beta})^2} \!\!\int_0^1\!\!e^{i(s-1)\beta(1+k_1)}\hat{\theta}(k_1) dsd\beta\\
		&= \!\frac{1}{\pi}\sum_{k_1\in\mathbb{Z}}\pv\!\!\int_{-\pi}^{\pi}\!\frac{\hat{f}(k-k_1)e^{-i(k-k_{1})\beta}}{(1\!-\!e^{-i\beta})^2} \frac{1-e^{-i\beta(1+k_1)}\hat{\theta}(k_1)}{1+k_1} d\beta.
	    \end{aligned}
	\end{equation*}
		Using the convention \eqref{convention}, we have
		\begin{equation*}
	\widehat{\mathcal{R}(f)}(k)=  \sum_{k_1\in\mathbb{Z}} \hat{f}(k-k_1)\hat{\theta}(k_1)I(k,k_1),
	\end{equation*}
   where
	\begin{equation}\label{I.function.def}
	 I(k,k_1)\eqdef \frac{1}{\pi}\pv\!\int_{-\pi}^{\pi}\!\frac{e^{-i(k-k_{1})\beta}}{(1\!-\!e^{-i\beta})^2} \frac{1-e^{-i\beta(1+k_1)}}{1+k_1} d\beta.
	\end{equation}
	For $k_{1} = -1$, using \eqref{convention}, we split \eqref{I.function.def} as 
	\begin{equation*}
		I(k,-1) = \frac{1}{\pi} \pv\int_{-\pi}^{\pi}\frac{e^{-i(k+1)\beta}}{1-e^{-i\beta}} d\beta 
		+ \frac{1}{\pi}\pv\int_{-\pi}^{\pi}\frac{e^{-i(k+1)\beta}}{1-e^{-i\beta}} \Big(\frac{i\beta}{1-e^{-i\beta}} - 1 \Big).
	\end{equation*}
We will first calculate each of these two integrals separately below.

We can calculate the general integral formula:
\begin{equation}\label{j1}
 \frac{1}{\pi} \pv \int_{-\pi}^{\pi} \frac{e^{-i \ell \beta}}{1-e^{-i\beta}} d\beta
 =
 1_{\ell \le 0}(\ell) -  1_{\ell \ge 1}(\ell).
\end{equation}
This will be used several times below.  It is proven using  \eqref{hilbert} and noticing that $\mathcal{F}(e^{-i \ell \beta})(0)=\delta(\ell)$.  Further from \eqref{defHilbertTransform} we have
$$
-i\mH(e^{-i \ell \beta})(0) = 
-\frac{1}{2\pi}\pv\int_{-\pi}^{\pi}\frac{\sin(\ell\beta)}{\tan{(\beta/2)}} d\beta
= 1_{\ell \le -1}(\ell) -  1_{\ell \ge 1}(\ell).
$$  
Combining these calculations gives \eqref{j1}.
	
	From \eqref{j1}, the first term in $I(k,-1)$ is $\pm 1$
	depending on the sign of $k+1$. For the second integral, we use \eqref{usefuloperatorbound} for $l=1$ to obtain
	\begin{align*}
	\Big|\frac{1}{\pi}\pv\int_{-\pi}^{\pi}\frac{e^{-i(k+1)\beta}}{1-e^{-i\beta}} \Big(\frac{i\beta}{1-e^{-i\beta}} - 1 \Big)\Big| &\leq \frac{1}{\pi}\sqrt{\frac{1}{4} + \frac{\pi^{2}}{16}} \int_{-\pi}^{\pi} \frac{|\beta|}{2|\sin(\beta/2)|} d\beta\\
	&\leq \tilde{\mathcal{C}},
	\end{align*}
	where 
	\begin{equation}\label{catalanconstant}
		\tilde{\mathcal{C}} \eqdef \frac{4}{\pi}\mathcal{V}\sqrt{1 + \frac{\pi^{2}}{4}}.
	\end{equation}
In this calculation we used that $|1-e^{-i\beta}| = 2|\sin(\beta/2)|$.  Here $\mathcal{V} \approx 0.916$ is Catalan's constant:
	\begin{equation}\notag
	    \mathcal{V} = \frac{1}{4}\int_{-\pi/2}^{\pi/2} \frac{\beta}{\sin\beta} d\beta
	    =\frac{1}{16}\int_{-\pi}^{\pi} \frac{\beta}{\sin(\beta/2)} d\beta.
	\end{equation}
	Hence, for $k_{1}= -1$, using \eqref{CR}, we have the bound
	\begin{equation}\label{k1negativeone}
	|I(k,-1)| \leq 1+ \tilde{\mathcal{C}} = {\crconstant}.
	\end{equation}
This will be our main estimate for the case for $k_{1}= -1$.

Now generally in \eqref{I.function.def} for $k_1>-1$ we have
	\begin{equation}\label{big.one.def}
    	\frac{1-e^{-i\beta (1+k_{1})}}{1-e^{-i\beta}} = \sum_{r=0}^{k_{1}} e^{-ir\beta}
\end{equation}
	and
    if $k_1 \leq -2$ then we have
    \begin{equation}\label{small.one.def}
\frac{1-e^{-i\beta (1+k_{1})}}{1-e^{-i\beta}} = -\sum_{r=1}^{-1-k_{1}} e^{i\beta r}.
\end{equation}
Thus for $k_1>-1$, using \eqref{I.function.def},  \eqref{j1} and \eqref{big.one.def} we have
\begin{multline*}
	I(k,k_1) 
	= \frac{1}{1+k_1}\sum_{r=0}^{k_1} \frac{1}{\pi}\pv\!\int_{-\pi}^{\pi}\!\frac{e^{-i(k-k_{1}+r)\beta}}{1\!-\!e^{-i\beta}}  d\beta
	\\
	=\frac{1}{1+k_1}\sum_{r=0}^{k_1} \left( 1_{k-k_{1}+r \le 0} -  1_{k-k_{1}+r \ge 1} \right),
	\end{multline*}
	while for $k_1 \leq -2$ we similarly, using \eqref{small.one.def}, have
	\begin{multline*}
	I(k,k_1) = - \frac{1}{1+k_1}\sum_{r=1}^{-1-k_1} \frac{1}{\pi}\pv\!\int_{-\pi}^{\pi}\!\frac{e^{-i(k-k_{1}-r)\beta}}{1\!-\!e^{-i\beta}}  d\beta
	\\
	=\frac{-1}{1+k_1}\sum_{r=1}^{-1-k_1} \left( 1_{k-k_{1}-r \le 0} -  1_{k-k_{1}-r \ge 1} \right),
	\end{multline*}
	In both cases, $k_1>-1$ and $k_1 \leq -2$, we conclude that 
	    \begin{equation}\label{I.est.rest}
    |I(k,k_1)| \leq 1.
    \end{equation}
Hence we conclude that
	\begin{equation*}
	|\widehat{\mathcal{R}(f)}(k)| \leq  	{\crconstant}|(\hat{f}\ast \hat{\theta})(k)|.   
	\end{equation*}
	Applying ${\fzeronenu}$ and ${\fsonenu}$ norms to both sides, using Lemma \ref{triangleprop}, gives the result.
\end{proof}

\section{Estimates on the non-linear operator $\mathcal{S}$}
		
Next, we proceed with the estimates for $\mathcal{S}$ from \eqref{S}.
\begin{prop}\label{Smultiplierprop}
	The operator $\mathcal{S}$ in \eqref{S} satisfies the estimates
	\begin{equation}\label{Sestimates}
    \begin{aligned}
    	\|\mathcal{S}(f)\|_{\fzeronenu} &\leq C_1 \|\theta\|_{\fzeronenu}^2\|f\|_{\fzeronenu},\\
			\|\mathcal{S}(f)\|_{\fsonenu}& \leq C_3\|\theta\|_{\fzeronenu}^{2}\|f\|_{\fsonenu}+C_4\|\theta\|_{\fzeronenu}\|f\|_{\fzeronenu}\|\theta\|_{\fsonenu},     \quad s>0,
    \end{aligned}
	\end{equation}
	where the positive constant $C_1$ is given by \eqref{C1}.  Further the positive constants $C_{3}$ and $C_{4}$ for $s>0$ are given in \eqref{C3}.
\end{prop}

\begin{proof}
From \eqref{S}, we split  the operator $\mathcal{S}$ as: 
	\begin{equation}\label{Ssum}
	\mathcal{S}(f)(\alpha) = 
	{\sonesplit}(f)(\alpha) + {\stwosplit}(f)(\alpha),
	\end{equation}
	where
	\begin{multline}\label{breveS}
	{\sonesplit}(f)(\alpha) \eqdef \\ \frac{1}{\pi}\pv\int_{-\pi}^{\pi}\frac{f(\alpha-\beta)\beta^{2}}{\beta(1-e^{-i\beta})^{2}}
	\sum_{m\geq 2}\frac{i^m}{m!}\int_0^1 e^{i(s-1)\beta}(\theta(\alpha+(s-1)\beta))^m ds d\beta,
	\end{multline}
	and
	\begin{equation}\label{tildeS}
		{\stwosplit}(f)(\alpha)\eqdef \sum_{\substack{n,l\geq0 \\ n+l\geq 2}}\frac{(-1)^n i^{l+n+1}(\theta(\alpha))^l}{l!} \mathcal{S}_{n}(f)(\alpha).
	\end{equation}
Here for $n\ge0$ we define
	\begin{equation}\label{Snl}
		\mathcal{S}_{n}(f)(\alpha) \eqdef \frac{1}{\pi} \pv\int_{-\pi}^{\pi}\frac{f(\alpha-\beta)\beta^{n+1}}{\beta(1-e^{-i\beta})^{n+1}}
		M(\alpha,\beta)^{n} d\beta,
	\end{equation}
	with
	\begin{equation}\label{M}
		M(\alpha,\beta) \eqdef \sum_{m\geq 1}\frac{i^m}{m!}\int_0^1 e^{i(s-1)\beta}(\theta(\alpha+(s-1)\beta))^m ds.
	\end{equation}
We take the Fourier transform to obtain
	\begin{equation}\label{fourierS}
		\widehat{{\stwosplit}}(f)(k) = \sum_{\substack{n,l\geq0 \\ n+l\geq 2}}\frac{(-1)^ni^{l+n+1}(\ast^{l}\hat{\theta}(k))}{l!} \ast \widehat{\mathcal{S}_{n}(f)}(k),
	\end{equation}
	where 
	\begin{equation}\label{fourierSnl}
		\widehat{\mathcal{S}_{n}(f)}(k) \eqdef \frac{1}{\pi} \pv\int_{-\pi}^{\pi}\frac{\hat{f}(k)e^{-ik\beta}\beta^{n+1}}{\beta(1-e^{-i\beta})^{n+1}}
		\ast^{n}\widehat{M}(k,\beta) d\beta.
	\end{equation}
For $k_{1} \neq -1$, from \eqref{M} we have
	\begin{align*}
		\widehat{M}(k_{1},\beta) &= \sum_{m\geq 1}\frac{i^m}{m!}\int_0^1 e^{i(s-1)\beta}\Big(\ast^{m}(\hat{\theta}(k_{1}) e^{ik_{1}(s-1)\beta})\Big) ds\\
		&= \sum_{m\geq 1}\sum_{k_{2},\ldots,k_{m}\in\mathbb{Z}}\frac{i^m}{m!} \int_0^1 e^{i(s-1)\beta}\left(\prod_{j=1}^{m-1}\hat{\theta}(k_{j}-k_{j+1}) e^{i(k_{j}-k_{j+1})(s-1)\beta}\right)\\
		&\hspace{3in} \cdot\hat{\theta}(k_{m}) e^{ik_{m}(s-1)\beta}  ds\\
		&= \sum_{m\geq 1}\sum_{k_{2},\ldots,k_{m}\in\mathbb{Z}}\frac{i^m}{m!}\int_0^1 e^{i(s-1)\beta(1+k_{1})}\left(\prod_{j=1}^{m-1}\hat{\theta}(k_{j}-k_{j+1}) \right) \hat{\theta}(k_{m}) ds\\
		&= \sum_{m\geq 1}\sum_{k_{2},\ldots,k_{m}\in\mathbb{Z}}\frac{i^m}{m!} \left(\prod_{j=1}^{m-1}\hat{\theta}(k_{j}-k_{j+1}) \right) \hat{\theta}(k_{m}) \frac{1-e^{-i\beta (1+k_{1})}}{i\beta (1+k_{1})}\\
		&=  \sum_{m\geq 1}\frac{i^m}{m!} \left(\ast^{m}\widehat{\theta}(k_{1})\right) \frac{1-e^{-i\beta (1+k_{1})}}{i\beta (1+k_{1})}.
	\end{align*}
And when $k_{1} = -1$,  analogously we have
	$$\widehat{M}(k_{1},\beta) =  \sum_{m\geq 1}\frac{i^m}{m!} \left(\ast^{m}\widehat{\theta}(k_{1})\right).$$
Thus we define
	\begin{equation}\label{P}
		P(k) \eqdef \sum_{m\geq 1}\frac{i^m}{m!} (\ast^{m}\widehat{\theta}(k)).
	\end{equation}
Now we use the convention \eqref{convention} and plug the above into \eqref{fourierSnl} to obtain
	\begin{align}
		\widehat{\mathcal{S}_{n}}(k_{1}) &= \frac{1}{\pi} \pv\int_{-\pi}^{\pi}\frac{\hat{f}(k_{1})e^{-ik_{1}\beta}\beta^{n+1}}{\beta(1-e^{-i\beta})^{n+1}}\ast
		\left(\ast^{n}P(k_{1})\frac{1-e^{-i\beta (1+k_{1})}}{i\beta (1+k_{1})}\right) d\beta \nonumber\\
		&= \frac{1}{i^{n}}\sum_{k_{2},\ldots, k_{n+1}\in \mathbb{Z}}I(k_{1},\ldots,k_{n+1})\hat{f}(k_{n+1})\prod_{j=1}^{n}P(k_{j}-k_{j+1}) \hspace{0.05cm}, \label{Snfourier}
	\end{align}
	where $I = I(k_{1},\ldots,k_{n+1})$ is defined by 
	\begin{equation}\label{I}
		I \eqdef \frac{1}{\pi}\pv \int_{-\pi}^{\pi} \frac{e^{-i k_{n+1}\beta}}{1-e^{-i\beta}}\prod_{j=1}^{n}
		\frac{1-e^{-i\beta (1+k_{j}-k_{j+1})}}{(1+k_{j}-k_{j+1})(1-e^{-i\beta})}d\beta.
	\end{equation}
	We suppose that $l$ elements of $\{k_{j}-k_{j+1}\}_{j=1}^n$ satisfy $k_{j} - k_{j+1} =  -1$ for $0\le l \le n$.  Then ordering the subscripts such that $k_{j} - k_{j+1} \neq  -1$ for $j= 1,\ldots, n-l$, we obtain with the convention \eqref{convention} that the integral $I = I(k_{1},\ldots, k_{n+1})$ given by \eqref{I} becomes 
	\begin{equation*}
		I = \frac{1}{\pi}\pv \int_{-\pi}^{\pi} \frac{e^{-i k_{n+1}\beta}}{1-e^{-i\beta}}\frac{(i\beta)^{l}}{(1-e^{-i\beta})^l}
		\prod_{j=1}^{n-l}
		\frac{1-e^{-i\beta (1+k_{j}-k_{j+1})}}{(1+k_{j}-k_{j+1})(1-e^{-i\beta})}d\beta.
	\end{equation*}
	Now, if $k_{j}-k_{j+1} > -1$ then similar to \eqref{big.one.def} we have
\begin{equation}\notag 
    	\frac{1-e^{-i\beta (1+k_{j}-k_{j+1})}}{1-e^{-i\beta}} = \sum_{r_{j}=0}^{k_{j}-k_{j+1}} e^{-ir_{j}\beta}
\end{equation}
	and
    if $k_{j}-k_{j+1} \leq -2$ then similar to \eqref{small.one.def} we have
    \begin{equation}\notag 
\frac{1-e^{-i\beta (1+k_{j}-k_{j+1})}}{1-e^{-i\beta}} = \sum_{r_{j}=1}^{-1-(k_{j}-k_{j+1})} -e^{i\beta r_j}.
\end{equation}
    Hence, if $k_{j}-k_{j+1} \leq -2$ only for $j=m,\ldots,n-l$ then
    \begin{multline*}
    \prod_{j=1}^{n-l}
		\frac{1-e^{-i\beta (1+k_{j}-k_{j+1})}}{1-e^{-i\beta}}\\ = \sum_{r_{1}=0}^{k_{1}-k_{2}} e^{-ir_{1}\beta} \cdots \sum_{r_{m-1}=0}^{k_{m-1}-k_{m}} e^{-ir_{m-1}\beta} \sum_{r_{m}=1}^{k_{m+1}-1-k_{m}} -e^{ir_{m}\beta}\\ \cdots \sum_{r_{n-l}=1}^{k_{n-l+1}-1-k_{n-l}} -e^{ir_{n-l}\beta}\\
		=\sum_{r_{1}=0}^{k_{1}-k_{2}}\cdots \sum_{r_{m-1}=0}^{k_{m-1}-k_{m}}  \sum_{r_{m}=1}^{k_{m+1}-1-k_{m}} \cdots \sum_{r_{n-l}=1}^{k_{n-l+1}-1-k_{n-l}}(-1)^{n-l-m+1} e^{-i\tilde{A}\beta}
		\end{multline*}
		where
		$$\tilde{A} = r_{1} + \ldots + r_{m-1} -r_{m} -\ldots -r_{n-l}.$$
		Hence, in this case
	\begin{multline}\label{IJ}
	    I =  \prod_{j=1}^{n-l}
		\frac{1}{1+k_{j}-k_{j+1}}\sum_{r_{1}=0}^{k_{1}-k_{2}}\cdots \sum_{r_{m-1}=0}^{k_{m-1}-k_{m}}  \sum_{r_{m}=1}^{k_{m+1}-1-k_{m}} \\ \cdots \sum_{r_{n-l}=1}^{k_{n-l+1}-1-k_{n-l}}(-1)^{n-l-m+1}
		\frac{1}{\pi}\pv \int_{-\pi}^{\pi} \frac{e^{-i A\beta}(i\beta)^{l}}{(1-e^{-i\beta})^{l+1}} d\beta
	\end{multline}
	where
	$$A = k_{n+1}+\tilde{A}.$$
	For the inner integral, which we define as $J$, we have
	\begin{equation}\label{j1j2}
		J \eqdef \frac{1}{i^{l}\pi}\pv \int_{-\pi}^{\pi} \frac{e^{-i A\beta}(i\beta)^{l}}{(1-e^{-i\beta})^{l+1}} d\beta = \frac{1}{i^{l}}( J_{1} + J_{2}), 
	\end{equation}
	where we calculate $J_1$ as in \eqref{j1}
	and
	\begin{equation}\label{j2}
		J_{2} \eqdef  \frac{1}{\pi} \pv \int_{-\pi}^{\pi} \frac{e^{-i A\beta}}{1-e^{-i\beta}}\Big[\Big( \frac{i\beta}{(1-e^{-i\beta})}\Big)^{l}-1\Big] d\beta.
	\end{equation}
	We can bound $J_{2}$ as well. Using the bound \eqref{usefuloperatorbound} in \eqref{j2}, we have 
	\begin{align*}
		|J_{2}| &\leq \frac{1}{2\pi}l\Big(\frac{\pi}{2}\Big)^{l-1}\sqrt{1 + \frac{\pi^{2}}{4}} \int_{-\pi}^{\pi} \frac{|\beta|}{2|\sin(\beta/2)|} d\beta\\
		&=\tilde{\mathcal{C}} l\Big(\frac{\pi}{2}\Big)^{l-1},
	\end{align*}
	where the calculation above is similar to the calculation above \eqref{catalanconstant} and the constant $\tilde{\mathcal{C}}$ is given by \eqref{catalanconstant}.

Thus, from \eqref{IJ} and \eqref{j1j2}, we have
	\begin{align*}
		|I|&\leq\prod_{j=1}^{n-l}
		\frac{1}{|1+k_{j}-k_{j+1}|}\sum_{r_{1}=0}^{k_{1}-k_{2}}\cdots \sum_{r_{m-1}=0}^{k_{m-1}-k_{m}}  \sum_{r_{m}=1}^{k_{m+1}-1-k_{m}} \\ 
		&\hspace{2in}\cdots \sum_{r_{n-l}=1}^{k_{n-l+1}-1-k_{n-l}}  \left(1 + \tilde{\mathcal{C}} l\big(\frac{\pi}{2}\big)^{l-1}\right)\\
		&\leq\prod_{j=1}^{n-l}
		\frac{1}{|1+k_{j}-k_{j+1}|}\sum_{r_{1}=0}^{k_{1}-k_{2}}\cdots \sum_{r_{m-1}=0}^{k_{m-1}-k_{m}}  \sum_{r_{m}=1}^{k_{m+1}-1-k_{m}} 
		\\ 
		&\hspace{2in}\cdots \sum_{r_{n-l}=1}^{k_{n-l+1}-1-k_{n-l}} \left( 1 + \tilde{\mathcal{C}} n\big(\frac{\pi}{2}\big)^{n-1}\right).
	\end{align*}
Then we denote
\begin{equation}\label{an.def}
	    a_n\eqdef 1 + \tilde{\mathcal{C}} n\big(\frac{\pi}{2}\big)^{n-1}.
\end{equation}
Therefore we have
	\begin{equation}\label{Ibound}
		|I(k_{1},\ldots,k_{n+1})| \leq a_n \ \ \forall  (k_{1},\ldots, k_{n+1}) \in \mathbb{Z}^{n+1}.
	\end{equation}	
	Plugging \eqref{Ibound} into the estimate for \eqref{Snfourier}, we obtain
	\begin{align}
		|\widehat{\mathcal{S}_{n}}(k_{1})|&\leq  \sum_{k_{2},\ldots, k_{n+1}\in \mathbb{Z}} |I(k_{1},\ldots,k_{n+1})||\hat{f}(k_{n+1})|
		\prod_{j=1}^n |P(k_{j}-k_{j+1})| \nonumber\\
		&\leq  \sum_{k_{2},\ldots, k_{n+1}\in \mathbb{Z}}  a_n |\hat{f}(k_{n+1})|\prod_{j=1}^n |P(k_{j}-k_{j+1})|\nonumber\\
		& = a_n(|\hat{f}|\ast^{n} |P|)(k_{1}). \label{Snfbound}
	\end{align}	
	We will use this estimate below to obtain the appropriate upper bound for $\mathcal{S}(f)$ given by \eqref{fourierS}.

Recalling \eqref{nu}, in the rest of this proof for convenience of notation we define the $\ell^1_\nu$ norm of a sequence $a=\{a_k\}_{k\in \mathbb{Z}}$ as
\begin{equation}\notag
    \| a\|_{\ell^{1}_\nu} \eqdef \sum_{k\in\mathbb{Z}} e^{\nu |k|} |a_k|^p.
\end{equation}
Then for $s=0$ using \eqref{tildeS} and \eqref{fourierS} we have
	\begin{align}
		\|{\stwosplit}(f)\|_{\fzeronenu} &= \Big\|\sum_{\substack{n,l\geq 0 \\ n+l\geq 2}}\frac{(-1)^ni^{l+n+1}(\ast^{l}\hat{\theta}(k))}{l!} \ast \widehat{\mathcal{S}_{n}(f)}(k) \Big\|_{\ell^{1}_{\nu}} \nonumber \\
		&\leq \sum_{\substack{n\geq 0,l\geq 0 \\n+l \geq 2}}\frac{1}{l!}\|\hat{\theta}(k)\|_{\ell^{1}_{\nu}}^{l}\|\widehat{\mathcal{S}_{n}(f)}(k)\|_{\ell^{1}_{\nu}} \label{SSn0}.
	\end{align}
	Let us examine the bound on the $\widehat{\mathcal{S}_{n}(f)}$ term in \eqref{SSn0} using \eqref{Snfbound}. For the quantity $P$ given by \eqref{P}, 	we have
	\begin{equation}\label{P0}
		\|P\|_{\ell^{1}_{\nu}} \leq \sum_{m\geq 1}\frac{1}{m!} \|\theta\|_{\fzeronenu}^{m} = \exp(\|\theta\|_{\fzeronenu})-1.
	\end{equation}
	From \eqref{Snfbound} and \eqref{an.def}, we then have
	\begin{equation}\label{Snfzero}
		\|\mathcal{S}_{n}(f)\|_{\fzeronenu} \leq   a_n\|f\|_{\fzeronenu}(\exp(\|\theta\|_{\fzeronenu})-1)^{n}.
	\end{equation}
	Hence, from \eqref{SSn0}, we have
	\begin{equation}\label{SSn00}
		\|{\stwosplit}(f)\|_{\fzeronenu} \leq 
		\! \! \! \! \! \sum_{\substack{n\geq 0,l\geq 0 \\n+l \geq 2}}\frac{a_n}{l!}\|\hat{\theta}(k)\|_{\ell^{1}_{\nu}}^{l}   \|f\|_{\fzeronenu}(\exp(\|\theta\|_{\fzeronenu})-1)^{n}.
	\end{equation}
We will separately look at the above sum when $n=0$, $n=1$ and $n\geq 2$.

First, when $n=0$ then $a_0=1$ and we have
	\begin{equation}\label{SSn00n0}
		\sum_{l\geq 2}\frac{1}{l!}\|\hat{\theta}(k)\|_{\ell^{1}_{\nu}}^{l}   \|f\|_{\fzeronenu} = \frac{\exp(\|\theta\|_{\fzeronenu}) - 1 - \|\theta\|_{\fzeronenu}}{\|\theta\|_{\fzeronenu}^{2}} \|\theta\|_{\fzeronenu}^{2}\|f\|_{\fzeronenu}.
	\end{equation}
	When $n=1$ then $a_1=1+\tilde{\mathcal{C}}$ and we have
	\begin{multline}\label{SSn0n1}
		\sum_{l\geq 1}\frac{1}{l!}\|\hat{\theta}(k)\|_{\ell^{1}_{\nu}}^{l}\Big(1 + \tilde{\mathcal{C}} \Big)(\exp(\|\theta\|_{\fzeronenu})-1)   \|f\|_{\fzeronenu} \\ =\Big(1 + \tilde{\mathcal{C}}\Big)\frac{(\exp(\|\theta\|_{\fzeronenu})-1)^{2}}{\|\theta\|_{\fzeronenu}^{2}}\|\theta\|_{\fzeronenu}^{2}\|f\|_{\fzeronenu}.
	\end{multline}
	Finally, for $n\geq 2$ using also \eqref{an.def} we have
	\begin{multline}\label{SSn0n2}
		\sum_{n\geq 2,l\geq 0 }\frac{a_n}{l!}\|\hat{\theta}(k)\|_{\ell^{1}_{\nu}}^{l}   \|f\|_{\fzeronenu}(\exp(\|\theta\|_{\fzeronenu})-1)^{n} \\= \sum_{n\geq 2,l\geq 0 }\frac{1}{l!}\|\hat{\theta}(k)\|_{\ell^{1}_{\nu}}^{l} \|f\|_{\fzeronenu}(\exp(\|\theta\|_{\fzeronenu})-1)^{n}\\ 
		+ \sum_{n\geq 2,l\geq 0 }\frac{1}{l!}
		\tilde{\mathcal{C}} n\big(\frac{\pi}{2}\big)^{n-1}
		\|\hat{\theta}(k)\|_{\ell^{1}_{\nu}}^{l}   
		\|f\|_{\fzeronenu}(\exp(\|\theta\|_{\fzeronenu})-1)^{n}.
	\end{multline}
Now considering the first sum on the right hand side of \eqref{SSn0n2}, we have
	\begin{multline}\label{SSn0n21}
		\sum_{n\geq 2,l\geq 0 }\frac{1}{l!}\|\hat{\theta}(k)\|_{\ell^{1}_{\nu}}^{l} \|f\|_{\fzeronenu}(\exp(\|\theta\|_{\fzeronenu})-1)^{n}\\ = \frac{\exp(\|\theta\|_{\fzeronenu})(\exp(\|\theta\|_{\fzeronenu})-1)^{2}}{\|\theta\|_{\fzeronenu}^{2}}\sum_{n\geq 2} \|f\|_{\fzeronenu}\|\theta\|_{\fzeronenu}^{2}(\exp(\|\theta\|_{\fzeronenu})-1)^{n-2}\\
		=\frac{\exp(\|\theta\|_{\fzeronenu})(\exp(\|\theta\|_{\fzeronenu})-1)^{2}}{\|\theta\|_{\fzeronenu}^{2}(2-\exp(\|\theta\|_{\fzeronenu}))} \|f\|_{\fzeronenu}\|\theta\|_{\fzeronenu}^{2}.
	\end{multline}
Then the second sum on the right hand side of \eqref{SSn0n2} is
\begin{multline*}
\sum_{n\geq 2,l\geq 0 }\frac{1}{l!}\tilde{\mathcal{C}} n\big(\frac{\pi}{2}\big)^{n-1}\|\hat{\theta}(k)\|_{\ell^{1}_{\nu}}^{l}   \|f\|_{\fzeronenu}(\exp(\|\theta\|_{\fzeronenu})-1)^{n} 
\\
=\frac{\exp(\|\theta\|_{\fzeronenu})(\exp(\|\theta\|_{\fzeronenu})-1)}{\|\theta\|_{\fzeronenu}}\\ \cdot \sum_{n\geq 2}\tilde{\mathcal{C}} n\big(\frac{\pi}{2}(\exp(\|\theta\|_{\fzeronenu})-1)\big)^{n-1} \|f\|_{\fzeronenu}\|\theta\|_{\fzeronenu}.
\end{multline*}
Now using that $$\sum_{ n\geq 2} nx^{n-1} = -1 + \sum_{n\geq 1} nx^{n-1} = -1 + \frac{1}{(1-x)^{2}} = \frac{x(2-x)}{(1-x)^{2}}, $$
	we obtain
\begin{multline}\label{SSn00n22}
		\sum_{n\geq 2,l\geq 0 }\frac{1}{l!}\|\hat{\theta}(k)\|_{\ell^{1}_{\nu}}^{l}   \tilde{\mathcal{C}} n\big(\frac{\pi}{2}\big)^{n-1}\|f\|_{\fzeronenu}(\exp(\|\theta\|_{\fzeronenu})-1)^{n} \\
=\frac{\pi\tilde{\mathcal{C}}}{2}
		\frac{\exp(\|\theta\|_{\fzeronenu})(\exp(\|\theta\|_{\fzeronenu})-1)^{2}}{\|\theta\|_{\fzeronenu}^{2}}\|f\|_{\fzeronenu}\|\theta\|_{\fzeronenu}^{2}\\\cdot\Big( \frac{2-\frac{\pi}{2}(\exp(\|\theta\|_{\fzeronenu})-1)}{\big(1-\frac{\pi}{2}(\exp(\|\theta\|_{\fzeronenu})-1)\big)^{2}}\Big).
\end{multline}	
	Combining \eqref{SSn00n0}, \eqref{SSn0n1}, \eqref{SSn0n2}, \eqref{SSn0n21} and \eqref{SSn00n22} into \eqref{SSn00}, we obtain in the space $\fzeronenu$ with $\nu=0$ that
	\begin{equation*}
		\|{\stwosplit}(f)\|_{\fzeronenu} \leq \tilde{C}_1\|\theta\|_{\fzeronenu}^{2}\|f\|_{\fzeronenu},
	\end{equation*} 
	where $\tilde{C}_1\eqdef \tilde{C}_1(\|\theta\|_{\fzeronenu})$ is an increasing function of $\|\theta\|_{\fzeronenu}$ given by
	\begin{multline}\label{tildeC1}
		\tilde{C}_1=
		\\
		\frac{\pi\tilde{\mathcal{C}}}{2}\frac{\exp(\|\theta\|_{\fzeronenu})(\exp(\|\theta\|_{\fzeronenu})-1)^{2}}{\|\theta\|_{\fzeronenu}^{2}} \frac{2-\frac{\pi}{2}(\exp(\|\theta\|_{\fzeronenu})-1)}{\big(1-\frac{\pi}{2}(\exp(\|\theta\|_{\fzeronenu})-1)\big)^{2}} \\+ \frac{\exp(\|\theta\|_{\fzeronenu}) - 1 - \|\theta\|_{\fzeronenu}}{\|\theta\|_{\fzeronenu}^{2}} + \Big(1 + \tilde{\mathcal{C}}\Big)\frac{(\exp(\|\theta\|_{\fzeronenu})-1)^{2}}{\|\theta\|_{\fzeronenu}^{2}} \\+\frac{\exp(\|\theta\|_{\fzeronenu})(\exp(\|\theta\|_{\fzeronenu})-1)^{2}}{\|\theta\|_{\fzeronenu}^{2}(2-\exp(\|\theta\|_{\fzeronenu}))}.
	\end{multline}
 This completes the estimate \eqref{Sestimates} for the operator \eqref{tildeS} and for $s=0$.

For the operator \eqref{tildeS} and $s>0$, from \eqref{fourierS}, \eqref{fourierSnl}, \eqref{P}, \eqref{Snfourier} with \eqref{I} we have
	\begin{multline}\label{S1S2S3}
	\|{\stwosplit}(f)\|_{\fsonenu}
	= \Big\||k_{1}|^{s}\sum_{\substack{n,l\geq 0 \\ n+l\geq 2}}\frac{(-1)^n i^{l+n+1}(\ast^{l}\hat{\theta}(k_{1}))}{l!} \ast \widehat{\mathcal{S}_{n}(f)}(k_{1}) \Big\|_{\ell^{1}_{\nu}} 
	\\
	= \Big\||k_{1}|^{s}\sum_{\substack{n,l\geq 0 \\ n+l\geq 2}}\frac{(\ast^{l}\hat{\theta}(k_{1}))}{l!} \ast \sum_{k_{2},\ldots, k_{n+1}\in \mathbb{Z}}I(k_{1},\ldots,k_{n+1})\hat{f}(k_{n+1})\prod_{j=1}^{n}P(k_{j}-k_{j+1})  \Big\|_{\ell^{1}_{\nu}}
	\\
	\leq S_{1} + S_{2} + S_{3},
	\end{multline}
	where we use $a_n$ from \eqref{an.def} and \eqref{bfcn.def} so that,
	using \eqref{s.inequality} and \eqref{Ibound} with \eqref{P}, the terms $S_i$ are given by
	\begin{equation*}
	    \begin{aligned}
	S_{1}  &=  \sum_{\substack{n,l\geq 0 \\ n+l\geq 2}} \frac{{\bfcn}(l+n+1,s)}{l!}a_n l\|\theta\|_{\fzeronenu}^{l-1}\|\theta\|_{\fsonenu}\|f\|_{\fzeronenu}\|P\|_{\ell^{1}_{\nu}}^{n},\\
	S_{2} &= \sum_{\substack{n,l\geq 0 \\ n+l\geq 2}} \frac{{\bfcn}(l+n+1,s)}{l!}a_n \|\theta\|_{\fzeronenu}^{l}\|f\|_{\fsonenu}\|P\|_{\ell^{1}_{\nu}}^{n},\\
	S_{3} &= \sum_{\substack{n,l\geq 0 \\ n+l\geq 2}} \frac{{\bfcn}(l+n+1,s)}{l!}a_n n \|\theta\|_{\fzeronenu}^{l}\|f\|_{\fzeronenu}\|P\|_{_{\ell^{1}_{\nu}}}^{n-1}\|\mathcal{F}^{-1}(P)\|_{\fsonenu} .
	    \end{aligned}
	\end{equation*}
	We recall from \eqref{P} and \eqref{P0} that
	$\|P\|_{\ell^{1}_{\nu}}\le \exp{(\|\theta\|_{\fzeronenu})}-1$  and notice with \eqref{s.inequality} that for $s>0$ we have
	\begin{equation}\label{Ps.estimate}
	\|\mathcal{F}^{-1}(P)\|_{\fsonenu} \leq \sum_{m\geq 1}\frac{{\bfcn}(m,s)}{(m-1)!} \|\theta\|_{\fzeronenu}^{m-1}\|\theta\|_{\fsonenu}.
		\end{equation}
	Then, we have that
	\begin{equation*}
	S_{2}\le \tilde{C}_3\|\theta\|_{\fzeronenu}^2\|f\|_{\fsonenu},
	\end{equation*}
	where
	\begin{equation}\label{tildeC3}
	\begin{aligned}
	 \tilde{C}_3&\eqdef \sum_{\substack{n\geq 0,l\geq 0 \\ n+l\geq 2}} \frac{{\bfcn}(l+n+1,s)}{l!} a_n\|\theta\|_{\fzeronenu}^{n+l-2}\Big(\frac{e^{\|\theta\|_{\fzeronenu}}-1}{\|\theta\|_{\fzeronenu}}\Big)^n,
	\end{aligned}
	\end{equation}
	and
\begin{equation*}
S_{1}  +S_{3}\le  \tilde{C}_4\|\theta\|_{\fzeronenu}\|\theta\|_{\fsonenu}\|f\|_{\fzeronenu},
\end{equation*}
where 
\begin{equation}\label{tildeC4}
\begin{aligned}
\tilde{C}_4&\eqdef \!\!\!\sum_{\substack{n\geq 0,l\geq 0 \\ n+l\geq 2}}\!\!\!\! \frac{{\bfcn}(l\!+\!n\!+\!1,s)}{l!} a_n\|\theta\|_{\fzeronenu}^{n+l-2}\Big(l\Big(\frac{e^{\|\theta\|_{\fzeronenu}}\!-\!1}{\|\theta\|_{\fzeronenu}}\Big)^n+\! \tilde{C}_5\Big),
\end{aligned}
\end{equation}
with 
\begin{equation}\label{C5}
\tilde{C}_5\eqdef n\Big(\frac{e^{\|\theta\|_{\fzeronenu}}\!-\!1}{\|\theta\|_{\fzeronenu}}\Big)^{n-1}\sum_{m\geq 1}\frac{{\bfcn}(m,{s})}{(m-1)!} \|\theta\|_{\fzeronenu}^{m-1}.    
\end{equation}
Finally, going back to \eqref{S1S2S3}, we obtain the result for $s>0$ that
\begin{equation*}
    \|{\stwosplit}(f)\|_{\fsonenu}\leq \tilde{C}_3\|\theta\|_{\fzeronenu}^2\|f\|_{\fsonenu}+\tilde{C}_4\|\theta\|_{\fzeronenu}\|\theta\|_{\fsonenu}\|f\|_{\fzeronenu}.
\end{equation*}
This completes the desired estimate for ${\stwosplit}$ in \eqref{Sestimates} for $s>0$.

Now it only remains to bound ${\sonesplit}(f)(\alpha)$ as defined by \eqref{breveS} in \eqref{Sestimates}.  Analogously to \eqref{Snfourier} and \eqref{I}, one can obtain that
\begin{equation*}
\widehat{{\sonesplit}}(k_{1}) = \sum_{k_{2}\in\mathbb{Z}} I(k_{1},k_{1}-k_{2})\hat{f}(k_{2})P_2(k_{1}-k_{2})
\end{equation*}
where
$$ P_2(k) = \sum_{m\geq 2} \frac{i^{m}}{m!} (\ast^{m}\hat{\theta}(k))$$
and $I(k_{1},k_{1}-k_{2})$ is given by \eqref{I.function.def} with $k=k_1$ and $k_1$ in \eqref{I.function.def} replaced by $k_{1}-k_{2}$ using also \eqref{convention}, so that in particular
$$ 
I(k_{1},k_{1}-k_{2}) =  \frac{1}{\pi}\pv \int_{-\pi}^{\pi} \frac{e^{-i k_{2}\beta}}{1-e^{-i\beta}}
		\frac{1-e^{-i\beta (1+k_{1}-k_{2})}}{i(1+k_{1}-k_{2})(1-e^{-i\beta})}d\beta.
$$
Then analogously to \eqref{k1negativeone} and \eqref{I.est.rest} we have $$| I(k_{1},k_{1}-k_{2})| \le C_\mathcal{R},$$ with $C_\mathcal{R}$ from \eqref{CR}.  Note that $\widehat{{\sonesplit}}$ is the operator $\widehat{\mathcal{S}_{1}}$ in \eqref{Snfourier} with $n=1$ if you replace $P$ from \eqref{P} with $P_2$ above.  Then we have the same estimate as \eqref{Snfbound} with $P_2$ replacing $P$ and $n=1$ recalling also \eqref{an.def}.   We estimate $P_2$ similarly to \eqref{P0} (except that $m\ge 2$).  We conclude that
\begin{equation}\label{breveS0bound}
\|{\sonesplit}\|_{\fzeronenu} \leq C_\mathcal{R} \|f\|_{\fzeronenu} \|P_2\|_{\fzeronenu} 
\leq C_\mathcal{R} \breve{C}_{1}\|\theta\|_{\fzeronenu}^{2}\|f\|_{\fzeronenu}
\end{equation}
where
\begin{equation}\label{breveS1constant1}
\breve{C}_{1} = \frac{\exp(\|\theta\|_{\fzeronenu}) - 1 - \|\theta\|_{\fzeronenu}}{\|\theta\|_{\fzeronenu}^{2}}.
\end{equation}
Thus, recalling \eqref{CR}, \eqref{catalanconstant}, \eqref{tildeC1} and \eqref{breveS1constant1}, then we have
\begin{equation}\label{C1}
C_{1} = \tilde{C}_{1} + C_\mathcal{R} \breve{C}_{1}.
\end{equation}
We obtain \eqref{Sestimates} for $s=0$ by combining \eqref{breveS0bound} with the bound above \eqref{tildeC1}.

We turn to the estimate for ${\sonesplit}$ in \eqref{breveS} for $s>0$.    We will use \eqref{s.inequality} since $s>0$ and \eqref{bfcn.def}.  We will in this case estimate $P_2$ similarly to \eqref{Ps.estimate} (except with $m\ge 2$).
We then have the estimate
\begin{equation*}
 \|{\sonesplit}\|_{\fsonenu} \leq C_\mathcal{R} \sum_{m\geq 2} \frac{{\bfcn}(m+1,s)}{m!} \big(\|\theta\|_{\fzeronenu}^{m}\|f\|_{\fsonenu} + m \|\theta\|_{\fzeronenu}^{m-1}\|f\|_{\fzeronenu}\|\theta\|_{\fsonenu} \big).
\end{equation*}
Now define
\begin{equation}\label{breveCs}
\breve{C}_{3} \eqdef \sum_{m\geq 0} \frac{{\bfcn}(m+3,s)}{(m+2)!}\|\theta\|_{\fzeronenu}^{m}, \quad
\breve{C}_{4} \eqdef \sum_{m\geq 0} \frac{{\bfcn}(m+3,s)}{(m+1)!} \|\theta\|_{\fzeronenu}^{m},
\end{equation}
then we have
\begin{equation}\label{breveSsbound}
\|{\sonesplit}\|_{\fsonenu} \leq C_\mathcal{R} \breve{C}_{3}\|\theta\|_{\fzeronenu}^{2}\|f\|_{\fsonenu} + C_\mathcal{R} \breve{C}_{4} \|\theta\|_{\fzeronenu}\|f\|_{\fzeronenu}\|\theta\|_{\fsonenu}.
\end{equation}
Hence further define
\begin{equation}\label{C3}
C_{3} \eqdef \tilde{C}_{3} + C_\mathcal{R} \breve{C}_{3}, \quad C_{4} \eqdef \tilde{C}_{4} + C_\mathcal{R} \breve{C}_{4}.
\end{equation}
Then using \eqref{tildeC3}, \eqref{tildeC4} (and the estimate below them) and \eqref{breveCs} (and the estimate above it) we obtain \eqref{Sestimates} for $s>0$.
\end{proof}

In the next chapter we will prove the \textit{a priori} estimates on the vorticity strength $\omega$.

\chapter{\textit{A priori} estimates on the vorticity strength}\label{secw}

This chapter includes the \textit{a priori} estimates for the vorticity strength $\omega$ from \eqref{omega} that will be used in particular in Section \ref{subsecGlobal}.  The main result of the section is the following Proposition \ref{vorticityestimatess}.

\begin{prop}\label{vorticityestimatess}
	The linear part of the vorticity, $\omega_1$  in \eqref{omegasplit}, satisfies the following estimates:
	\begin{equation}\label{omega1f0}
    \|\omega_1\|_{\fzeronenu}\leq A_\sigma\frac{4\pi}{L(t)}\thetatwoonenu+(1+2|A_\mu|)|A_\rho|\frac{L(t)}{\pi}e^{\nu(t)}\thetazeronenu,
    \end{equation}
    for $s>0$, recalling \eqref{bfcn.def}, we have
    \begin{equation}\label{omega1fss}
    \|\omega_1\|_{\fsonenu}
    \leq A_\sigma\frac{4\pi}{L(t)}\thetatwosonenu+(1+2|A_\mu|)|A_\rho|\frac{L(t)}{\pi}2\bfcn(2,s)e^{\nu(t)}\thetasonenu.
    \end{equation}
The nonlinear part  $\omega_{\geq2}$ from \eqref{omegasplit} satisfies
\begin{equation}\label{omega2f0}
	\|\omega_{\geq2}\|_{\fzerone_{\nu}}
	\leq |A_\mu|A_\sigma\frac{4\pi}{L(t)}C_9\thetazeronenu\thetatwoonenu
	+|A_\rho| \frac{L(t)}{\pi}e^{\nu(t)}C_{10}\thetazeronenu^2,
\end{equation}
while $s>0$ we have
\begin{multline}\label{omega2fss}
	\|\omega_{\geq2}\|_{\fsonenu}
	\!\leq\!
	|A_\mu|A_\sigma\frac{4\pi}{L(t)}C_{13}\thetazeronenu\thetatwosonenu\!
	\\
+
	\!|A_\rho| \frac{L(t)}{\pi}e^{\nu(t)}C_{14}\thetazeronenu\thetasonenu,
\end{multline}
where $C_9$ and $C_{10}$ are defined in \eqref{C9C10} and $C_{13}$ and $C_{14}$ are in \eqref{C13C14}.
\end{prop}

\begin{proof}
First, recalling \eqref{bfcn.def}, we note that 
\begin{equation}\label{cosine.calc.FT}
\begin{aligned}
    \|\cos{(\alpha\!+\!\hat{\vartheta}(0))}\theta(\alpha)\|_{\fzeronenu}&\leq e^{\nu(t)}\thetazeronenu,\\
    \|\cos{(\alpha\!+\!\hat{\vartheta}(0))}\theta(\alpha)\|_{\fsonenu}&\leq 2 \bfcn(2,s) e^{\nu(t)}\thetasonenu,\qquad s>0,
\end{aligned}    
\end{equation}
as similar to the calculations in the proof of Lemma \ref{triangleprop}.  We point out that the same calculations also hold with $\cos{(\alpha\!+\!\hat{\vartheta}(0))}$ replaced by $\sin{(\alpha\!+\!\hat{\vartheta}(0))}$.  
For example, to see the above in the case $0 < s \leq 1$, 
we have 
\begin{multline}\label{proof.angle.nu.ineq}
\|\cos{(\alpha\!+\!\widehat{\vartheta}(0))}\theta(\alpha)\|_{\fsonenu} 
= \sum_{k\in\mathbb{Z}} e^{\nu(t)|k|}|k|^{s}|((\cos{(\alpha\!+\!\widehat{\vartheta}(0))})^{\wedge}\ast \widehat{\theta}) (k)|
\\
= \sum_{k, k_{1}\in\mathbb{Z}} e^{\nu(t)|k|} |k|^{s} |\frac{1}{2}(\delta(1-k_{1})
e^{i\widehat{\vartheta}(0) }
+ \delta(1+k_{1})
e^{-i\widehat{\vartheta}(0)}
) \widehat{\theta} (k-k_{1})|\\
\leq  \sum_{k\in\mathbb{Z}} e^{\nu(t)|k|}|k|^{s}  \frac{1}{2} (|\widehat{\theta} (k-1)|+|\widehat{\theta} (k+1)|)\\
\leq  \frac{1}{2}\sum_{k\in\mathbb{Z}} e^{\nu(t)|k-1|}e^{\nu(t)}(|k-1|^{s}+1)   |\widehat{\theta} (k-1)|\\ \hspace{0.5in}+e^{\nu(t)|k+1|}e^{\nu(t)}(|k+1|^{s}+1) |\widehat{\theta} (k+1)|\\
= e^{\nu(t)}(\thetazeronenu + \thetasonenu) \leq 2e^{\nu(t)}\thetasonenu.
\end{multline}
This explains the difference between \eqref{omega1f0} and \eqref{omega1fss}.
We will explain the proof of \eqref{omega1fss} for $s>0$.  The proof of the other case, \eqref{omega1f0} when $s=0$, is analogous.  It follows from \eqref{omegasplit} and \eqref{cosine.calc.FT} that
\begin{equation*}
\|\omega_1\|_{\fsonenu}
\leq 
|A_\mu|\frac{L(t)}{\pi}\|\mathcal{D}_1(\omega_0)\|_{\fsonenu}
\!+\!
A_\sigma\frac{4\pi}{L(t)}\|\theta\|_{\dot{\mathcal{F}}^{2+s,1}_\nu}
\!+\!
|A_\rho| \frac{L(t)}{\pi}2\bfcn(2,s) e^{\nu(t)}\thetasonenu,
\end{equation*}
and from \eqref{mdsplit} with \eqref{cosine.calc.FT} and \eqref{hilbertTcalc} we have
\begin{equation*}
\begin{aligned}
\|\mathcal{D}_1(\omega_0)\|_{\fsonenu}\leq\frac{\pi}{L(t)}\Big(|A_\rho|\frac{L(t)}{\pi}2 \bfcn(2,s) e^{\nu(t)}\thetasonenu+\|\imag\hspace{0.05cm} \mathcal{R}(\omega_0)\|_{\fsonenu}\Big). 
\end{aligned}
\end{equation*}
Recalling  \eqref{fourierimag} together with Lemma  \eqref{cosine.calc.FT} gives the following estimate 
\begin{equation*}
    \|\imag\hspace{0.05cm} \mathcal{R}(\omega_0)\|_{\fsonenu}\leq |A_\rho|\frac{L(t)}{\pi}2 \bfcn(2,s) e^{\nu(t)}\thetasonenu.
\end{equation*}
Therefore,  
\begin{equation*}
\begin{aligned}
\|\mathcal{D}_1(\omega_0)\|_{\fsonenu}\leq 4|A_\rho|e^{\nu(t)}\bfcn(2,s) \thetasonenu, 
\end{aligned}
\end{equation*}
and thus the estimate \eqref{omega1fss} for $\omega_1$ is complete.  The estimate for \eqref{omega1f0} is proven in the same way.

Now we proceed to bound $\omega_{\geq 2}$ from \eqref{omegasplit} in $\fzeronenu$.  In this case we have
\begin{equation}\label{omega2aux}
\begin{aligned}
\|\omega_{\geq2}\|_{\fzeronenu}
&\leq 
|A_\mu|\frac{L(t)}{\pi}\|\mathcal{D}_{\geq2}(\omega)\|_{\fzeronenu}
+|A_\rho| \frac{L(t)}{\pi}e^{\nu(t)}\sum_{j\geq2}\frac{\|\theta\|_{\fzerone_{\nu}}^j}{j!}
\\
&=
|A_\mu|\frac{L(t)}{\pi}\|\mathcal{D}_{\geq2}(\omega)\|_{\fzeronenu}
+
|A_\rho| \frac{L(t)}{\pi}e^{\nu(t)}C_6\thetazeronenu^2,
\end{aligned}
\end{equation}
where
\begin{equation}\label{C6}
C_6=\frac{e^{\thetazeronenu}-1-\thetazeronenu}{\thetazeronenu^2}.
\end{equation}
From \eqref{mdsplit}, one has that
\begin{equation*}
\begin{aligned}
\|\mathcal{D}_{\geq2}\|_{\fzeronenu}&\leq \frac{\pi}{L(t)}\Big(\thetazeronenu\|\omega_{\geq1}\|_{\fzeronenu}+\|\mathcal{R}(\omega_{\geq1})\|_{\fzeronenu}+\|\mathcal{S}(\omega)\|_{\fzeronenu}\Big).
\end{aligned}
\end{equation*}
Using the estimates \eqref{Restimates} and \eqref{Sestimates} and splitting the vorticity terms as $\omega_{\geq1} = \omega_{1}+\omega_{\geq2}$, we obtain the estimate
\begin{equation*}
\begin{aligned}
\|\mathcal{D}_{\geq2}&\|_{\fzeronenu}\leq \frac{\pi}{L(t)}\Big(\thetazeronenu\|\omega_{1}\|_{\fzeronenu}+\thetazeronenu\|\omega_{\geq2}\|_{\fzeronenu}+
{\crconstant}\thetazeronenu\|\omega_{1}\|_{\fzeronenu}\\
&+{\crconstant}\thetazeronenu\|\omega_{\geq2}\|_{\fzeronenu}+|A_\rho|\frac{L(t)}{\pi}e^{\nu(t)}C_1\thetazeronenu^2+C_1\thetazeronenu^2\|\omega_{1}\|_{\fzeronenu}\\
&+C_1\thetazeronenu^2\|\omega_{\geq2}\|_{\fzeronenu}
\Big),
\end{aligned}
\end{equation*}
so, substituting back in \eqref{omega2aux} and solving for $\|\omega_{\geq2}\|_{\fzeronenu}$, we obtain that
\begin{equation*}
    \begin{aligned}
\|\omega_{\geq2}\|_{\fzeronenu}&\leq C_8\Big( |A_\mu|C_7\thetazeronenu\|\omega_{1}\|_{\fzeronenu}+|A_\mu||A_\rho|\frac{L(t)}{\pi}e^{\nu(t)}C_1\thetazeronenu^2\\
&\quad\quad+|A_\rho|\frac{L(t)}{\pi}e^{\nu(t)}C_6\thetazeronenu^2\Big),
    \end{aligned}
\end{equation*}
where we defined
\begin{equation}\label{C7C8}
    C_7=1+{\crconstant}+C_1\thetazeronenu,\qquad C_8=\frac{1}{1-|A_\mu|C_7\thetazeronenu},
\end{equation}
with $C_1$ given in \eqref{C1}.
Substituting in \eqref{omega1f0}  we conclude that
\begin{equation*}
    \begin{aligned}
    \|\omega_{\geq2}\|_{\fzeronenu}&\leq 2|A_\mu|A_\sigma \frac{2\pi}{L(t)}C_7 C_8\thetazeronenu\thetatwoonenu\\
    &+|A_\rho|\frac{L(t)}{\pi}e^{\nu(t)}C_8\Big(|A_\mu|(1+2|A_\mu|)C_7+|A_\mu|C_1+C_6\Big)\thetazeronenu^2,
    \end{aligned}
\end{equation*}
which gives the  estimate  \eqref{omega2f0} by defining
\begin{equation}\label{C9C10}
\begin{aligned}
C_9&=C_7C_8,\\
C_{10}&=|A_\mu|(1+2|A_\mu|)C_7C_8+|A_\mu|C_1C_8+C_6C_8,
\end{aligned}
\end{equation}
where $C_1$, $C_6$, $C_7$, and $C_8$ were previously defined in \eqref{C1}, \eqref{C6}, and \eqref{C7C8}.

We consider now the case $s>0$ in \eqref{omega2fss}. From \eqref{omegasplit}, also using \eqref{proof.angle.nu.ineq}, we have
\begin{multline}\label{omega2aux2}
\|\omega_{\geq2}\|_{\fsonenu}
\leq |A_\mu|\frac{L(t)}{\pi}\|\mathcal{D}_{\geq2}(\omega)\|_{\fsonenu}
+
|A_\rho| \frac{L(t)}{\pi}e^{\nu(t)}\sum_{j\geq2}\frac{\bfcn(j,s) \|\theta\|_{\fzerone_{\nu}}^{j}}{j!}
\\
\quad
+|A_\rho| \frac{L(t)}{\pi}e^{\nu(t)}\sum_{j\geq1}\frac{\bfcn(j,s) \|\theta\|_{\fzerone_{\nu}}^{j}}{j!}\thetasonenu
\\
\leq
|A_\mu|\frac{L(t)}{\pi}\|\mathcal{D}_{\geq2}(\omega)\|_{\fsonenu}
+
|A_\rho| \frac{L(t)}{\pi}e^{\nu(t)}C_{11}\thetazeronenu\thetasonenu,
\end{multline}
where 
\begin{equation}\label{C11}
C_{11}=\sum_{j\geq 2}\frac{\bfcn(j,s) \|\theta\|_{\fzerone_{\nu}}^{j-2}}{j!}+\sum_{j\geq1}\frac{\bfcn(j,s) \|\theta\|_{\fzerone_{\nu}}^{j-1}}{j!}.
\end{equation}
Recalling \eqref{mdsplit}, and \eqref{s.inequality}, one has that
\begin{equation*}
\begin{aligned}
\|\mathcal{D}_{\geq2}\|_{\fsonenu}&\leq \frac{\pi}{L(t)}\Big(\bfcn(2,s) \thetasonenu\|\omega_{\geq1}\|_{\fzeronenu}+ \bfcn(2,s)\thetazeronenu\|\omega_{\geq1}\|_{\fsonenu}
\\
&\qquad\qquad+\|\mathcal{R}(\omega_{\geq1})\|_{\fsonenu}+\|\mathcal{S}(\omega)\|_{\fsonenu}\Big).
\end{aligned}
\end{equation*}
Using the estimates \eqref{Restimates} and \eqref{Sestimates} with $s>0$ and splitting the vorticity terms again, we obtain
\begin{multline*}
\|\mathcal{D}_{\geq2}\|_{\fsonenu}\leq \frac{\pi}{L(t)}\bigg((1+b(2,s){\crconstant})\thetasonenu\|\omega_{1}\|_{\fzeronenu}+(1+b(2,s){\crconstant})\thetasonenu\|\omega_{\geq2}\|_{\fzeronenu}
\\
+
(1+b(2,s){\crconstant})\thetazeronenu\|\omega_{1}\|_{\fsonenu}+(1+b(2,s){\crconstant})\thetazeronenu\|\omega_{\geq2}\|_{\fsonenu}
\\
+
|A_\rho|\frac{L(t)}{\pi}e^{\nu(t)}(C_3+C_4)\thetazeronenu\thetasonenu+C_3\thetazeronenu^2\|\omega_{1}\|_{\fsonenu}
\\
+\!C_3\thetazeronenu^2\|\omega_{\geq2}\|_{\fsonenu}\!+\!C_4\thetazeronenu\thetasonenu\|\omega_{1}\|_{\fzeronenu}\!+\!C_4\thetazeronenu\thetasonenu\|\omega_{\geq2}\|_{\fzeronenu}\!\!\bigg),
\end{multline*}
which becomes 
\begin{equation*}
\begin{aligned}
\|\mathcal{D}_{\geq2}\|_{\fsonenu}&\leq \frac{\pi}{L(t)}\Big(|A_\rho|\frac{L(t)}{\pi}e^{\nu(t)}(C_3\!+\!C_4)\thetazeronenu\thetasonenu\!+\!C_{12}\thetasonenu\|\omega_{1}\|_{\fzeronenu}\\
&+\!C_2\thetazeronenu\|\omega_{1}\|_{\fsonenu}\!+\!C_{12}\thetasonenu\|\omega_{\geq2}\|_{\fzeronenu}\!+\!C_2\thetazeronenu\|\omega_{\geq2}\|_{\fsonenu}\Big),
\end{aligned}
\end{equation*}
where 
\begin{equation}\label{C12}
\begin{aligned}
C_2&=1+\bfcn(2,s){\crconstant}+C_3\thetazeronenu\\
C_{12}&=1+\bfcn(2,s){\crconstant}+C_4\thetazeronenu,
\end{aligned}
\end{equation}
with $C_\mathcal{R}$ and $C_4$ given in \eqref{CR} and \eqref{C3}, respectively.
Substituting back in \eqref{omega2aux2}, and solving for $\|\omega_{\geq2}\|_{\fsonenu}$, we have
\begin{equation*}
    \begin{aligned}
\|\omega_{\geq2}\|_{\fsonenu}&\leq |A_\mu|\tilde{C}_8 \Big( 
|A_\rho|\frac{L(t)}{\pi}e^{\nu(t)}(C_3+C_4)\thetazeronenu\thetasonenu\\
&+
C_{12}\thetasonenu\|\omega_{1}\|_{\fzeronenu}+C_2\thetazeronenu\|\omega_{1}\|_{\fsonenu}+C_{12}\thetasonenu\|\omega_{\geq2}\|_{\fzeronenu}\Big)\\
&+|A_\rho|\frac{L(t)}{\pi}e^{\nu(t)}C_{11}\tilde{C}_8\thetazeronenu\thetasonenu,
    \end{aligned}
\end{equation*}
where we defined
\begin{equation}\label{tildeC8}
\tilde{C}_8=\frac{1}{1-|A_\mu|C_2\thetazeronenu}.
\end{equation}
We introduce the estimates \eqref{omega2f0}, \eqref{omega1f0}, and \eqref{omega1fss} to obtain that
\begin{equation*}
    \begin{aligned}
\|\omega_{\geq2}&\|_{\fsonenu}\leq |A_\mu|\tilde{C}_8\bigg( 
|A_\rho|\frac{L(t)}{\pi}e^{\nu(t)}(C_3+C_4)\thetazeronenu\thetasonenu\\
&\hspace{-0.7cm}+
A_\sigma \frac{4\pi}{L(t)}C_{12}\thetasonenu\thetatwoonenu+(1+2|A_\mu|)|A_\rho|\frac{L(t)}{\pi}e^{\nu(t)}C_{12}\thetasonenu\thetazeronenu\\
&\hspace{-0.7cm}+A_\sigma\frac{4\pi}{L(t)} C_2\thetazeronenu\|\theta\|_{\dot{\mathcal{F}}^{2+s,1}}+(1\!+\!2|A_\mu|)|A_\rho|\frac{L(t)}{\pi}2\bfcn(2,s) e^{\nu(t)}C_2\thetazeronenu\thetasonenu\\
&\hspace{-0.7cm}+|A_\mu|A_\sigma\frac{4\pi}{L(t)}C_9C_{12}\thetazeronenu\thetasonenu\thetatwoonenu\\
&\hspace{-0.7cm}+|A_\rho|\frac{L(t)}{4\pi}e^{\nu(t)}C_{10}C_{12}\thetasonenu\thetazeronenu^2\!
\bigg)\!\!+\!|A_\rho|\frac{L(t)}{\pi}e^{\nu(t)}C_{11}\tilde{C}_8\thetazeronenu\thetasonenu.
    \end{aligned}
\end{equation*}
Next, we use the interpolation inequality \eqref{interpolation} to find that
\begin{equation}\label{interp}
    \thetasonenu\thetatwoonenu\leq \thetazeronenu\thetatwosonenu,
\end{equation}
and therefore
\begin{equation*}
\begin{aligned}
    \|\omega_{\geq2}\|_{\fsonenu}&\leq |A_\mu|A_\sigma \frac{4\pi}{L(t)}\tilde{C}_8\Big(C_2+C_{12}+|A_\mu|C_9C_{12}\thetazeronenu\Big)\thetazeronenu\thetatwosonenu\\
    &+|A_\rho|\frac{L(t)}{\pi}e^{\nu(t)}\tilde{C}_8\Big(|A_\mu|(C_3+C_4)+|A_\mu|(1+2|A_\mu|)(2C_2+C_{12})\\
    &+|A_\mu|C_{10}C_{12}\thetazeronenu+C_{11}\Big)\thetazeronenu\thetasonenu,
\end{aligned}
\end{equation*}
which gives the inequality \eqref{omega2fss} by defining
\begin{equation}\label{C13C14}
\begin{aligned}
C_{13}&=\tilde{C}_8\Big(C_2+C_{12}+|A_\mu|C_9C_{12}\thetazeronenu\Big),\\
C_{14}&=|A_\mu|\tilde{C}_8\Big(C_3\!+\!C_4\!+\!(1\!+\!2|A_\mu|)(2C_2\!+\!C_{12})\!+\!C_{10}C_{12}\thetazeronenu\Big)\!
\\
& \hspace{1cm}+\!C_{11}\tilde{C}_8,
\end{aligned}
\end{equation}
where $C_3$ and $C_4$ in \eqref{C3}, $C_2$ in \eqref{C12}, $\tilde{C}_8$ in \eqref{tildeC8}, $C_9$ and $C_{10}$ in\eqref{C9C10}, $C_{11}$ in \eqref{C11}, and $C_{12}$ in \eqref{C12} were previously defined.
\end{proof}

In the next chapter we put together all the previous results to prove the global existence and instant analyticity as in our main Theorem \ref{thm:global}.

\chapter{Global existence and instant analyticity}\label{secanalytic}

This chapter is dedicated to the proof of Theorem \ref{thm:global}. In Section \ref{sec:Lestimate} we prove the claimed bound for the length $L(t)$ from \eqref{Lfinalbound}.  Then Section \ref{subsecGlobal} proves the main \textit{a priori}  estimates for a solution $\theta$.  We prove the global energy inequality from \eqref{estimatef12}.  In particular we diagonalize the linearized operator in Proposition \ref{diagonalization}.  We prove operator bounds for the corresponding linear transformations in Lemma \ref{lemmaS}.   In Section \ref{sec:NonLinearEst} we prove the corresponding non-linear estimates that were used in the a priori estimates from the previous section.  In particular we prove Theorem \ref{thm:Nbound}.  Lastly, in Section \ref{subregular} we describe the scheme of the proof of our main theorem via a regularization argument.

\section{Estimate for $L(t)$}\label{sec:Lestimate}

In this section, we prove 
the bound for $L(t)$ from \eqref{Lfinalbound}.  Equation \eqref{Lequation}, together with the bound
\begin{equation}\label{lengthauxbound}
\begin{aligned}
    \imag\Big(\int_{-\pi}^\pi\int_0^\alpha e^{i(\alpha-\eta)} \sum_{n\geq1}\frac{i^n}{n!}(\theta(\alpha)-\theta(\eta))^n d\eta d\alpha\Big)&\leq \pi^2\sum_{n\geq1}\frac{2^n\|\theta\|_{\fzerone}^n}{n!}\\
    &\leq \pi^2\big(e^{2\thetazerone}-1\big),
\end{aligned}
\end{equation}
implies that
\begin{equation}\label{Lboundaux}
    \begin{aligned}
        \frac{R^2}{1+\frac{\pi}{2} \big(e^{2\thetazerone}-1\big)}\leq\Big(\frac{L(t)}{2\pi}\Big)^2\leq \frac{R^2}{1-\frac{\pi}{2}\big(e^{2\thetazerone}-1\big)},
    \end{aligned}
\end{equation}
and therefore
\begin{equation}\label{Lbound}
    \begin{aligned}
        2\pi R C_{37}\leq L(t)\leq 2\pi RC_{38},
    \end{aligned}
\end{equation}
where
\begin{equation}\label{C37C38}
    C_{37}=\frac{1}{\sqrt{1+\frac{\pi}{2} \big(e^{2\thetazerone}-1\big)}},\quad C_{38}=\frac{1}{\sqrt{1-\frac{\pi}{2}\big(e^{2\thetazerone}-1\big)}}.
\end{equation}
We have thus shown that the length of the curve is controlled  for all times. In particular, we also have the following estimates
\begin{equation}\label{Lestimates}
\begin{aligned}
    \Big|R^2\Big(\frac{2\pi}{L(t)}\Big)^2-1\Big|&\leq  \frac{\pi}{2}\big(e^{2\thetazerone}-1\big),
    \\
    \Big|R\Big(\frac{2\pi}{L(t)}\Big)-1\Big|&\leq  \sqrt{1+\frac{\pi}{2}\big(e^{2\thetazerone}-1\big)}-1=C_{39}\thetazerone,
    \\
    \Big|R^3\Big(\frac{2\pi}{L(t)}\Big)^3-1\Big|&\leq  \Big(1+\frac{\pi}{2}\big(e^{2\thetazerone}-1\big)\Big)^{3/2}-1=C_{40}\thetazerone,
    \end{aligned}
\end{equation}
with
\begin{equation}\label{C39C40}
    C_{39}=\frac{\sqrt{1+\frac{\pi}{2}\big(e^{2\thetazerone}-1\big)}-1}{\thetazerone},\quad C_{40}=\frac{\Big(1+\frac{\pi}{2}\big(e^{2\thetazerone}-1\big)\Big)^{3/2}-1}{\thetazerone}.
\end{equation}
These are the main estimates that we will use for $L(t)$.  We remark that all the estimates in this chapter hold with the same proof in the norms $\fzeronenu$ and $\fhonenu$.   In the next section we prove the main a priori estimates for $\theta$.

\section{\textit{A priori} Estimates for $\theta(t)$}\label{subsecGlobal}

In this section we obtain the \textit{a priori} estimates that guarantee the global existence of the solutions, the instant in time analyticity, and the exponentially fast convergence to a circle. In particular, the main goal is to show the energy inequality \eqref{estimatef12}.

The results of this section are ordered as follows: First, we write the system \eqref{system} using the linearization \eqref{linearfourier} of Section \ref{sec:linearization}; then, we diagonalize the system according to Proposition \ref{diagonalization} to show its dissipative character; the bounds for this change of frame are proven in Lemma \ref{lemmaS}; the estimates for the nonlinear terms are proven in Section \ref{sec:NonLinearEst}, together with the control of the length in Section \ref{sec:Lestimate}.

\begin{proof}[Proof of the global energy inequality from  \eqref{estimatef12}]
Equation \eqref{system} and  Proposition \ref{linearfourier} show that the equation for $\widehat{\theta}(k)$, $1\leq k\neq2$, is given by
\begin{equation*}
    \begin{aligned}
   \widehat{\theta}_t(k)&=-A_\sigma \Big(\frac{2\pi}{L(t)}\Big)^3k(k^2\!-\!1)\widehat{\theta}(k)\!-\!\big(1\!+\!A_\mu\big)A_\rho\frac{2\pi}{L(t)}\frac{(k^2\!-\!1)(k\!+\!1)}{k(k\!+\!2)}e^{-i\widehat{\vartheta}(0)}\widehat{\theta}(k\!+\!1)\\
   &\quad+\frac{2\pi}{L(t)}\widehat{N}(k),
    \end{aligned}
\end{equation*}
and for $k=2$, we have
\begin{equation*}
    \begin{aligned}
   \widehat{\theta}_t(2)&=-A_\sigma \Big(\frac{2\pi}{L(t)}\Big)^36\widehat{\theta}(2)\!-(1+A_\mu)A_\rho\frac{2\pi}{L(t)}\frac{9}{8}e^{-i\widehat{\vartheta}(0)}\widehat{\theta}(3)\\
   &\quad +(1-A_\mu)A_\rho\frac{2\pi}{L(t)}\frac{3}{2}\left(\frac34-\log{2}\right)e^{i\widehat{\vartheta}(0)}\widehat{\theta}(1) +\frac{2\pi}{L(t)}\widehat{N}(k).
    \end{aligned}
\end{equation*}
We notice in the equation above that the terms that are linear in $\widehat{\theta}(k)$ have time-dependent coefficients. However, this dependency happens only through $L(t)$. We will show that $L(t)$ is bounded from below and above (see Section \ref{sec:Lestimate}). In fact, it is not hard to see from \eqref{Lequation} that, to leading order, $L(t)$ equals $2\pi R$. Thus we rewrite the equation above as follows
\begin{equation}\label{thetaequation}
    \begin{aligned}
   \widehat{\theta}_t(k)&=-\frac{A_\sigma}{R^3} k(k^2-1)\widehat{\theta}(k)-\big(1+A_\mu\big)\frac{A_\rho}{R}\frac{(k^2-1)(k+1)}{k(k+2)}e^{-i\widehat{\vartheta}(0)}\widehat{\theta}(k+1)\\
   &\quad+\frac{2\pi}{L(t)}\widehat{N}(k)-\frac{A_\sigma}{R^3} k(k^2-1)\widehat{\theta}(k)\Big(R^3\Big(\frac{2\pi}{L(t)}\Big)^3-1\Big)\\
   &\quad-\big(1+A_\mu\big)\frac{A_\rho}{R}\frac{(k^2-1)(k+1)}{k(k+2)}e^{-i\widehat{\vartheta}(0)}\widehat{\theta}(k+1)\Big(R\frac{2\pi}{L(t)}-1\Big),
    \end{aligned}
\end{equation}
for $1\leq k\neq2$, and we decompose the equation analogously for $k=2$.

Next, we write the corresponding linear system as follows
\begin{equation}\label{linearZ}
    \begin{aligned}
    z_t(k)=\sum_{j\geq1}M_{k,j}z(j),\quad k\geq1,
    \end{aligned}
\end{equation}
where we denote
\begin{equation*}
    M_{k,j}=\left\{
    \begin{aligned}
    &-a(k),\quad j=k,\\
    &b(k),\hspace{0.8cm} j=k+1,\\
    &c(1),\hspace{0.8cm}j=1,k=2,\\
    &0,\hspace{1.2cm} \text{otherwise},
    \end{aligned}\right.
\end{equation*}
with 
\begin{equation}\label{coeffs}
    a(k)=\frac{A_\sigma}{R^3} k(k^2-1),\quad b(k)=-\big(1+A_\mu\big)\frac{A_\rho}{R}\frac{(k^2-1)(k+1)}{k(k+2)}e^{-i\widehat{\vartheta}(0)},
\end{equation}
\begin{equation*}
    c(1)=(1-A_\mu)\frac{A_\rho}{R}\frac{3}{2}\left(\frac34-\log{2}\right)e^{i\widehat{\vartheta}(0)}.
\end{equation*}
Notice that this is an upper triangular system except for the entry $k=2$,$j=1$.  This entry, $k=2$ with $j=1$, will require special attention.  The eigenvalues of this system are $-a(k)$.  This is given in the following proposition.

\begin{prop}[Diagonalization]\label{diagonalization}
Consider $z=(z(k))_{k\geq1},y=(y(k))_{k\geq1}\in \ell^1$ and the linear operator \begin{equation*}
    \begin{aligned}
    S^{-1}:\ell^1&\rightarrow \ell^1\\
    z&\mapsto y=S^{-1}z,
    \end{aligned}
\end{equation*}
defined by
    \begin{equation*}
        y(k)=\sum_{j\geq1}S^{-1}_{k,j}z(j),
    \end{equation*}
    with
    \begin{equation}\label{Sm1}
        S^{-1}_{k,j}=\left\{
        \begin{aligned}
        (&-1)^{j-k}\prod_{l=1}^{j-k} \frac{b(k-1+l)}{a(k)-a(k+l)},\hspace{0.4cm}j\geq k\geq2,\\
        &1,\hspace{1.2cm}j=k=1,\\
        -&\frac{c(1)}{a(2)},\hspace{1.2cm}k=2,j=1,\\
        &0,\hspace{1.72cm} \text{otherwise}.
        \end{aligned}\right.
    \end{equation}
Then, the inverse operator $S$ is given by
    \begin{equation*}
        S_{k,j}=\left\{
        \begin{aligned}
        &\prod_{l=1}^{j-k} \frac{b(k-1+l)}{a(k-1+l)-a(j)},\hspace{0.4cm}j\geq k\geq2,\\
        &1,\hspace{1.2cm}j=k=1,\\
        &\frac{c(1)}{a(2)},\hspace{1.72cm}k=2,j=1,\\
        &0,\hspace{1.72cm} \text{otherwise}.
        \end{aligned}\right.
    \end{equation*}
Moreover, the linear operator $S^{-1}$ diagonalizes the system \eqref{linearZ} as follows
\begin{equation*}
    y_t(k)=-a(k)y(k),\quad k\geq1.
\end{equation*}
\end{prop}

We remark that since $\prod_{l=1}^{0}=1$ by definition then $S_{k,k} = S^{-1}_{k,k} =1$ when $j=k$ for all $k\ge 1$.   We also have the following lemma which gives uniform bounds for the operators $S$ and $S^{-1}$.

\begin{lemma}\label{lemmaS}
    The operator norms in $\ell^1$ of the linear operators $S$ and $S^{-1}$ satisfy the following bounds
    \begin{equation}\label{CSbounds}
        \begin{aligned}
        \|S^{-1}\|_{\ell^1\rightarrow \ell^1}\leq C_{S},\qquad  \|S\|_{\ell^1\rightarrow \ell^1}\leq C_{S},
        \end{aligned}
    \end{equation}
with $C_S = C_S(A_\mu,\frac{|A_\rho| R^2}{A_\sigma})$ where
\begin{equation*}
    C_S
    \eqdef \max\Big\{1+\frac{1}{4}\big(1\!-\!A_\mu\big)\frac{|A_\rho| R^2}{A_\sigma}\Big(\frac34-\log{2}\Big),6\frac{I_3\Big(2\sqrt{\big(1\!+\!A_\mu\big)\frac{|A_\rho| R^2}{A_\sigma}}\Big)}{\Big(\big(1\!+\!A_\mu\big)\frac{|A_\rho| R^2}{A_\sigma}\Big)^{3/2}}\Big\}.
\end{equation*}
\end{lemma}

In the constant $C_S$ above we used the modified Bessel function of the first kind of order three.  In general, for an integer $n\ge 0$, we define
\begin{equation}\label{bessel.In.def}
I_n(z) \eqdef \left(\frac{z}{2} \right)^n
 \sum_{j=0}^{\infty}
 \frac{(z/2)^{2j}}{j!(j+n)!}.
\end{equation}
We will prove Proposition \ref{diagonalization} and Lemma \ref{lemmaS} after we finish the proof of the main global energy inequality.

\begin{remark}\label{analytic.norm.remark.S}
The results in Proposition \ref{diagonalization} and Lemma \ref{lemmaS} also hold in the space $\ell^1$ with weight $e^{\nu(t)|k|}|k|^s$ for any $s\ge 0$ without any change to the proof; i.e.
\begin{equation*}
    \|z\|_{\fsonenu}
    \leq 
     C_S\|y\|_{\fsonenu},
     \quad 
         \|y\|_{\fsonenu}
    \leq 
     C_S\|z\|_{\fsonenu},
\end{equation*}
where $y(k)$ and $z(k)$ are defined in Proposition \ref{diagonalization}.
\end{remark}

According to Proposition \ref{diagonalization}, we apply the linear transformation $S^{-1}$ to \eqref{thetaequation} to rewrite the system with the linear part in diagonal form
\begin{equation}\label{thetaeqdiag}
    \begin{aligned}
   (S^{-1}\widehat{\theta})_t(k)&=-\frac{A_\sigma}{R^3} k(k^2-1)(S^{-1}\widehat{\theta})(k)+\frac{2\pi}{L(t)}(S^{-1}\widehat{N})(k)\\
   &\quad-\frac{A_\sigma}{R^3} (S^{-1}u)(k)\Big(R^3\Big(\frac{2\pi}{L(t)}\Big)^3-1\Big)\\
   &\quad-\big(1+A_\mu\big)\frac{A_\rho}{R}e^{-i\widehat{\vartheta}(0)}(S^{-1}v)(k)\Big(R\frac{2\pi}{L(t)}-1\Big)\\
   &\quad+(1-A_\mu)\frac{A_\rho}{R}\frac{3}{2}\left(\frac34-\log{2}\right)e^{i\widehat{\vartheta}(0)}\Big(R\frac{2\pi}{L(t)}-1\Big)(S^{-1}w)(k),
    \end{aligned}
\end{equation}
where we introduced the notation
\begin{equation*}
    \begin{aligned}
    u(k)=k(k^2-1)\widehat{\theta}(k),\quad v(k)=\frac{(k^2-1)(k+1)}{k(k+2)}\widehat{\theta}(k+1),\quad w(k)=1_{k=2}\widehat{\theta}(1).
    \end{aligned}
\end{equation*}
Taking into account that $\theta(\alpha)$ is a real-valued function, we can write the norm \eqref{fsonenorm} in terms of the positive frequencies alone
\begin{equation*}
\begin{aligned}
    \|\theta\|_{\fhonenu}=\sum_{k\in\mathbb{Z}}e^{\nu(t)|k|}|k|^{1/2}|\widehat{\theta}(k)|=2\sum_{k\geq1}e^{\nu(t)k}k^{1/2}|\widehat{\theta}(k)|.
\end{aligned}
\end{equation*}
Define
\begin{equation*}
    \begin{aligned}
    \widehat{y}(k)=(S^{-1}\widehat{\theta})(k),\quad \widehat{y}(-k)=\overline{\widehat{y}(k)},\quad k\geq1, \quad \widehat{y}(0)=0,\quad y=\mathcal{F}^{-1}\widehat{y},
    \end{aligned}
\end{equation*}
and consider the evolution of the quantity
\begin{equation*}
    \|y\|_{\fhonenu}=2\sum_{k\geq1} e^{\nu(t)k}k^{1/2}|\widehat{y}(k)|.
\end{equation*}
Taking the derivative in time we obtain that
\begin{equation*}
\begin{aligned}
    \frac{d}{dt}\|y\|_{\fhonenu}\!=\!2\!\sum_{k\geq1}\!\nu'(t) k^{3/2}e^{\nu(t)k}|\widehat{y}(k)|\!+\!2\sum_{k\geq1}e^{\nu(t)k}k^{1/2}\frac12\frac{\widehat{y}_t(k)\overline{\widehat{y}(k)}\!+\!\widehat{y}(k)\overline{\widehat{y}_t(k)}}{|\widehat{y}(k)|}.
\end{aligned}
\end{equation*}
Therefore, substituting \eqref{thetaeqdiag}, one finds the following equation
\begin{equation}\label{auxx}
\begin{aligned}
    \frac{d}{dt}\|y\|_{\fhonenu}\!&=\!2\!\sum_{k\geq1}\!\nu'(t) k^{3/2}e^{\nu(t)k}|\widehat{y}(k)|-2\frac{A_\sigma}{R^3}\!\!\sum_{k\geq1}\!e^{\nu(t)k}k^{3/2}(k^2-1)|\widehat{y}(k)|
    \\
    &\quad+2\frac{2\pi}{L(t)}\sum_{k\geq1}e^{\nu(t)k}k^{1/2}\frac12\frac{(S^{-1}\widehat{N})(k)\overline{\widehat{y}(k)}+\overline{(S^{-1}\widehat{N})(k)}\widehat{y}(k)}{|\widehat{y}(k)|}\\
    &\quad+Y_u+Y_v+Y_w,
\end{aligned}
\end{equation}
where
\begin{equation*}
    Y_u=-2\frac{A_\sigma}{R^3}\Big(R^3\Big(\frac{2\pi}{L(t)}\Big)^3\!-\!1\Big)\sum_{k\geq1}\!\!e^{\nu(t)k}k^{1/2}\frac{(S^{-1}u)(k)\overline{\widehat{y}(k)}+\overline{(S^{-1}u)(k)}\widehat{y}(k)}{2|\widehat{y}(k)|},
\end{equation*}
\begin{equation*}
\begin{aligned}
    Y_v&=-2(1\!+\!A_\mu)\frac{A_\rho}{R}\Big(R\frac{2\pi}{L(t)}\!-\!1\Big)
    \\
    &\quad\times\sum_{k\geq1}\!\!e^{\nu(t)k}k^{1/2}\frac{e^{-i\widehat{\vartheta}(0)}(S^{-1}v)(k)\overline{\widehat{y}(k)}\!+\!e^{i\widehat{\vartheta}(0)}\overline{(S^{-1}v)(k)}\widehat{y}(k)}{2|\widehat{y}(k)|},
\end{aligned}
\end{equation*}
\begin{equation*}
\begin{aligned}
    Y_w&=2(1\!-\!A_\mu)\frac{A_\rho}{R}\Big(R\frac{2\pi}{L(t)}\!-\!1\Big) \frac32\Big(\frac34\!-\!\log{2}\Big)
    \\
    &\quad\times\sum_{k\geq1}e^{\nu(t)k}k^{1/2}\frac12\frac{e^{+i\widehat{\vartheta}(0)}(S^{-1}w)(k)\overline{\widehat{y}(k)}\!+\!e^{-i\widehat{\vartheta}(0)}\overline{(S^{-1}w)(k)}\widehat{y}(k)}{|\widehat{y}(k)|}.
    \end{aligned}
\end{equation*}
We have the bounds
\begin{equation*}
    \begin{aligned}
    |Y_u|&\leq 2\frac{A_\sigma}{R^3}\Big|R^3\Big(\frac{2\pi}{L(t)}\Big)^3-1\Big|\sum_{k\geq1}\!\!e^{\nu(t)k}k^{1/2}|(S^{-1}u)(k)|,
    \\
    |Y_v|&\leq 2(1+A_\mu)\frac{|A_\rho|}{R}\Big|R\frac{2\pi}{L(t)}-1\Big|\!\!\sum_{k\geq1}\!\!e^{\nu(t)k}k^{1/2}|(S^{-1}v)(k)|,
    \\
    |Y_w|&\leq 2(1-A_\mu)\frac{|A_\rho|}{R}\Big|R\frac{2\pi}{L(t)}-1\Big| \frac32\Big(\frac34-\log{2}\Big)\sum_{k\geq1}e^{\nu(t)k}k^{1/2}|(S^{-1}w)(k)|.
    \end{aligned}
\end{equation*}
Noticing that $(S^{-1}w)(k)=w(k)=1_{k=2}\widehat{\theta}(1)$ and using the inequalities given by Lemma \ref{lemmaS}, we have
\begin{equation*}
    \begin{aligned}
    |Y_u|&\leq2\frac{A_\sigma}{R^3}\Big|R^3\Big(\frac{2\pi}{L(t)}\Big)^3-1\Big|C_{S}(A_\mu,\frac{|A_\rho| R^2}{A_\sigma})\sum_{k\geq2}e^{\nu(t)k}k^{1/2}|u(k)|,
    \\
    |Y_v|&\leq2(1+A_\mu)\frac{|A_\rho|}{R}\Big|R\frac{2\pi}{L(t)}-1\Big|C_{S}(A_\mu,\frac{|A_\rho| R^2}{A_\sigma})\sum_{k\geq2}e^{\nu(t)k}k^{1/2}|v(k)|,
    \\
    |Y_w|&\leq 2^{1/2}3(1-A_\mu)\frac{|A_\rho|}{R}\Big|R\frac{2\pi}{L(t)}-1\Big| \Big(\frac34-\log{2}\Big)e^{2\nu(t)}|\widehat{\theta}(1)|.
    \end{aligned}
\end{equation*}
Since 
\begin{equation*}
    k(k^2-1)\leq k^3,\quad \text{for }k\geq2,\qquad \frac{(k^2-1)(k+1)}{k^{1/2}(k+2)}\leq (k+1)^{3/2}, \quad \text{for }k\geq 1,
\end{equation*}
it holds that
\begin{equation*}
    |u(k)|\leq  k^3 \left| \widehat{\theta}(k)\right|, \qquad |v(k)|\leq (k+1)\widehat{\theta}(k+1).
\end{equation*}
Now we use the inequalities for $L(t)$ from  \eqref{Lestimates} (see Section \ref{sec:Lestimate}) 
to obtain that
\begin{equation}\label{Ybounds}
    \begin{aligned}
    |Y_u|&\leq 2\frac{A_\sigma}{R^3}C_{S}C_{40}\thetazeronenu \sum_{k\geq2}e^{\nu(t)k}k^{\frac72}|\widehat{\theta}(k)|,\\
    |Y_v|&\leq2(1+A_\mu)\frac{|A_\rho|}{R}
    C_{S}
    C_{39}\thetazeronenu  
    \sum_{k\geq2}e^{\nu(t)k}(k+1)^{3/2}|\widehat{\theta}(k+1)|,\\
    |Y_w|&\leq 2^{1/2}3(1-A_\mu)\frac{|A_\rho|}{R}C_{39}\thetazeronenu \Big(\frac34-\log{2}\Big)e^{2\nu(t)}|\widehat{\theta}(1)|.
    \end{aligned}
\end{equation}
Above we wrote $C_{S}=C_{S}(A_\mu,\frac{|A_\rho| R^2}{A_\sigma})$.
Going back to \eqref{auxx} with the bound \eqref{Lbound} and the inequality
\begin{equation*}
    k^{3/2}(k^2-1)\geq \frac34 k^{7/2},\quad \text{for }k\geq 2,
\end{equation*}
we find using also \eqref{CSbounds} that
\begin{equation}\label{auxx2}
    \begin{aligned}
    \frac{d}{dt}\|y\|_{\fhonenu}&\leq \nu'(t)\|y\|_{\fthreehonenu} -\frac32\frac{A_\sigma}{R^3}\sum_{k\geq2}e^{\nu(t)k}k^{7/2}|\widehat{y}(k)|
    \\
    &\quad+C_{S}(A_\mu,\frac{|A_\rho| R^2}{A_\sigma})C_{37}^{-1}\frac{1}{R}\|N\|_{\fhonenu}+|Y_u|+|Y_v|+|Y_w|
    .
    \end{aligned}
\end{equation}
We will next control the $k=1$ frequency in the dissipation part above.
To that end recall from Proposition \ref{IFTprop} that, if $\thetazerone < \frac12\log{\big(\frac54\big)}$, then we have 
\begin{equation*}
    2|\widehat{\theta}(1)|\leq 2C_I(\thetazeronenu)\thetazeronenu\sum_{k\geq2}|\widehat{\theta}(k)|,
\end{equation*}
which implies
\begin{equation*}
\begin{aligned}
    \thetasonenu &=2  e^{\nu(t)} |\widehat{\theta}(1)|+2\sum_{k\geq2}e^{\nu(t)k}k^s|\widehat{\theta}(k)|\\
    &\leq 2\left(C_I(\thetazeronenu)\thetazeronenu+1\right)\sum_{k\geq2}e^{\nu(t)k}k^s|\widehat{\theta}(k)|,\qquad s\geq0.
\end{aligned}
\end{equation*}  
We will find an analogous inequality in terms of $y$.  We remark that at the end of the calculation the constant $K$ in \eqref{condition} and in \eqref{Kdef} will be smaller than $\frac12\log{\big(\frac54\big)}$.

Now since $\widehat{y}(1)=\widehat{\theta}(1)$, for $s\geq0$ we have that
\begin{equation*}
\begin{aligned}
    \|y\|_{\dot{\mathcal{F}}^{s,1}_\nu}&=2  e^{\nu(t)} |\widehat{\theta}(1)|+2\sum_{k\geq2}e^{\nu(t)k}k^s|\widehat{y}(k)|\\
    &\leq 2C_I(\thetazeronenu)\thetazeronenu\sum_{k\geq2}e^{\nu(t)k}k^s|\widehat{\theta}(k)|+2\sum_{k\geq2}e^{\nu(t)k}k^s|\widehat{y}(k)|,
\end{aligned}
\end{equation*}
thus substituting $\widehat{\theta}(k)=(S\widehat{y})(k)$ and using \eqref{equivaux2} with \eqref{sigmabound} and \eqref{bessel.In.def}, we obtain
\begin{equation*}
    \begin{aligned}
    \|y\|_{\dot{\mathcal{F}}^{s,1}_\nu}&\leq 2C_I(\thetazeronenu)\thetazeronenu e^{2\nu(t)}2^s\Big|\frac{c(1)}{a(2)}\Big||\widehat{y}(1)|\\
    &\quad+2\bigg(C_I(\thetazeronenu)6\frac{I_3\Big(2\sqrt{\big|1\!+\!A_\mu\big|\frac{|A_\rho| R^2}{A_\sigma}}\Big)}{\Big(\big|1\!+\!A_\mu\big|\frac{|A_\rho| R^2}{A_\sigma}\Big)^{3/2}}\thetazeronenu+1\bigg)\sum_{k\geq2}e^{\nu(t)k}k^s|\widehat{y}(k)|.
    \end{aligned}
\end{equation*}
Subtracting the $\widehat{y}(1)$ term and using that  $e^{\nu(t)}\leq e^{\nu_0}$ for all $t\geq0$, we find
\begin{equation}\label{ybound.imlicit}
    \begin{aligned}
    \|y\|_{\dot{\mathcal{F}}^{s,1}_\nu}&\leq 2C_y\sum_{k\geq2}e^{\nu(t)k}k^s|\widehat{y}(k)|, \qquad s\geq0,
    \end{aligned}
\end{equation}
where using also \eqref{coeffs} we define
\begin{equation}\label{Cy}
    \begin{aligned}
    C_y&=\frac{C_I(\thetazeronenu)6\frac{I_3\Big(2\sqrt{\big|1\!+\!A_\mu\big|\frac{|A_\rho| R^2}{A_\sigma}}\Big)}{\Big(\big|1\!+\!A_\mu\big|\frac{|A_\rho| R^2}{A_\sigma}\Big)^{3/2}}\thetazeronenu+1}{1-C_I(\thetazeronenu) 2^s e^{\nu_0}\frac{1}{4}\big|1\!-\!A_\mu\big|\frac{|A_\rho| R^2}{A_\sigma}\Big(\frac34-\log{2}\Big)\thetazeronenu}.
    \end{aligned}
\end{equation}
This is the bound for $\|y\|_{\dot{\mathcal{F}}^{s,1}_\nu}$ that we will use in \eqref{auxx2}.

Notice that in \eqref{Ybounds}, regarding $Y_w$, that $3 \sqrt{2} \left(\frac{3}{4}-\log(2)\right) \le 1.91\le 2$.  Also in Lemma \ref{lemmaS} we have $C_S \ge 1$.  Thus, from \eqref{Ybounds}, we have 
$$
|Y_v|+|Y_w|
\leq
    2\left(1+|A_\mu|\right)\frac{|A_\rho|}{R}
    C_{S}
    C_{39}
    e^{\nu(t)}\thetazeronenu  
    \sum_{k\geq1}e^{\nu(t)k}k^{3/2}|\widehat{\theta}(k)|,
$$
Now we go back to \eqref{auxx2}, use \eqref{ybound.imlicit}, and substitute in the bounds for $Y_u, Y_v, Y_w$ \eqref{Ybounds} to obtain
\begin{equation*}
\begin{aligned}
    \frac{d}{dt}\|y\|_{\fhonenu}&\leq \nu'(t)\|y\|_{\fthreehonenu} -\frac34\frac{A_\sigma}{R^3}C_y^{-1}\|y\|_{\dot{\mathcal{F}}^{\frac72,1}_\nu}
    \\
    &\quad+C_{S} 
    C_{37}^{-1}\frac{1}{R}\|N\|_{\fhonenu}\\
    &\quad
    +2\frac{A_\sigma}{R^3}C_{S} 
    C_{40}\thetazeronenu \sum_{k\geq2}e^{\nu(t)k}k^{\frac72}|\widehat{\theta}(k)|
    \\
    &\quad 
    +
    2\left(1+|A_\mu|\right)\frac{|A_\rho|}{R}
    C_{S}
    C_{39}
    e^{\nu(t)}\thetazeronenu  
    \sum_{k\geq1}e^{\nu(t)k}k^{3/2}|\widehat{\theta}(k)|,
\end{aligned}
\end{equation*}    
where we have combined the bounds for $Y_v$ and $Y_w$ as above.
The reverse inequalities \eqref{CSbounds} and the embeddings \eqref{embed} give that
\begin{equation}\label{balance}
\begin{aligned}
    \frac{d}{dt}\|y\|_{\fhonenu}&\leq  \frac{A_\sigma}{R^3}\bigg(-\frac34C_y^{-1}+\nu'(t)\frac{R^3}{A_\sigma}+2\big(C_{S}\big)^2C_{40}\|y\|_{\dot{\mathcal{F}}^{0,1}_\nu} \\
    &\quad+2(1+|A_\mu|)\frac{|A_\rho|R^2}{A_\sigma} \big(C_{S}\big)^2C_{39}e^{\nu(t)}\|y\|_{\dot{\mathcal{F}}^{0,1}_\nu}\bigg)\|y\|_{\dot{\mathcal{F}}^{\frac72,1}_\nu}\\
    &+ 
    C_{S} C_{37}^{-1}\frac{1}{R}\|N\|_{\fhonenu}.
\end{aligned}
\end{equation}    
Next, we will use \eqref{Nbound} to control the nonlinear term $\|N\|_{\fhonenu}$. Together with \eqref{Lbound}, we obtain
\begin{equation*}
    \begin{aligned}
        \|N\|_{\fhonenu}&\leq \frac{A_\sigma}{R^2}C_{35}C_{37}^{-2}\thetahonenu\thetasevenhonenu
        +|A_\rho|e^{\nu(t)}C_{36}\thetahonenu\thetasevenhonenu.
    \end{aligned}
\end{equation*}
Thus using also \eqref{CSbounds} we have
\begin{equation}\label{nonlinearaux}
    \begin{aligned}
        \|N\|_{\fhonenu}&\leq \frac{A_\sigma}{R^2}\big(C_{S}\big)^2\Big(C_{35}C_{37}^{-2}
        +\frac{|A_\rho|R^2}{A_\sigma}e^{\nu(t)}C_{36}\Big)\|y\|_{\dot{\mathcal{F}}^{\frac12,1}_\nu}\|y\|_{\dot{\mathcal{F}}^{\frac72,1}_\nu}.
    \end{aligned}
\end{equation}
Substitution of \eqref{nonlinearaux}  into \eqref{balance}, and using once more \eqref{embed}, provides that
\begin{equation*}
\begin{aligned}
    \frac{d}{dt}\|y\|_{\fhonenu}&\leq - \frac{A_\sigma}{R^3}\bigg(\frac34C_y^{-1}-\nu'(t)\frac{R^3}{A_\sigma}-2\big(C_{S}\big)^2C_{40}\|y\|_{\dot{\mathcal{F}}^{\frac12,1}_\nu} \\
    &-2(1+|A_\mu|)\frac{|A_\rho|R^2}{A_\sigma} \big(C_{S}\big)^2C_{39}e^{\nu(t)}\|y\|_{\dot{\mathcal{F}}^{\frac12,1}_\nu} \\   &-\big(C_{S}\big)^3C_{37}^{-1}\Big(C_{35}C_{37}^{-2}\!+\!\frac{|A_\rho|R^2}{A_\sigma}e^{\nu(t)}C_{36}\Big)\|y\|_{\dot{\mathcal{F}}^{\frac12,1}_\nu}
    \bigg)\|y\|_{\dot{\mathcal{F}}^{\frac72,1}_\nu}.
\end{aligned}
\end{equation*}  
Since $C_y(\thetazeronenu)$ in \eqref{Cy}, $C_{35}(\thetazeronenu)$ and $C_{36}(\thetazeronenu)$ in \eqref{C35C36}, $\big(C_{37}(\thetazeronenu)\big)^{-1}$ in \eqref{C37C38}, $C_{39}(\thetazeronenu)$ and $C_{40}(\thetazeronenu)$ in \eqref{C39C40}, are increasing functions of the argument, we can bound all of them by evaluating at the bigger quantity $C_S\|y\|_{\dot{\mathcal{F}}^{\frac12,1}_\nu}$ since 
$\|\theta\|_{\dot{\mathcal{F}}^{0,1}_\nu} \le \|\theta\|_{\dot{\mathcal{F}}^{\frac12,1}_\nu}\le C_S\|y\|_{\dot{\mathcal{F}}^{\frac12,1}_\nu}$.
Here we suppress the dependency on $A_\mu$, $|A_\rho|R^2/A_\sigma$ from $C_S$. 

For clarity of notation, in formulas \eqref{dissipation} and \eqref{C41} below, we understand that all of these functions are evaluated at $C_S\|y\|_{\dot{\mathcal{F}}^{\frac12,1}_\nu}$.
We conclude that
\begin{equation}\label{balance2}
\begin{aligned}
    \frac{d}{dt}\|y\|_{\fhonenu}&\leq -\frac{A_\sigma}{R^3} \dissconstant \|y\|_{\dot{\mathcal{F}}^{\frac72,1}_\nu},
\end{aligned}
\end{equation}
with
\begin{equation}\label{dissipation}
    \begin{aligned}
       \dissconstant\Big(\|y\|_{\fhonenu},\frac{|A_\rho|R^2}{A_\sigma},
       A_\mu,\frac{A_\sigma}{R^3},\nu\Big)&=\frac34C_y^{-1}-\nu'(t)\frac{R^3}{A_\sigma}-C_{41}\|y\|_{\dot{\mathcal{F}}^{\frac12,1}_\nu},
    \end{aligned}
\end{equation}
where
\begin{multline}\label{C41}
    C_{41}=
    \\
    C_S^2\Big(
       2C_{40}+2(1+|A_\mu|)\frac{|A_\rho|R^2}{A_\sigma}C_{39}e^{\nu(t)}+C_{S}C_{37}^{-1}\Big(C_{35}C_{37}^{-2}\!+\!\frac{|A_\rho|R^2}{A_\sigma}e^{\nu(t)}C_{36}\Big)\Big).
\end{multline}
Note that we can initially choose $\nu_0>0$ in \eqref{nu} to be arbitrarily small, and that $\nu'(t) = \frac{\nu_0}{(1+t)^2}\le \nu_0$.  

Finally, suppose that the following condition holds initially
\begin{equation}\label{dissipationcond}
     \frac34C_y^{-1}-C_{41}\|y\|_{\dot{\mathcal{F}}^{\frac12,1}_\nu}>0,
\end{equation}
where $C_y$ was defined in \eqref{Cy}.
For $\nu_0$ small enough, using \eqref{balance2} and the fact that $C_{41}$ decreases as $\|y\|_{\dot{\mathcal{F}}^{\frac12,1}_\nu}$ decreases, then this condition will be propagated in time.   Thus it holds that
\begin{equation}\label{auxbal}
    	\|y\|_{\fhonenu}(t)+ \frac{A_\sigma}{R^3}  \dissconstant
       \!\int_0^t\! \|y\|_{\dot{\mathcal{F}}^{\frac72,1}_\nu}(\tau) d\tau \leq \|y_0\|_{\fhone}.
\end{equation}
Above $\dissconstant=\dissconstant\Big(\|y_0\|_{\fhonenu},\frac{|A_\rho|R^2}{A_\sigma}, A_\mu,\frac{A_\sigma}{R^3},\nu\Big)$.
Now since $C_y$ and $C_{41}$ are monotonically increasing in $\|y\|_{\dot{\mathcal{F}}^{\frac12,1}_\nu}$, the positivity condition \eqref{dissipationcond} is equivalent to a medium-size condition on the initial data
\begin{equation*}
    \|y_0\|_{\dot{\mathcal{F}}^{\frac12,1}}<h\Big(\frac{|A_\rho|R^2}{A_\sigma}, A_\mu\Big),
\end{equation*}
where the function $h$ is implicitly defined via \eqref{dissipationcond}. 
Recalling the bound for the inverse relation $\widehat{\theta}(k)=(S\widehat{y})(k)$ in \eqref{CSbounds}, and denoting \begin{equation}\label{Kdef}
    K\Big(\frac{|A_\rho|R^2}{A_\sigma}, A_\mu\Big)=C_S\big(A_\mu,\frac{|A_\rho|R^2}{A_\sigma}\big) h\Big(\frac{|A_\rho|R^2}{A_\sigma}, A_\mu\Big),
\end{equation}
we obtain the medium-size condition for $\theta_0$ given in \eqref{condition}.  This completes the proof of \eqref{condition} and \eqref{estimatef12}.  The proof of \eqref{decay} follows from applying \eqref{embed} to \eqref{balance2}.

To perform the estimate for $\widehat{\vartheta}(0)$ from \eqref{zeroBoundUniform}, we use \eqref{eqcurve3}$_1$ to obtain
\begin{equation*}
    \widehat{\vartheta}_t(0)=\frac{2\pi}{L(t)}\widehat{T}*\widehat{(1+\theta_\alpha)}(0).
\end{equation*}
Note that $\widehat{U_\alpha}(0)=0$.  We use the splitting for $T$ from \eqref{Tsplit}, and notice from \eqref{u0t0} with \eqref{T0} that 
\begin{equation*}
    \widehat{T_0}(k)=0 \text{ for }|k|\neq 1, \qquad |\widehat{T_0}(k)|\leq \frac{|A_\rho|}2 \text{ for }|k|= 1,
\end{equation*}
thus $\|T_0\|_{\fzerone}\leq A_\rho$. Then we have
\begin{equation*}
    |\widehat{\vartheta}(0)|
    \leq |\widehat{\vartheta}_0(0)|+
    \int_0^t \frac{2\pi}{L(\tau)}\Big(A_\rho\|\theta\|_{\foneone}+\big(\|T_1\|_{\fzerone}+\|T_{\geq2}\|_{\fzerone}\big)\big(1+\|\theta\|_{\foneone}\big)\Big) d \tau.
\end{equation*}
The estimates \eqref{T1aux} and \eqref{T2aux} with $s=0$ in Section \ref{sec:NonLinearEst}, together with the bounds for $L(t)$ in \eqref{Lbound}, give 
\begin{equation*}
    \begin{aligned}
    \|T_1\|_{\fzerone}&\leq \frac{A_\sigma}{R^2}\Big(2\frac{|A_\rho|R^2}{A_\sigma}e^{\nu(t)}
    \thetaoneone\!+\!C_{37}^{-2}\thetatwoone\!+\!2\frac{|A_\rho|R^2}{A_\sigma}\big((1\!+\! 2|A_\mu|)\!+\!2{\crconstant}\big)e^{\nu(t)}\thetazerone\Big),
    \\
    \|T_{\geq2}\|_{\fzerone}&\leq \frac{A_\sigma}{R^2}\Big(\frac{|A_\rho|R^2}{A_\sigma}C_{34}e^{\nu(t)}\thetazerone\thetaoneone\!+\!C_{37}^{-2}C_{33}\thetaoneone\thetatwoone\Big).
    \end{aligned}
\end{equation*}
Therefore, reordering and using \eqref{auxbal}, \eqref{interp} and \eqref{embed}, we conclude that
\begin{equation*}
    \begin{aligned}
    |\widehat{\vartheta}(0)|&\leq |\widehat{\vartheta}_0(0)|+C_{42}\|\theta_0\|_{\fhone},
    \end{aligned}
\end{equation*}
where 
\begin{multline}    \label{C42}
    C_{42}=
    \\ 
    \frac{1}{\dissconstant}C_S^2C_{37}^{-1}\Bigg(\frac{|A_\rho|R^2}{A_\sigma}+C_{37}^{-2}+2\frac{|A_\rho|R^2}{A_\sigma}e^{\nu(t)}\Big(1+\big((1+2|A_\mu|)+2{\crconstant}\big)\Big)
    \\+C_S^2\|\theta_0\|_{\fhone}\bigg(C_{37}^{-2}+C_{37}^{-2}C_{33}(1+C_S\|\theta_0\|_{\fhone})
    \\
    +
    2\frac{|A_\rho|R^2}{A_\sigma}e^{\nu(t)}\Big(1+\big((1+2|A_\mu|)+2{\crconstant}\big)+\frac{1}{2}C_{34}(1+C_S\|\theta_0\|_{\fhone})\Big)\bigg)\Bigg).
\end{multline} 
In $C_{42}$ above  the quantities ($\dissconstant$ in \eqref{dissipation}, $C_{37}$ in \eqref{C37C38}, $C_{33}$ and $C_{34}$ in \eqref{C33C34}) are evaluated at
$C_S\|\theta_0\|_{\fhone}$, with $C_S$ defined in \eqref{CSbounds}.
This completes the proof of \eqref{zeroBoundUniform}.
\end{proof}

We will now prove Proposition \ref{diagonalization}.

\begin{proof}[Proof of Proposition \ref{diagonalization}]
Because $S^{-1}$ is injective and surjective, it suffices to prove that $S$ is a right inverse of $S^{-1}$, i.e., that
\begin{equation}\label{rightinverse.proof}
    \sum_{j\geq1}S^{-1}_{k,j}S_{j,m}=\delta_{k,m},
\end{equation}
where $\delta_{k,m}$ is the Kronecker delta. 
The cases $k=1$ and $k=2, m=1$ are straightforward. For the rest of cases, we have that
\begin{equation*}
    \begin{aligned}
    \sum_{j\geq1}S^{-1}_{k,j}S_{j,m}&=\sum_{j\geq1}(-1)^{j-k} \delta_{j\geq k}\prod_{l=1}^{j-k} \frac{b(k-1+l)}{a(k)-a(k+l)}\delta_{m\geq j}\prod_{n=1}^{m-j}\frac{b(j-1+n)}{a(j-1+n)-a(m)}\\
    &=\sum_{j=k}^m (-1)^{j-k}\prod_{n=1}^{m-j}\prod_{l=1}^{j-k} \frac{b(k-1+l)b(j-1+n)}{(a(k)-a(k+l))(a(j-1+n)-a(m))}.
    \end{aligned}
\end{equation*}
If $m=k$, then clearly $\sum_{j\geq1}S^{-1}_{k,j}S_{j,m}=1$. Now, we notice that
\begin{equation*}
    \begin{aligned}
    \prod_{n=1}^{m-j}\prod_{l=1}^{j-k} b(k\!-\!1\!+\!l)b(j\!-\!1\!+\!n)&=b(k)b(k\!+\!1)\dots b(j\!-\!2)b(j\!-\!1)b(j)b(j\!+\!1)\dots b(m\!-\!1)\\
    &=\prod_{n=0}^{m-k-1}b(k+n)
    \end{aligned}
\end{equation*}
is independent of $j$, and thus comes out of the sum as
\begin{equation*}
    \begin{aligned}
    \sum_{j\geq1}S^{-1}_{k,j}S_{j,m}=\prod_{n=0}^{m-k-1}b(k+n) \sum_{j=k}^m (-1)^{j-k}A(j),
    \end{aligned}
\end{equation*}
where we define
\begin{equation*}
    A(j)=\prod_{n=1}^{m-j}\prod_{l=1}^{j-k} \frac{1}{(a(k)\!-\!a(k\!+\!l))(a(j\!-\!1\!+\!n)\!-\!a(m))}.
\end{equation*}
Changing $n\leftarrow j-k-1+n$, it becomes
\begin{equation*}
    A(j)=\prod_{n=j-k}^{m-k-1}\frac{1}{a(k+n)\!-\!a(m)}\prod_{l=1}^{j-k} \frac{1}{a(k)\!-\!a(k\!+\!l)}.
\end{equation*}
Thus, multiplying and dividing by the product of $a(k+n)-a(m)$ from $n=0$ to $n=j-k-1$, we rewrite it as follows
\begin{equation*}
    \begin{aligned}
    A(j)=\prod_{n=0}^{m-k-1}\frac{1}{a(k+n)-a(m)}\prod_{l=1}^{j-k}\frac{a(k+l-1)-a(m)}{a(k)-a(k+l)}.
    \end{aligned}
\end{equation*}
Since the first product does not depend on $j$, we obtain that
\begin{equation*}
    \begin{aligned}
    \sum_{j\geq1}S^{-1}_{k,j}S_{j,m}=\prod_{n=0}^{m-k-1}\frac{b(k+n)}{a(k+n)-a(m)}\sum_{j=k}^m (-1)^{j-k}\prod_{l=1}^{j-k}\frac{a(k+l-1)-a(m)}{a(k)-a(k+l)}.
    \end{aligned}
\end{equation*}
Changing the index $j\leftarrow j-k$ and denoting $\tilde{m}=m-k$, we are left to show that, for all $\tilde{m}\geq1$,
\begin{equation*}
    \begin{aligned}
    \sum_{j=0}^{\tilde{m}} (-1)^{j}\prod_{l=1}^{j}\frac{a(k+l-1)-a(\tilde{m}+k)}{a(k)-a(k+l)}=0.
    \end{aligned}
\end{equation*}
Writing the sum with a common denominator, the identity that we need to prove is
\begin{equation*}
    \begin{aligned}
    W=\sum_{j=0}^{\tilde{m}} (-1)^{j}\prod_{l=1}^{j}\big(a(k+l-1)-a(\tilde{m}+k)\big)\prod_{l=j+1}^{\tilde{m}}\big(a(k)-a(k+l)\big)=0.
    \end{aligned}
\end{equation*}
We add the first two terms in the sum by noticing that they only differ in one term, which partially cancels as
\begin{equation*}
    \begin{aligned}
    W&=-\big(a(k\!+\!1)\!-\!a(k\!+\!\tilde{m})\big)\prod_{l=2}^{\tilde{m}}\big(a(k)\!-\!a(k\!+\!l)\big)\\
    &+
    \sum_{j=2}^{\tilde{m}} (-1)^{j}\prod_{l=1}^{j}\big(a(k+l-1)-a(\tilde{m}+k)\big)\prod_{l=j+1}^{\tilde{m}}\big(a(k)-a(k+l)\big).
    \end{aligned}
\end{equation*}
We now realize that we can repeat the same step with the new first two terms to obtain that
\begin{equation*}
    \begin{aligned}
    W&=\big(a(k\!+\!1)\!-\!a(k\!+\!\tilde{m})\big)\big(a(k\!+\!2)\!-\!a(k\!+\!\tilde{m})\big)\prod_{l=3}^{\tilde{m}}\big(a(k)\!-\!a(k\!+\!l)\big)\\
    &+
    \sum_{j=3}^{\tilde{m}} (-1)^{j}\prod_{l=1}^{j}\big(a(k+l-1)-a(\tilde{m}+k)\big)\prod_{l=j+1}^{\tilde{m}}\big(a(k)-a(k+l)\big).
    \end{aligned}
\end{equation*}
We can continue the process to find that
\begin{equation*}
    \begin{aligned}
    &W=\!(-1)^r\big(a(k\!+\!1)\!-\!a(k\!+\!\tilde{m})\big)\big(a(k\!+\!2)\!-\!a(k\!+\!\tilde{m})\big)\\ &\hspace{3cm}\dots\big(a(k\!+\!r)\!-\!a(k\!+\!\tilde{m})\big)\!\!\!\!
    \prod_{l=r+1}^{\tilde{m}}\!\!\!\big(a(k)\!-\!a(k\!+\!l)\big)\\
    &\hspace{1cm}+
    \sum_{j=r+1}^{\tilde{m}} (-1)^{j}\prod_{l=1}^{j}\big(a(k+l-1)-a(\tilde{m}+k)\big)\prod_{l=j+1}^{\tilde{m}}\big(a(k)-a(k+l)\big).
    \end{aligned}
\end{equation*}
Thus for $r=\tilde{m}-1$, we conclude that
\begin{equation*}
    \begin{aligned}
    W
    &=\!(-1)^{\tilde{m}-1}\big(a(k\!+\!1)\!-\!a(k\!+\!\tilde{m})\big)\dots \big(a(k\!+\!\tilde{m}\!-\!1)\!-\!a(k\!+\!\tilde{m})\big)\big(a(k)\!-\!a(k\!+\!\tilde{m})\big)\\
    &+ (-1)^{\tilde{m}}\prod_{l=1}^{\tilde{m}}\big(a(k+l-1)-a(\tilde{m}+k)\big)=0.
    \end{aligned}
\end{equation*}
This proves \eqref{rightinverse.proof}.

We next prove that $S^{-1}$ in \eqref{Sm1} diagonalizes \eqref{linearZ}. Notice that $S^{-1}_{k,k}=1$ and $S^{-1}_{k,j}=0$ for $j<k$, $k\geq2$. Then, for $k\geq3$, we can write
\begin{equation*}
    \begin{aligned}
    y_t(k)=\sum_{j\geq k}S^{-1}_{k,j}z_t(j)=-a(k)S^{-1}_{k,k}z(k)+\sum_{j\geq k+1}\Big(-a(j)S^{-1}_{k,j}+b(j-1)S^{-1}_{k,j-1}\Big)z(j).
    \end{aligned}
\end{equation*}
Finally, notice that 
\begin{equation*}
    -a(j)S^{-1}_{k,j}+b(j-1)S^{-1}_{k,j-1}=-a(k)S^{-1}_{k,j},\qquad j-1\geq k\geq 2,
\end{equation*}
precisely when $S^{-1}_{k,j}$ is given by \eqref{Sm1}. The case $k=1$ is trivial, while for $k=2$ it holds that
\begin{equation*}
    \begin{aligned}
    y_t(2)&=\sum_{j\geq 2}S^{-1}_{2,j}z_t(j)=c(1)z(1)\!-\!a(2)z(2)\!+\!\sum_{j\geq 3}\!\Big(\!-a(j)S^{-1}_{2,j}\!+\!b(j\!-\!1)S^{-1}_{2,j-1}\Big)z(j)\\
    &=-a(2)\frac{c(1)}{a(2)}z(1)-a(2)\sum_{j\geq2}S^{-1}_{2,j}z(j)=-a(2)y(2).
    \end{aligned}
\end{equation*}
This completes the proof.
\end{proof}

We are now ready to prove Lemma \ref{lemmaS}.

\begin{proof}[Proof of Lemma \ref{lemmaS}]
Denote for $j \ge k \ge 2$
\begin{equation*}
\begin{aligned}
    \beta(k,j)=(-1)^{j-k}\prod_{m=1}^{j-k} \frac{b(k-1+m)}{a(k)-a(k+m)}.
\end{aligned}
\end{equation*}
Plugging in the expressions for $b(k)$ and $a(k)$  in \eqref{coeffs}, we find that
\begin{equation*}
    \begin{aligned}
    \beta(k,j)\!=\!\Big(\big(1\!+\!A_\mu\big)\frac{A_\rho R^2}{A_\sigma}\Big)^{j-k}\!\!\!\!e^{-i(j-k)\widehat{\vartheta}(0)}\!\!\prod_{m=1}^{j-k}\!\!\frac{(k\!+\!m)^2(k\!-\!2\!+\!m)}{(k\!-\!1\!+\!m)(k\!+\!1\!+\!m)}\frac{1}{k^3\!-\!(k\!+\!m)^3\!+\!m}.
    \end{aligned}
\end{equation*}
Then, we can write 
\begin{equation}\label{equivaux}
    \begin{aligned}
    \|y\|_{\ell^1}&=\sum_{k\geq1}|y(k)|\leq \Big(1+\Big|\frac{c(1)}{a(2)}\Big|\Big)|z(1)|+\sum_{k\geq2}\sum_{j\geq k}|\beta(k,j)||z(j)|\\
    &=\Big(1+\frac{1}{4}\big|1\!-\!A_\mu\big|\frac{|A_\rho| R^2}{A_\sigma}\Big(\frac34-\log{2}\Big)\Big)|z(1)|+\sum_{l\geq2}\sum_{j=0}^{l-2}|\beta(l-j,l)||z(l)|\\
    &\leq \Big(1+\frac{1}{4}\big|1\!-\!A_\mu\big|\frac{|A_\rho| R^2}{A_\sigma}\Big(\frac34-\log{2}\Big)\Big)|z(1)|+\sum_{l\geq2}\lambda(l)|z(l)|,
    \end{aligned}
\end{equation}
where
\begin{equation*}
    \lambda(l)=\sum_{j=0}^{l-2}|\beta(l-j,l)|.
\end{equation*}
In the following we will estimate $\lambda(l)$ to control $\beta$.

The identity $k^3-(k+m)^3+m=-m\big(m^2+3m k+3k^2-1\big)$ gives the bound
\begin{equation}\label{index.lower}
    |(l-j)^3-(l-j+m)^3+m|\geq m\big(m^2+2m(l-j)+(l-j)^2\big)=m(m+l-j)^2.
\end{equation}
Thus
\begin{equation}\label{betabound}
    \begin{aligned}
    |\beta(l-j,l)|&\leq \Big|\big(1\!+\!A_\mu\big)\frac{A_\rho R^2}{A_\sigma}\Big|^{j}\prod_{m=1}^{j}\frac{l-j\!-\!2\!+\!m}{m(l-j\!-\!1\!+\!m)(l-j\!+\!1\!+\!m)}\\
    &\leq \Big|\big(1\!+\!A_\mu\big)\frac{A_\rho R^2}{A_\sigma}\Big|^{j}\prod_{m=1}^{j}\frac{1}{m(l-j\!+\!1\!+\!m)}\\
    &\leq \Big(\big|\!+\!A_\mu\big|\frac{|A_\rho| R^2}{A_\sigma}\Big)^{j}\prod_{m=1}^{j}\frac{1}{m(m+3)},
    \end{aligned}
\end{equation}
where the last step is due to the fact that $j\leq l-2$. Therefore, we have that
\begin{equation*}
    \begin{aligned}
    \lambda(l)&\leq \sum_{j=0}^{l-2}\Big(\big|1\!+\!A_\mu\big|\frac{|A_\rho| R^2}{A_\sigma}\Big)^j\frac{6}{j!(j+3)!} \leq \sum_{j=0}^{\infty}\Big(\big|1\!+\!A_\mu\big|\frac{|A_\rho| R^2}{A_\sigma}\Big)^j\frac{6}{j!(j+3)!}\\
    &=6\frac{I_3\Big(2\sqrt{\big|1\!+\!A_\mu\big|\frac{|A_\rho| R^2}{A_\sigma}}\Big)}{\Big(\big|1\!+\!A_\mu\big|\frac{|A_\rho| R^2}{A_\sigma}\Big)^{3/2}},
    \end{aligned}
\end{equation*}
where $I_3$ is the modified Bessel function of the first kind of order three as in \eqref{bessel.In.def}.
We conclude thereby from \eqref{equivaux} that
\begin{equation*}
    \begin{aligned}
    \|y\|_{\ell^1}&\leq
    \Big(1\!+\!\frac{1}{4}\big|1\!-\!A_\mu\big|\frac{|A_\rho| R^2}{A_\sigma}\Big(\frac34-\log{2}\Big)\Big)|z(1)|+6\frac{I_3\Big(2\sqrt{\big|1\!+\!A_\mu\big|\frac{|A_\rho| R^2}{A_\sigma}}\Big)}{\Big(\big|1\!+\!A_\mu\big|\frac{|A_\rho| R^2}{A_\sigma}\Big)^{3/2}}\sum_{l\geq2}|z(l)|\\
    &\leq C_{S}(A_\mu,\frac{|A_\rho| R^2}{A_\sigma})\|z\|_{\ell^1},
    \end{aligned}
\end{equation*}
with $C_{S}$ given in \eqref{CSbounds}. 

We proceed to prove the reverse direction, $z=S y$. For $j \ge k \ge 2$, denote
\begin{equation*}
\begin{aligned}
    \alpha(k,j)=\prod_{m=1}^{j-k} \frac{b(k-1+m)}{a(k-1+m)-a(j)}
.
\end{aligned}
\end{equation*}
Recalling \eqref{coeffs}, it becomes
\begin{equation*}
    \begin{aligned}
    \alpha(k,j)&=(-1)^{j-k}\Big(\big|1\!+\!A_\mu\big|\frac{|A_\rho| R^2}{A_\sigma}\Big)^{j-k}e^{-i(j-k)\widehat{\vartheta}(0)}\\
    &\quad\cdot\prod_{m=1}^{j-k}\frac{(k\!+\!m)^2(k\!-\!2\!+\!m)}{(k\!-\!1\!+\!m)(k\!+\!1\!+\!m)}\frac{1}{(k\!-\!1\!+\!m)^3\!-\!j^3\!+\!j\!-\!k\!+\!1\!-\!m}.
    \end{aligned}
\end{equation*}
Next, proceeding as in \eqref{equivaux}, we have
\begin{equation}\label{equivaux2}
    \begin{aligned}
    \|z\|_{\ell^1}&=\sum_{k\geq1}|z(k)|\leq \big(1+\Big|\frac{c(1)}{a(2)}\Big|\big)|y(1)|+\sum_{k\geq2}\sum_{j\geq k}|\alpha(k,j)||y(j)|\\
    &\leq \Big(1+\frac14\big|1\!-\!A_\mu\big|\frac{|A_\rho| R^2}{A_\sigma}\Big(\frac34-\log{2}\Big)\Big)|y(1)|+\sum_{l\geq2}\sigma(l)|y(l)|,
    \end{aligned}
\end{equation}
where
\begin{equation*}
    \sigma(l)=\sum_{j=0}^{l-2}|\alpha(l-j,l)|.
\end{equation*}
The inequality below follows from \eqref{index.lower} with $M+L-J = j$ and $L-J=m+k-1$
\begin{equation*}
    \begin{aligned}
    |(k-1+m)^3-j^3+j-k+1-m|\geq (j-k+1-m)j^2.
    \end{aligned}
\end{equation*}
Using the inequality above,  and $m \le j \le l -2$, we obtain that
\begin{equation*}
    \begin{aligned}
    |\alpha(l-j,l)|&\leq \Big|\big(1\!+\!A_\mu\big)\frac{A_\rho R^2}{A_\sigma}\Big|^{j}\prod_{m=1}^{j}\frac{(l-j\!+\!m)^2(l-j-2+m)}{(l-j\!-\!1\!+\!m)(l-j\!+\!1\!+\!m)}\frac{1}{l^2(j\!+\!1\!-\!m)}\\
    &\leq \Big|\big(1\!+\!A_\mu\big)\frac{A_\rho R^2}{A_\sigma}\Big|^{j}\prod_{m=1}^{j}\frac{1}{m(l-j\!+\!1\!+\!m)}\\
    &\leq \Big(\big|1\!+\!A_\mu\big|\frac{|A_\rho| R^2}{A_\sigma}\Big)^{j}\prod_{m=1}^{j}\frac{1}{m(m+3)},
    \end{aligned}
\end{equation*}
which coincides with the bound found for $\beta$ in \eqref{betabound}. Therefore, 
\begin{equation}\label{sigmabound}
    \begin{aligned}
    \sigma(l)&\leq6\frac{I_3\Big(2\sqrt{\big|1\!+\!A_\mu\big|\frac{|A_\rho| R^2}{A_\sigma}}\Big)}{\Big(\big|1\!+\!A_\mu\big|\frac{|A_\rho| R^2}{A_\sigma}\Big)^{3/2}}.
    \end{aligned}
\end{equation}
Above we recall $I_3$ from \eqref{bessel.In.def}.  Then going back to \eqref{equivaux2} we have
\begin{equation*}
    \begin{aligned}
    \|z\|_{\ell^1}&\leq 
     C_S\left(A_\mu,\frac{|A_\rho| R^2}{A_\sigma}\right)\|y\|_{\ell^1},
    \end{aligned}
\end{equation*}
with $C_S$ defined in \eqref{CSbounds}.   This completes the proof.
\end{proof}

\section{Nonlinear Estimates in $\fsonenu$} \label{sec:NonLinearEst}
We recall that the nonlinear terms \eqref{system} are given by
\begin{equation*}
    N(\alpha)=N_1+N_2+N_3,
\end{equation*}
with
\begin{equation}\label{N1N2N3}
\begin{aligned}
    N_1=(U_{\geq 2})_\alpha(\alpha),\quad N_2=T_{\geq 2}(\alpha)(1+\theta_\alpha(\alpha)),\quad
    N_3=T_1(\alpha)\theta_\alpha(\alpha),
\end{aligned}
\end{equation}
where $U$ and $T$ are described in \eqref{Usplit} and \eqref{Tsplit}.

We will first perform the estimates for $N_1$ for $s\ge 0$.  From \eqref{Usplit}, we have
\begin{multline}\label{N1f0}
    \|N_1\|_{\fsonenu}=\|U_{\geq2}\|_{\fonesonenu}
    \\
    \!\leq\! \frac{\pi}{L(t)}\Big(\|\omega_{\geq2}\|_{\fonesonenu}\!+\| \mathcal{R}(\omega_{\geq1})\|_{\fonesonenu}\!+\| \mathcal{S}(\omega)\|_{\fonesonenu}\Big).
\end{multline}
Using the bound \eqref{Restimates}$_2$ and splitting $\omega_{\geq1}=\omega_{1}+\omega_{\geq2}$ gives that
\begin{equation*}
    \begin{aligned}
        \| \mathcal{R}(\omega_{\geq1})\|_{\fonesonenu}&\leq   \bfcn(2,s){\crconstant}\thetazeronenu\|\omega_1\|_{\fonesonenu}+\bfcn(2,s){\crconstant}\thetazeronenu\|\omega_{\geq2}\|_{\fonesonenu}\\
        &\quad+\bfcn(2,s){\crconstant}\|\theta\|_{\fonesonenu}\|\omega_1\|_{\fzeronenu}
        +\bfcn(2,s){\crconstant}\|\theta\|_{\fonesonenu}\|\omega_{\geq2}\|_{\fzeronenu}.
    \end{aligned}
\end{equation*}
Then introducing the vorticity estimates \eqref{omega1f0}, \eqref{omega1fss}, \eqref{omega2f0}, and \eqref{omega2fss}, and ordering the terms in $A_\rho$ and $A_\sigma$, it follows that
\begin{equation*}
    \begin{aligned}
        \| \mathcal{R}&(\omega_{\geq1})\|_{\fonesonenu}\leq A_\sigma\frac{4\pi}{L(t)}\bfcn(2,s){\crconstant}\Big(\thetazeronenu\thetathreesonenu
        +|A_\mu|C_{13}\thetazeronenu^2\thetathreesonenu\\
        &\quad+\thetaonesonenu\thetatwoonenu+|A_\mu|C_{9}\thetaonesonenu\thetazeronenu\thetatwoonenu\Big)\\
        &\quad+|A_\rho|\frac{L(t)}{\pi}e^{\nu(t)}\bfcn(2,s){\crconstant}\Big(2(1+2|A_\mu|)\thetazeronenu\thetaonesonenu+ C_{14}\thetazeronenu^2\thetaonesonenu\\
        &\quad +(1+2|A_\mu|)\thetaonesonenu \thetazeronenu+C_{10}\thetaonesonenu\thetazeronenu^2\Big).
    \end{aligned}
\end{equation*}
Interpolation as in \eqref{interpolation} yields that
\begin{equation}\label{est1}
    \begin{aligned}
        \| \mathcal{R}(\omega_{\geq1})\|_{\fonesonenu}&\leq A_\sigma\frac{4\pi}{L(t)}C_{21}\thetazeronenu\thetathreesonenu
        \\
        &\quad +|A_\rho|\frac{L(t)}{\pi}e^{\nu(t)}C_{22}\thetazeronenu\thetaonesonenu,
    \end{aligned}
\end{equation}
where
\begin{equation}\label{C21C22}
\begin{aligned}
    C_{21}&=2\bfcn(2,s){\crconstant}+|A_\mu|\bfcn(2,s){\crconstant}(C_{9}+C_{13})\thetazeronenu,\\
    C_{22}&=3\bfcn(2,s){\crconstant}(1+2|A_\mu|)+ \bfcn(2,s){\crconstant}(C_{10}+C_{14})\thetazeronenu,
\end{aligned}
\end{equation}
and $C_\mathcal{R}$ in \eqref{CR}, $C_9$ and $C_{10}$ in \eqref{C9C10}, $C_{13}$ and $C_{14}$ in \eqref{C13C14} were defined previously.

Next, we split $\omega=\omega_0+\omega_1+\omega_{\geq2}$ as in \eqref{omegasplit} and use the bound \eqref{Sestimates}$_2$ to obtain 
\begin{equation*}
    \begin{aligned}
        \|\mathcal{S}(\omega)\|_{\fonesonenu}&\leq |A_\rho|\frac{L(t)}{\pi}(C_3+C_4)e^{\nu(t)}\thetazeronenu\thetaonesonenu
        +
        C_3\thetazeronenu^2\|\omega_1\|_{\fonesonenu}\\
        &\quad+C_3\thetazeronenu^2\|\omega_{\geq2}\|_{\fonesonenu}
        +C_4\thetazeronenu\thetaonesonenu\|\omega_1\|_{\fzeronenu}\\
        &\quad+C_4\thetazeronenu\thetaonesonenu\|\omega_{\geq2}\|_{\fzeronenu}.
    \end{aligned}
\end{equation*}
So again we introduce the vorticity estimates \eqref{omega1f0}, \eqref{omega1fss}, \eqref{omega2f0}, and \eqref{omega2fss}, and order the terms in $A_\rho$ and $A_\sigma$, to obtain 
\begin{equation*}
    \begin{aligned}
        \|\mathcal{S}&(\omega)\|_{\fonesonenu}\leq A_\sigma\frac{4\pi}{L(t)}\Big( C_3\thetazeronenu^2\thetathreesonenu+|A_\mu|C_3C_{13}\thetazeronenu^3\thetathreesonenu \\
        &\quad+C_4\thetazeronenu\thetaonesonenu\thetatwoonenu+|A_\mu|C_9C_4\thetazeronenu^2\thetaonesonenu \thetatwoonenu\Big)\\
        &\quad+|A_\rho|\frac{L(t)}{\pi}e^{\nu(t)}\Big( (C_3\!+\!C_4)\thetazeronenu\thetaonesonenu\\
        &\quad+2\bfcn(2,s) (1\!+\!2|A_\mu|)C_3\thetazeronenu^2\thetaonesonenu+C_3C_{14}\thetazeronenu^3\thetaonesonenu\\
        &\quad+C_4(1+2|A_\mu|)\thetazeronenu^2\thetaonesonenu+C_4C_{10}\thetazeronenu^3\thetaonesonenu\Big).
    \end{aligned}
\end{equation*}
Applying interpolation \eqref{interpolation} and the  embedding \eqref{embed}, 
we obtain that
\begin{multline}\label{est2}
        \|\mathcal{S}(\omega)\|_{\fonesonenu}\leq A_\sigma\frac{4\pi}{L(t)}C_{23}\thetazeronenu^2\thetathreesonenu
        \\
          +|A_\rho|\frac{L(t)}{\pi}e^{\nu(t)}C_{24}\thetazeronenu\thetaonesonenu
\end{multline}
where
\begin{equation}\label{C23C24}
    \begin{aligned}
        C_{23}&=C_3+C_4+|A_\mu|(C_3C_{13}+C_4C_9)\thetazeronenu,\\
        C_{24}&=C_3+C_4+(1+2|A_\mu|)(2\bfcn(2,s)C_3+C_4)\thetazeronenu
        \\
        & \quad +(C_3C_{14} +C_4C_{10})\thetazeronenu^2,
    \end{aligned}
\end{equation}
with $C_3$ and $C_4$ in \eqref{C3}, $C_9$ and $C_{10}$ in \eqref{C9C10}, $C_{13}$ and $C_{14}$ in \eqref{C13C14} previously defined.     

Finally, we combine estimates \eqref{est1} and \eqref{est2} together with vorticity estimate \eqref{omega2fss} to close the bound \eqref{N1f0}.  We obtain
\begin{equation}\label{N1final}
        \|N_1\|_{\fsonenu}\leq 
        A_\sigma\Big(\frac{2\pi}{L(t)}\Big)^2
        C_{25}\thetazeronenu\thetathreesonenu+|A_\rho|e^{\nu(t)}
        C_{26}\thetazeronenu\thetaonesonenu,
\end{equation}
where
\begin{equation}\label{C25C26}
    \begin{aligned}
        C_{25}=|A_\mu| C_{13}+C_{21}+C_{23}\thetazeronenu,\qquad       C_{26}=C_{14}+C_{22}+C_{24}.
    \end{aligned}
\end{equation}
Above  $C_{13}$ and $C_{14}$ are defined in \eqref{C13C14}, $C_{21}$ and $C_{22}$ in \eqref{C21C22}, and $C_{23}$ and $C_{24}$ in \eqref{C23C24}.  This completes the estimate for $N_1$.

We proceed now to estimate the term $N_2$ from \eqref{N1N2N3}.  For simplicity in this case we perform the estimate for $0\le s \le 1$.  The case $s>1$ would add several complications to the notation, but can also be proven similarly, as we will see in the following.  For $s\ge 0$ we use \eqref{bfcn.def} and \eqref{s.inequality} to obtain
\begin{equation}\label{N2aux}
\begin{aligned}
    \|N_2\|_{\fsonenu}\leq \|T_{\geq2}\|_{\fsonenu}(1+\bfcn(2,s) \thetaonenu)+\bfcn(2,s)\|T_{\geq2}\|_{\fzeronenu}\thetaonesonenu.
\end{aligned}
\end{equation}
We recall that for a function $f(\alpha)$ one has \eqref{fourierCalcAlpha} and then for $s\ge 0$ we have
\begin{equation}\label{negative.norm.bound}
    \begin{aligned}
        \left|\left|\int_0^\alpha f(\eta)d\eta-\frac{\alpha}{2\pi}\int_{-\pi}^\pi f(\eta)d\eta\right|\right|_{\fsonenu}&\leq \left( 1 + 1_{s=0}\right)\|f-\widehat{f}(0)\|_{\dot{\mathcal{F}}^{-1+s,1}_\nu} \\
        &\leq C(s)\|f\|_{\fsoneplusnu},
    \end{aligned}
\end{equation}
where $(s-1)^+ = s-1$ if $s\ge 1$ and $(s-1)^+ = 0$ if $s \le 1$ as usual.  We also define
\begin{equation}\label{def.CS}
    C(s) = \left( 1 + 1_{s=0}\right).
\end{equation}
If we performed the estimate below for $s>1$ we would need work with the norm of $\fsoneplusnu$ instead of the simpler $\fzeronenu$.  Thus we only do $0 \le s \le 1$ for the $N_2$ term.

Therefore, for $0 \le s \le 1$ recalling the expression for $T_{\geq2}$ in \eqref{Tsplit}, we have
\begin{equation*}
    \begin{aligned}
    \|T_{\geq2}\|_{\fsonenu}&\leq C(s)\|(1+\theta_\alpha)U_{\geq2}\|_{\fzeronenu}+C(s)\|\theta_\alpha U_{1}\|_{\fzeronenu}\\
    &\leq C(s)(1+\thetaonenu)\|U_{\geq2}\|_{\fzeronenu}+C(s)\thetaonenu\|U_1\|_{\fzeronenu}.
    \end{aligned}
\end{equation*}
Then introducing $U_{\geq2}$ and $U_1$  from\eqref{Usplit} we obtain that
\begin{equation*}
    \begin{aligned}
    \|T_{\geq2}\|_{\fsonenu}&\leq  C(s)\frac{\pi}{L(t)}(1+\thetaonenu)\big(\|\omega_{\geq2}\|_{\fzeronenu}+\|\mathcal{R}(\omega_{\geq1})\|_{\fzeronenu}+\|\mathcal{S}(\omega)\|_{\fzeronenu}\big)
    \\
    &\quad+C(s)\frac{\pi}{L(t)}\thetaonenu\big(\|\omega_{1}\|_{\fzeronenu}+\|\mathcal{R}(\omega_{0})\|_{\fzeronenu}\big).
    \end{aligned}
\end{equation*}
We split the vorticity terms, $\omega_{\geq1} = \omega_{1}+ \omega_{\geq2}$, and use $\omega_0$ in \eqref{omegasplit}.  We further use  the estimates for $\mathcal{R}$ and $\mathcal{S}$ in \eqref{Restimates} and \eqref{Sestimates} and use \eqref{cosine.calc.FT} to obtain
\begin{equation*}
    \begin{aligned}
    \|T_{\geq2}\|_{\fsonenu}&\leq  C(s)\frac{\pi}{L(t)}(1+\thetaonenu)\Big((1+{\crconstant}\thetazeronenu+C_1\thetazeronenu^2)\|\omega_{\geq2}\|_{\fzeronenu}\\
    &\quad+({\crconstant}+C_1\thetazeronenu)\thetazeronenu\|\omega_{1}\|_{\fzeronenu}+A_\rho\frac{L(t)}{\pi}e^{\nu(t)}C_1\thetazeronenu^2\Big)
    \\
    &\quad+C(s)\frac{\pi}{L(t)}\thetaonenu\Big(\|\omega_{1}\|_{\fzeronenu}+A_\rho\frac{L(t)}{\pi}e^{\nu(t)}{\crconstant}\thetazeronenu\Big).
    \end{aligned}
\end{equation*}
Then we introduce the vorticity estimates \eqref{vorticityestimatess} and \eqref{omega2f0} to obtain 
\begin{equation*}
    \begin{aligned}    \|T_{\geq2}&\|_{\fsonenu}\leq  
    A_\sigma\Big(\frac{2\pi}{L(t)}\Big)^2C(s)
\Big(|A_\mu|C_9\big(1+C_\mathcal{R}\thetazeronenu+C_1\thetazeronenu^2\big)\\
&\quad+(C_\mathcal{R}+C_1\thetazeronenu)\Big)(1+\thetaonenu)\thetazeronenu\thetatwoonenu
\\
&\quad+A_\sigma\Big(\frac{2\pi}{L(t)}\Big)^2 C(s)\thetaonenu\thetatwoonenu \!+\!|A_\rho|C(s)\Big(C_{10}(1\!+\!C_\mathcal{R}\thetazeronenu\!+\!C_1\thetazeronenu^2)\\
&\quad+(1+2|A_\mu|)(C_\mathcal{R}+C_1\thetazeronenu)+C_1\Big)e^{\nu(t)}(1+\thetaonenu)\thetazeronenu^2\\
&\quad
+|A_\rho|C_\mathcal{R}C(s)e^{\nu(t)}\thetaonenu\thetazeronenu
+(1+2|A_\mu|)|A_\rho|C(s)e^{\nu(t)}\thetaonenu\thetazeronenu.
    \end{aligned}
\end{equation*}
The above expression reduces to
\begin{equation}\label{T2aux}
    \begin{aligned}    \|T_{\geq2}\|_{\fsonenu}&\leq  
    A_\sigma\Big(\frac{2\pi}{L(t)}\Big)^2C_{33}\thetaonenu\thetatwoonenu+|A_\rho|e^{\nu(t)}C_{34}\thetazeronenu\thetaonenu,
    \end{aligned}
\end{equation}
where
\begin{equation}\label{C33C34}
    \begin{aligned}    
    C_{33}&=C(s)+C(s)\Big(|A_\mu|C_9\big(1+C_\mathcal{R}\thetazeronenu+C_1\thetazeronenu^2\big)\\
    &\qquad+(C_\mathcal{R}+C_1\thetazeronenu) \Big)(1+\thetazeronenu),\\
    C_{34}&=C(s)\Big(C_{10}(1+C_\mathcal{R}\thetazeronenu+C_1\thetazeronenu^2)\\
    &\quad+(1+2|A_\mu|)(C_\mathcal{R}+C_1\thetazeronenu)+C_1\Big)(1+\thetazeronenu)
    \\
    &\quad+C(s)C_\mathcal{R}+C(s)(1+2|A_\mu|),
    \end{aligned}
\end{equation}
and $C_\mathcal{R}$ is defined in \eqref{CR}, $C(s)$ in \eqref{def.CS}, $C_1$ in \eqref{C1}, $C_9$ and  $C_{10}$ \eqref{C9C10}.

Then, going back to \eqref{N2aux}, we find that 
\begin{equation}\notag
    \begin{aligned}
        \|N_2\|_{\fsonenu}&\leq \Big(A_\sigma\Big(\frac{2\pi}{L(t)}\Big)^2C_{33}\thetaonenu\thetatwoonenu\\
        &\quad+|A_\rho|e^{\nu(t)}C_{34}\thetazeronenu\thetaonenu\Big)(1+\thetaonenu+\thetaonesonenu)\\
        &\leq A_\sigma\Big(\frac{2\pi}{L(t)}\Big)^2C_{33}\thetaonenu \thetatwoonenu
        \left( 2\thetaonesonenu+ 1\right)\\
        &\quad+|A_\rho|e^{\nu(t)}C_{34}\thetazeronenu\thetaonenu\left( 2\thetaonesonenu+ 1\right).
    \end{aligned}
\end{equation}
Finally, interpolation \eqref{interpolation} gives that
\begin{equation*}
    \begin{aligned}
        \thetaonenu\thetaonesonenu\thetatwoonenu\leq \thetasonenu^{\frac{5+2s}{3}}\thetathreesonenu^{\frac{4-2s}{3}},
    \end{aligned}
\end{equation*}
Thus we have
\begin{equation}\label{N2bound}
    \begin{aligned}
        \|N_2\|_{\fsonenu}
        &\leq A_\sigma\Big(\frac{2\pi}{L(t)}\Big)^2C_{33}\left( 2\thetasonenu^{\frac{5+2s}{3}}\thetathreesonenu^{\frac{4-2s}{3}}+ \thetasonenu\thetathreesonenu \right)\\
        &\quad+|A_\rho|e^{\nu(t)}C_{34}\ \left( 2\thetazeronenu\thetaonenu\thetaonesonenu+ \thetazeronenu\thetaonenu\right).
    \end{aligned}
\end{equation}
This completes our estimates for $N_2$.

Lastly, the estimate for $N_3$ in \eqref{N1N2N3} for $s\ge 0$ is obtained with \eqref{bfcn.def} as follows
\begin{equation*}
    \begin{aligned}
        \|N_3\|_{\fsonenu}\leq \bfcn(2,s) \left(\|T_1\|_{\fsonenu}\thetaonenu+\|T_1\|_{\fzeronenu}\thetaonesonenu \right).
    \end{aligned}
\end{equation*}
The bound for $T_1$ in \eqref{Tsplit} for $s\ge 0$ from \eqref{negative.norm.bound} using \eqref{u0t0} and \eqref{cosine.calc.FT} is 
\begin{equation}\label{T1aux}
    \begin{aligned}
        \|T_1\|_{\fsonenu}&\leq \|U_1\|_{\fsoneplusnu}+\|\theta_\alpha U_0\|_{\fsoneplusnu}
        \\
        &\leq \|U_1\|_{\fsoneplusnu}+2|A_{\rho}| \bfcn(2,s) e^{\nu(t)} \|\theta_\alpha\|_{\fsoneplusnu}
        \\
        &\leq \frac{\pi}{L(t)}\big(\|\omega_1\|_{\fsoneplusnu}
        +\|\mathcal{R}(\omega_0)\|_{\fsoneplusnu}\big)
        \\
        &\quad +2|A_{\rho}| \bfcn(2,s) e^{\nu(t)} \|\theta_\alpha\|_{\fsoneplusnu}.
    \end{aligned}
\end{equation}
We use \eqref{omega1f0}, \eqref{omega1fss} and \eqref{Restimates} with \eqref{omegasplit} and  \eqref{cosine.calc.FT} to obtain
\begin{multline*}
\frac{\pi}{L(t)}\big(\|\omega_1\|_{\fsoneplusnu}
        +\|\mathcal{R}(\omega_0)\|_{\fsoneplusnu}\big)
        \\
        \leq 
        A_\sigma\Big(\frac{2\pi}{L(t)}\Big)^2
        \|\theta\|_{\dot{\mathcal{F}}^{2+(s-1)^+,1}_\nu}
        +2\bfcn(2,s)|A_\rho|\Big((1+2|A_\mu|)+2 C_\mathcal{R}\Big)e^{\nu(t)}
        \|\theta\|_{\fsoneplusnu}.
\end{multline*}
Thus
\begin{equation}\label{N3boud}
    \begin{aligned}
        \|N_3\|_{\fsonenu}&\leq A_\sigma\Big(\frac{2\pi}{L(t)}\Big)^2 
        \bfcn(2,s)\|\theta\|_{\dot{\mathcal{F}}^{2+(s-1)^+,1}_\nu}\thetaonesonenu\\
        &\quad+|A_\rho|e^{\nu(t)}C_3^{\prime}\|\theta\|_{\fsoneplusnu}\thetaonesonenu.
    \end{aligned}
\end{equation}
Here
\begin{equation}\label{C3prime.prime}
        C_3^{\prime}= 2\bfcn(2,s)^2\left(1+(1+2|A_\mu|)+2 C_\mathcal{R} \right).
\end{equation}
This completes our estimates for $N_3$.  

In summary, we can combine the bounds \eqref{N1final}, \eqref{N2bound}, and \eqref{N3boud} to prove the following theorem.

\begin{thm}\label{thm:nonLinear}
For $0\le s \le 1$ we have the estimate for $N$ from \eqref{system} as
\begin{equation}\label{thm:Nbound}
    \begin{aligned}
        \|N\|_{\fsonenu}&\leq A_\sigma\Big(\frac{2\pi}{L(t)}\Big)^2C_1(N)\thetasonenu\thetathreesonenu\\
        &\quad+
        A_\sigma\Big(\frac{2\pi}{L(t)}\Big)^2 2C_{33}\thetasonenu^{\frac{5+2s}{3}}\thetathreesonenu^{\frac{4-2s}{3}}\\
        &\quad+|A_\rho|e^{\nu(t)}C_2(N)\thetazeronenu\thetaonesonenu,
    \end{aligned}
\end{equation}
where 
\begin{equation}\label{C1NC2N}
    \begin{aligned}
        C_1(N)&=C_{25}+C_{33}+b(2,s),\\
        C_2(N)&=C_{26}+C_{34}(1+2\thetazeronenu)+C_{3}^\prime.
    \end{aligned}
\end{equation}
Further $C_{25}$ and $C_{26}$ are defined in \eqref{C25C26}, $C_{33}$ and $C_{34}$ in \eqref{C33C34}, and $C_{3}^\prime$ is previously defined in \eqref{C3prime.prime}.
\end{thm}

Plugging in $s=1/2$ in the bound for the nonlinear term in Theorem \ref{thm:nonLinear} in  \eqref{thm:Nbound},
we find that
\begin{equation}\label{Nbound}
    \begin{aligned}
        \|N\|_{\fhonenu}&\leq A_\sigma\Big(\frac{2\pi}{L(t)}\Big)^2C_{35}\thetahonenu\thetasevenhonenu
        +|A_\rho|e^{\nu(t)}C_{36}\thetahonenu\thetasevenhonenu,
    \end{aligned}
\end{equation}
where
\begin{equation}\label{C35C36}
    \begin{aligned}
        C_{35}&=C_{35}(\thetahonenu)=C_1(N)+2C_{33}\thetahonenu,\\
        C_{36}&=C_{36}(\thetahonenu)=C_2(N),
    \end{aligned}
\end{equation}
and $C_1(N)$ and $C_2(N)$ are defined in \eqref{C1NC2N}, and  $C_{33}$ is defined in \eqref{C33C34}. Notice that in the definition of $C_{35}$ and $C_{36}$ we can evaluated all the previous functions $C_i$ in the norm $\thetahonenu$ instead of $\thetazeronenu$, which could be done due to \eqref{embed} and the fact that these $C_i$ are increasing functions of the norm.

\section[Regularization scheme and completion of the proof]{Regularization scheme and completion of the proof of Theorem \ref{thm:global}}\label{subregular}

We will now put all the pieces together to complete the proof of Theorem \ref{thm:global}.

\begin{proof}[Proof of Theorem \ref{thm:global}]
With all of the previous developments, the proof then follows a standard regularization argument. Recall the high-frequency cut-off operator $\mathcal{J}_N$ by \eqref{CutOffHigh}.  Denote $f_N=\mathcal{J}_N f$ and consider the regularized version of system \eqref{finalsystem}:
\begin{equation*}
    \begin{aligned}
     (\vartheta_N)_t&=\mathcal{J}_N\big(\frac{2\pi}{L_N(t)} (U_N)_\alpha+\frac{2\pi}{L_N(t)}T_N(1+(\vartheta_N)_\alpha)\Big),\\
     \frac{L_N(t)}{2\pi}&=R\Big(1+\frac{1}{2\pi}\imag \int_{-\pi}^\pi\int_0^\alpha e^{i(\alpha-\eta)} \sum_{n\geq1}\frac{i^n}{n!}(\theta_N(\alpha)-\theta_N(\eta))^n d\eta d\alpha\Big)^{-\frac12},\\
     0&=\int_{-\pi}^{\pi} e^{i(\alpha+\theta_N(\alpha))}d\alpha.
    \end{aligned}
\end{equation*}
We abused notation in the definition of $L_N(t)$ above since we are not using $\mathcal{J}_N$ from \eqref{CutOffHigh}. Solving the last constraint by the implicit function theorem (see Proposition \ref{IFTprop} in Chapter \ref{IFTSection}) gives $F(\theta_N)=(\real\hat{\theta}(1), \imag \hat{\theta}(1))$ which can be solved for $\widehat{\theta}(\pm 1)$.   
Thus substituting in as well the expression for $L_N(t)$, we obtain one equation for $\varphi_N=\widehat{\vartheta}(0)+\mathcal{F}^{-1}(1_{|k|\neq1}\widehat{\theta_N}(k))$. We thereby have the system written as an ODE of the form 
\begin{equation*}
    \dot{\varphi}_N=\mathcal{J}_N \mathcal{G}\big(\varphi_N\big),\qquad \varphi_N(0)=\widehat{\vartheta}_0(0)+\mathcal{F}^{-1}(1_{|k|\neq1}\widehat{\theta_{N,0}}(k)),
\end{equation*}
for a certain nonlinear function $\mathcal{G}$.  Here $\theta_{N,0}=\mathcal{J}_N\theta_{0}$ is the initial condition.  Therefore, Picard's theorem on Banach spaces yields the local existence of regularized solutions $\varphi_N\in C^1([0,T_N);H_N^m)$, where the space $H_N^m$ is defined by $H_N^m=\{f\in H^m(\mathbb{T}): \text{supp}(\widehat{f})\subset [-N,N]\}.$ Furthermore the \textit{a priori} estimates in Section \ref{subsecGlobal}, in particular the energy balance \eqref{estimatef12} and \eqref{zeroBoundUniform}, hold for the regularized system, which provides uniform bounds for $\varphi_N$ in the space $L^\infty(\mathbb{R}_+; \dot{\mathcal{F}}^{\frac12,1}_\nu)\cap L^1(\mathbb{R}_+; \dot{\mathcal{F}}^{\frac72,1}_\nu)$. We will next use the following  version of the Aubin-Lions lemma \cite[Corollary 6]{MR916688}:
\begin{lemma}[Aubin-Lions' Lemma] 
Let $X_0$, $X$, and $X_1$ be Banach spaces such that
$$X_0\subset X\subset X_1,$$
with compact embedding $X_0\hookrightarrow X$, and let $p\in(1,\infty]$. Let $G$ be bounded in $L^p([0,T];X)\cap L^1_{\text{loc}}([0,T];X_0)$,
and  $\partial_t G$ be  bounded in $L^1_{\text{loc}}([0,T];X_1)$.
 Then $G$ is relatively compact in $L^q([0,T];X)$, $q\in[1,p)$.
\end{lemma}
Letting $X_0=\dot{\mathcal{F}}^{\frac72,1}_\nu$, $X_1=\mathcal{F}^{0,1}_\nu$, and $X=\dot{\mathcal{F}}^{\frac12,1}_\nu$, we get the strong convergence to the full system, up to a subsequence, of the approximated problems in $L^2([0,T];\dot{\mathcal{F}}^{\frac12,1}_\nu)$. Next, since $\widehat{\varphi}_N(n,t)\rightarrow \widehat{\varphi}(n,t)$ as $N\rightarrow \infty$ for all $n\in\mathbb{Z}$ and almost every $t$, recalling that $\widehat{\varphi}(n,t)=\widehat{\theta}(n,t)$ for $|n| \geq2$, Fatou's lemma gives that
\begin{equation*}
    \begin{aligned}
        	M(t)&=\thetahonenu(t)+ \frac{A_\sigma}{R^3}  \dissconstant \int_0^t \thetasevenhonenu(\tau) d\tau \\
        	&\leq\liminf_{N\rightarrow +\infty}\Big(\|\theta_N\|_{\dot{\mathcal{F}}^{\frac12,1}_\nu}(t)+ \frac{A_\sigma}{R^3}  \dissconstant \int_0^t \|\theta_N\|_{\dot{\mathcal{F}}^{\frac72,1}}(\tau) d\tau\Big) \\
        	&\leq C_S^2\|\theta_0\|_{\fhone}.
        	\end{aligned}
\end{equation*}
Thus we obtain that the limit function $\theta$ satisfies
\begin{equation*}
    \theta \in L^\infty([0,T];\dot{\mathcal{F}}^{\frac12,1}_\nu)\cap L^1([0,T];\dot{\mathcal{F}}^{\frac72,1}_\nu).
\end{equation*}
Since the equation for $\widehat{\varphi}(0)$ is decoupled from the rest and its right-hand side only depends on $\theta$, we conclude the strong convergence of $\varphi$ in $L^\infty([0,T];\dot{\mathcal{F}}^{\frac12,1}_\nu)\cap L^1([0,T];\dot{\mathcal{F}}^{\frac72,1}_\nu)$. 
 We refer to Section 5 of \cite{GG-BS2019} and \cite{2009.03360} for further details of such an approximation argument, in particular for the instant generation of analyticity and the continuity in time.
\end{proof}

\chapter{Uniqueness}\label{sec:uniqueness}

In this last chapter we will prove uniqueness in $\fhone$ of solutions to \eqref{finalsystem} with initial data of the size given by the constraint in \eqref{condition}.  In particular the main result of this chapter is Theorem \ref{uniquenessproposition} just below.  To prove this theorem, in Section \ref{sec:unique.length} we prove the required estimates for the differences of the lengths.  Then in Section \ref{subsec:uniq.vort} we prove the estimates on the differences of the vorticity strength.  After that in Section \ref{subsec:uniq.nonlinear} we prove the main estimates on the differences of the non-linear terms.  Lastly in Section \ref{subsec:uniq.proof} we collect all the previous estimates to prove the uniqueness of the solutions to \eqref{system} as in Theorem \ref{uniquenessproposition}.

Throughout the proof of Theorem \ref{uniquenessproposition}, we define coefficients that will be used in the rest of this chapter.

\begin{defn}\label{uniquenessimplicitdef}
We use the symbol $\timeintE>0$ to denote any coefficient that is integrable in time and may depend upon $\|\theta_{1},\theta_{2}\|_{\fhone}$, recalling \eqref{max.fcns}, which is bounded and  $\|\theta_{1},\theta_{2}\|_{\fsevenhone}$ which is time integrable.

The symbol $\diffusint>0$ will denote any coefficient that is bounded and can depend upon $\|\theta_{1},\theta_{2}\|_{\fhone}$.  

The symbol $\diffusint_{s}>0$ will denote any coefficient that is bounded and can depend upon $\|\theta_{1},\theta_{2}\|_{\fsone}$.  

The symbol $\timeintE_{s}>0$ will denote any coefficient that may depend upon $\diffusint\|\theta_{1},\theta_{2}\|_{\dot{\mathcal{F}}^{2+s,1}}$.
\end{defn}

\begin{thm}
\label{uniquenessproposition}
Consider two solutions $\vartheta_{1}$ and $\vartheta_{2}$ of \eqref{system} with the same initial data satisfying the medium size condition, as in Theorem \ref{thm:global}.   Then these solutions satisfy the following differential inequality
\begin{multline}\label{uniqineq}
\frac{d}{dt}\Big(|\widehat{\vartheta}_{1}(0) - \widehat{\vartheta}_{2}(0)| + \|\theta_{1}-\theta_{2}\|_{\fhone}\Big) 
\\
\leq (\diffusint+\timeintE)(|\widehat{\vartheta}_{1}(0) - \widehat{\vartheta}_{2}(0)| + \|\theta_{1}-\theta_{2}\|_{\fhone}).
\end{multline}
\end{thm}

With Theorem \ref{uniquenessproposition}, we can conclude by Gronwall's inequality that for any $T>0$, we have
\begin{multline*}
 \Big(|\widehat{\vartheta}_{1}(0) - \widehat{\vartheta}_{2}(0)| + \|\theta_{1}-\theta_{2}\|_{\fhone}\Big) \Big|_{t=T} 
 \\
 \leq \exp\left(\int_{0}^{T} (\diffusint+\timeintE) dt\right)
 \Big(|\widehat{\vartheta}_{1}(0) - \widehat{\vartheta}_{2}(0)| + \|\theta_{1}-\theta_{2}\|_{\fhone}\Big)\Big|_{t=0}= 0.
 \end{multline*}
This holds since $\timeintE=\diffusint\|\theta_{1},\theta_{2}\|_{\fsevenhone}$ with \eqref{max.fcns}, is time integrable by Theorem \ref{thm:global}.

The rest of this chapter is devoted to proving Theorem \ref{uniquenessproposition}.  Given two solutions $\vartheta_{1}(\alpha,t)$ and $\vartheta_{2}(\alpha,t)$ with initial data $\vartheta_{1}(\alpha,0) = \vartheta_{2}(\alpha,0)$ that satisfies \eqref{condition}; their respective evolution equations are given by \eqref{system} as follows
$$(\vartheta_{i})_t(\alpha)=\frac{2\pi}{L_{i}(t)}\Big(\mathcal{L}_{i}(\alpha)+N_{i}(\alpha)\Big),$$
where the vorticity terms are denoted by $\omega_{1}$ and $\omega_{2}$ respectively. The evolution of $\vartheta_{1} - \vartheta_{2}$ is then given by
\begin{align}\notag
 (\vartheta_{1}-\vartheta_{2})_t(\alpha) 
 &= \Big(\frac{2\pi}{L_{1}(t)} -\frac{2\pi}{L_{2}(t)}\Big) \Big(\mathcal{L}_{1}(\alpha)+N_{1}(\alpha)\Big) + \frac{2\pi}{L_{2}(t)}(\mathcal{L}_{1}(\alpha) - \mathcal{L}_{2}(\alpha))  \\&\hspace{2in}+ \frac{2\pi}{L_{2}(t)}(N_{1}(\alpha) - N_{2}(\alpha)). \label{theta1theta2}
\end{align}
Using the evolution equation \eqref{theta1theta2}, 
we will prove \eqref{uniqineq}. In Proposition \ref{lengthdifferenceprop}, we give an estimate to control the length difference in the first term on the right hand side of \eqref{theta1theta2}. By the estimates from Section \ref{sec:NonLinearEst}, the coefficient, $\mathcal{L}_{1}(\alpha)+N_{1}(\alpha)$, of the length difference in the first term is bounded by $\diffusint+\timeintE$.  The bound on the length difference is shown in in Proposition \ref{lengthdifferenceprop}. The second term on the RHS of \eqref{theta1theta2} gives a linear coercive estimate for the time evolution. Lastly, Section \ref{subsec:uniq.nonlinear} is dedicated to controlling the third terms on the RHS of \eqref{theta1theta2} using the idea of Proposition \ref{uniquenessprop} and the non-linear estimates as in Section \ref{sec:NonLinearEst}.

Additionally, we will use the following idea repeatedly:
\begin{prop}\label{uniquenessprop}
Consider two functions $f$ and $g$ in $\fsone$ for some $s \geq 0$.  We also consider some operator $T$. Then, for any $n \in \mathbb{N}$ we have
\begin{multline}\label{uniqueoperator}
\|   f(\alpha)^{n}Tf(\alpha)  - g(\alpha)^{n} Tg(\alpha)\|_{\fzerone} \\ \leq \|f\|_{\fzerone}^{n}\|Tf-Tg\|_{\fzerone} + \Big(\sum_{k=0}^{n-1}\|f\|_{\fzerone}^{n-k-1}\|g\|_{\fzerone}^{k}\Big)\|f-g\|_{\fzerone}\|Tg\|_{\fzerone}.
\end{multline}
For $s>0$, we have
\begin{multline}\label{uniqueoperatorsnorm}
\|   f(\alpha)^{n}Tf(\alpha)  - g(\alpha)^{n} Tg(\alpha)\|_{\fsone} \\ \leq b(n+1,s)\Big(\|f\|_{\fzerone}^{n}\|Tf-Tg\|_{\fsone}  + n\|f\|_{\fzerone}^{n-1}\|f\|_{\fsone}\|Tf-Tg\|_{\fzerone} \\+  n\|f,g\|_{\fzerone}^{n-1}\|f-g\|_{\fsone}\|Tg\|_{\fzerone} + n\|f,g\|_{\fzerone}^{n-1}\|f-g\|_{\fzerone}\|Tg\|_{\fsone} \\+ n(n-1)\|f,g\|_{\fzerone}^{n-2}\|f,g\|_{\fsone}\|f-g\|_{\fzerone}\|Tg\|_{\fzerone}\Big),
\end{multline}
where we recall the definition \eqref{max.fcns}.  In the special case where $T = \frac{d}{d\alpha^{j}}$ for some $j\in \mathbb{N}$, we obtain
\begin{multline}\label{uniquenessestimate}
\|   f(\alpha)^{n}\frac{d^{j}}{d\alpha^{j}}f(\alpha)  - g(\alpha)^{n} \frac{d^{j}}{d\alpha^{j}}g(\alpha)\|_{\fzerone} \\ \leq \|f\|_{\fzerone}^{n}\|f-g\|_{\mathcal{F}^{j,1}} + \Big(\sum_{k=0}^{n-1}\|f\|_{\fzerone}^{n-k-1}\|g\|_{\fzerone}^{k}\Big)\|f-g\|_{\fzerone}\|g\|_{\mathcal{F}^{j,1}}.
\end{multline}
\end{prop}

\begin{proof}
Since
\begin{align*}
 f(\alpha)^{n}Tf(\alpha)  - g(\alpha)^{n} Tg(\alpha) &=   f(\alpha)^{n}Tf(\alpha)  - f(\alpha)^{n} Tg(\alpha) \\&\hspace{0.5in}+  f(\alpha)^{n}Tg(\alpha)  - g(\alpha)^{n} Tg(\alpha)\\ &= f(\alpha)^{n}(Tf(\alpha)  -  Tg(\alpha)) +  (f(\alpha)^{n}-g(\alpha)^{n})Tg(\alpha).
\end{align*}
We obtain
\begin{multline*}
\| f(\alpha)^{n}Tf(\alpha)  - g(\alpha)^{n} Tg(\alpha)\|_{\fzerone} \\ \leq \|f\|_{\fzerone}^{n}\|Tf-Tg\|_{\fzerone} + \|f(\alpha)^{n} - g(\alpha)^{n}\|_{\fzerone}\|Tg\|_{\fzerone}.
\end{multline*}
Next, we have
\begin{align*}
\|f(\alpha)^{n} - g(\alpha)^{n}\|_{\fzerone} &= \Big\|\sum_{k=0}^{n-1}f(\alpha)^{n-k} g(\alpha)^{k} -f(\alpha)^{n-k-1} g(\alpha)^{k+1} \Big\|_{\fzerone}\\ &\leq \sum_{k=0}^{n-1}\|f\|_{\fzerone}^{n-k-1}\|g\|_{\fzerone}^{k}\|f-g\|_{\fzerone}.
\end{align*}
This yields \eqref{uniqueoperator}.  Then \eqref{uniqueoperatorsnorm} is proven similarly.
\end{proof}

\section{Estimates for the differences of the lengths}\label{sec:unique.length}

We need to control the difference in the lengths $L_{1}(t)$ and $L_{2}(t)$, for example, to control the first term on the right hand side of \eqref{theta1theta2}. In this section, we prove the following proposition on the differences of the lengths.

\begin{prop}\label{lengthdifferenceprop}
Consider the lengths, $L_{1}(t)$ and $L_{2}(t)$, of two solutions, $\vartheta_{1}$ and $\vartheta_{2}$ respectively, to \eqref{system} as defined by \eqref{modul}. Then, we have
\begin{equation}\label{lengthdifference}
\left|\frac{L_{1}(t)}{2\pi}-\frac{L_{2}(t)}{2\pi}\right| \leq  C_L \|\theta_{1}-\theta_{2}\|_{\fzerone}
\end{equation}
with $C_{L}$ defined by \eqref{CL}.  
\end{prop}

\begin{proof}
We recall equation \eqref{Lequation}.  Thus, denoting $\Delta f=f(\alpha)-f(\eta)$, we have that 
\begin{multline}\label{lenghtaux}
    \Big(\frac{L_1(t)}{2\pi}\Big)^2-\Big(\frac{L_2(t)}{2\pi}\Big)^2
    \\
    =\frac{\frac{R^2}{2\pi}\Big(\imag \int\limits_{-\pi}\limits^\pi\!\int\limits_{0}\limits^\alpha e^{i(\alpha-\eta)} \sum\limits_{n\geq1}\frac{i^n}{n!}(\Delta\theta_2)^n d\eta d\alpha\!-\!\imag \int\limits_{-\pi}\limits^\pi\!\int\limits_{0}\limits^\alpha e^{i(\alpha-\eta)} \sum\limits_{n\geq1}\frac{i^n}{n!}(\Delta\theta_1)^n d\eta d\alpha\Big)}{\Big(1\!+\!\imag \!\int\limits_{-\pi}\limits^\pi\!\int\limits_{0}\limits^\alpha e^{i(\alpha-\eta)}\!\! \sum\limits_{n\geq1}\frac{i^n}{n!}(\Delta\theta_1)^n d\eta \frac{d\alpha}{2\pi}\!\Big)\!\Big(1\!+\!\imag \!\int\limits_{-\pi}\limits^\pi\!\int\limits_{0}\limits^\alpha e^{i(\alpha-\eta)}\!\! \sum\limits_{n\geq1}\frac{i^n}{n!}(\Delta\theta_2)^n d\eta \frac{d\alpha}{2\pi}\!\Big)}.
\end{multline}
Recalling the estimates in \eqref{lengthauxbound} - \eqref{C37C38}, the denominator is bounded by
\begin{multline*}
\bigg(\prod_{m=1}^2\Big(\!1\!+\!\imag \!\int\limits_{-\pi}\limits^\pi\!\int\limits_{0}\limits^\alpha e^{i(\alpha-\eta)}\!\! \sum\limits_{n\geq1}\frac{i^n}{n!}(\Delta\theta_m)^n d\eta \frac{d\alpha}{2\pi}\!\Big)\!\bigg)^{-1}
\\
\leq 
\bigg(\prod_{m=1}^2
\Big(1-\frac{\pi}{2}\big(e^{2\|\theta_m\|_{\fzerone}}-1\big)\Big)
\bigg)^{-1}.
\end{multline*}
Further the numerator has the upper bound
\begin{equation*}
\begin{aligned}
&\imag \int\limits_{-\pi}\limits^\pi\!\int\limits_{0}\limits^\alpha e^{i(\alpha-\eta)} \sum\limits_{n\geq1}\frac{i^n}{n!}(\Delta\theta_2)^n d\eta d\alpha\!-\!\imag \int\limits_{-\pi}\limits^\pi\!\int\limits_{0}\limits^\alpha e^{i(\alpha-\eta)} \sum\limits_{n\geq1}\frac{i^n}{n!}(\Delta\theta_1)^n d\eta d\alpha\\
&=\imag \int\limits_{-\pi}\limits^\pi\!\int\limits_{0}\limits^\alpha e^{i(\alpha-\eta)} \sum\limits_{n\geq1}\frac{i^n}{n!}(\Delta\theta_2-\Delta\theta_1)\sum\limits_{m=0}\limits^{n-1}\big(\Delta\theta_1\big)^m\big(\Delta\theta_2\big)^{n-m-1} d\eta d\alpha\\
&\leq  \pi^2  \sum\limits_{n\geq1}\frac{2\|\theta_1-\theta_2\|_{\fzerone}}{n!}\sum\limits_{m=0}\limits^{n-1}\big(2\|\theta_1\|_{\fzerone}\big)^m\big(2\|\theta_2\|_{\fzerone}\big)^{n-m-1}.
\end{aligned}
\end{equation*}
Thus we further obtain the estimate
\begin{equation*}
\begin{aligned}
&
\left| \imag \int\limits_{-\pi}\limits^\pi\!\int\limits_{0}\limits^\alpha e^{i(\alpha-\eta)} \sum\limits_{n\geq1}\frac{i^n}{n!}(\Delta\theta_2)^n d\eta d\alpha\!-\!\imag \int\limits_{-\pi}\limits^\pi\!\int\limits_{0}\limits^\alpha e^{i(\alpha-\eta)} \sum\limits_{n\geq1}\frac{i^n}{n!}(\Delta\theta_1)^n d\eta d\alpha
\right|
\\
&\leq \frac{2\pi^2\|\theta_1-\theta_2\|_{\fzerone}}{2\|\theta_1\|_{\fzerone}-2\|\theta_2\|_{\fzerone}}\sum\limits_{n\geq1}\frac{1}{n!}\Big(\big(2\|\theta_1\|_{\fzerone}\big)^n-\big(2\|\theta_2\|_{\fzerone}\big)^n\Big)\\
&=2\pi^2\|\theta_1-\theta_2\|_{\fzerone}\frac{e^{2\|\theta_1\|_{\fzerone}}-e^{2\|\theta_2\|_{\fzerone}}}{2\|\theta_1\|_{\fzerone}-2\|\theta_2\|_{\fzerone}}.
\end{aligned}
\end{equation*}
We substitute this back into \eqref{lenghtaux} to obtain that
\begin{equation}\label{diflength}
\begin{aligned}
\Big|\Big(\frac{L_1(t)}{2\pi}\Big)^2-\Big(\frac{L_2(t)}{2\pi}\Big)^2\Big|\leq  R^2 C_{L,2} \|\theta_1-\theta_2\|_{\fzerone},
\end{aligned}
\end{equation}
where
\begin{equation}\notag
    \begin{aligned}
    C_{L,2}=
    \pi\frac{e^{2\|\theta_1\|_{\fzerone}}\!-\!e^{2\|\theta_2\|_{\fzerone}}}{2\|\theta_1\|_{\fzerone}\!-\!2\|\theta_2\|_{\fzerone}}
    \bigg(\prod_{m=1}^2
\Big(1-\frac\pi2\big(e^{2\|\theta_m\|_{\fzerone}}-1\big)\Big)
\bigg)^{-1}.
    \end{aligned}
\end{equation}
The estimate \eqref{diflength} allows to easily bound terms like $L_1(t)-L_2(t)$ or $L_1(t)^{-1}-L_2(t)^{-1}$. In fact, using \eqref{Lbound}, we obtain that
\begin{equation*}
    \begin{aligned}
    \Big|\frac{L_1(t)}{2\pi}-\frac{L_2(t)}{2\pi}\Big|=\frac{1}{\frac{L_1(t)}{2\pi}+\frac{L_2(t)}{2\pi}}\Big|\Big(\frac{L_1(t)}{2\pi}\Big)^2-\Big(\frac{L_2(t)}{2\pi}\Big)^2\Big|\leq  C_L\|\theta_1-\theta_2\|_{\fzerone},
    \end{aligned}
\end{equation*}
where, with $C_{37}$ is defined in \eqref{C37C38}, we have
\begin{equation}\label{CL}
    \begin{aligned}
    C_{L}=R\frac{C_{L,2}}{2C_{37}}.
    \end{aligned}
\end{equation}
This completes the proof of \eqref{lengthdifference}.
\end{proof}

\section{Estimates for the differences of the vorticity strength}\label{subsec:uniq.vort}
In this section we will estimate the differences of the vorticity strength terms.
We use the splitting in \eqref{omegasplit} as
\begin{multline*}
\omega_{1}(\alpha)-\omega_{2}(\alpha) 
\\
=
(\omega_{1})_{0}(\alpha)-(\omega_{2})_{0}(\alpha) + (\omega_{1})_{1}(\alpha)-(\omega_{2})_{1}(\alpha) + (\omega_{1})_{\geq 2}(\alpha)-(\omega_{2})_{\geq 2}(\alpha).
\end{multline*}
Above $(\omega_{i})_{0}(\alpha)$ is the zero component, as defined in \eqref{omegasplit}, of the vorticity term $\omega_{i}$ for $i=1,2$ etc.  In this section, we prove the following estimates on each difference in the vorticity decomposition.

\begin{prop}\label{prop:diff.vorticity.ests}
For $s \ge 0$, we have the estimates
\begin{equation}\label{omega0diffest}
\|(\omega_{1})_{0}-(\omega_{2})_{0}\|_{\fsone} \leq   2 \left|A_\rho \right| C_{L}\|\theta_{1}-\theta_{2}\|_{\fzerone} + \frac{A_\rho}{\pi}L_{2}(t)|\widehat{\vartheta}_{1}(0)-\widehat{\vartheta}_{2}(0)|,
\end{equation}
and
\begin{multline}\label{linearomegadiffnorm}
\|(\omega_{1})_{1}-(\omega_{2})_{1}\|_{\fsone}\leq 
\diffusint_{s}(\|\theta_{1}-\theta_{2}\|_{\fzerone} +|\widehat{\vartheta}_{1}(0)-\widehat{\vartheta}_{2}(0)|)\\ + \timeintE \|\theta_{1}-\theta_{2}\|_{\fsone}  + \frac{4A_{\sigma}\pi}{L_{2}(t)}\|\theta_{1} - \theta_{2}\|_{\mathcal{F}^{s+2,1}}.
\end{multline}
For $s>0$ 
we further have
\begin{multline}\label{omeganonlinearest}
\|(\omega_{1})_{\geq 2}-(\omega_{2})_{\geq 2}\|_{\fsone} \leq \timeintE_{s}(\|\theta_{1}-\theta_{2}\|_{\fzerone} +|\hat{\vartheta_{1}}(0)-\hat{\vartheta_{2}}(0)|) \\+ \timeintE_{0}\|\theta_{1}-\theta_{2}\|_{\fsone} + \diffusint_{s}\|\theta_{1}-\theta_{2}\|_{\dot{\mathcal{F}}^{2,1}} + \tilde{\Gamma}\|\theta_{1}-\theta_{2}\|_{\dot{\mathcal{F}}^{2+s,1}}
\end{multline}
where $\tilde{\Gamma}$ is given by \eqref{tildeGamma}.
\end{prop}

We further give the estimate of the form \eqref{omeganonlinearest} when $s=0$ in \eqref{omega.diff.zero.est}.  It is important to notice that $\tilde{\Gamma}$ given by \eqref{tildeGamma} is smaller than the corresponding coefficient of $\|\theta\|_{\dot{\mathcal{F}}^{s+2,1}}$ in the estimate of \eqref{omega2fss}.

\begin{proof}
For the zero-th order term in the splitting we have
\begin{align*}
(\omega_{1})_{0}(\alpha)-(\omega_{2})_{0}(\alpha) &= -2A_\rho\Big(\frac{L_{1}(t)}{2\pi}-\frac{L_{2}(t)}{2\pi} \Big)\sin{(\alpha+\widehat{\vartheta}_{1}(0))}\\ &\hspace{0.5in}- 2A_\rho\frac{L_{2}(t)}{2\pi}(\sin{(\alpha+\widehat{\vartheta}_{1}(0))}-\sin{(\alpha+\widehat{\vartheta}_{2}(0))}).
\end{align*}
Hence, for $s \ge 0$, 
we have the estimate
\begin{align*}
\|(\omega_{1})_{0}-(\omega_{2})_{0}\|_{\fsone} &\leq \frac{\left| A_\rho \right| }{\pi} |L_{1}(t)-L_{2}(t)|\|\sin{(\alpha+\widehat{\vartheta}_{1}(0))}\|_{\fsone} \\&\hspace{0.25in}+  \frac{\left| A_\rho \right|}{\pi}L_{2}(t)\|\sin{(\alpha+\widehat{\vartheta}_{1}(0))}-\sin{(\alpha+\widehat{\vartheta}_{2}(0))}\|_{\fsone}.
\end{align*}
We have
$\|\sin{(\alpha+\widehat{\vartheta}_{1}(0))}\|_{\fsone} \leq 1$
and
\begin{multline*}
\|\sin{(\alpha+\widehat{\vartheta}_{1}(0))}-\sin{(\alpha+\widehat{\vartheta}_{2}(0))}\|_{\fsone} \leq 2\Big|\sin\Big(\frac{\widehat{\vartheta}_{1}(0)-\widehat{\vartheta}_{2}(0)}{2} \Big)\Big|\\ \cdot \Big\|\cos\Big(\alpha +
\frac{\widehat{\vartheta}_{1}(0)+\widehat{\vartheta}_{2}(0)}{2} \Big)\Big\|_{\fsone}
\end{multline*}
We have 
$\Big\|\cos\Big(\alpha +\frac{\widehat{\vartheta}_{1}(0)+\widehat{\vartheta}_{2}(0)}{2} \Big)\Big\|_{\fsone} \leq 1$
and since we assume the difference is small we have
\begin{equation}\label{sinezero}
\Big|\sin\Big(\frac{\widehat{\vartheta}_{1}(0)-\widehat{\vartheta}_{2}(0)}{2} \Big)\Big| \leq \frac{1}{2} |\widehat{\vartheta}_{1}(0)-\widehat{\vartheta}_{2}(0)|.
\end{equation}
Hence, we obtain \eqref{omega0diffest}, which completes the difference estimates for the zero order vorticity strength terms.

Next, the linear difference in the vorticity strength terms from \eqref{omegasplit} is
\begin{equation}\label{linearomegadiff}
(\omega_{1})_{1}(\alpha)-(\omega_{2})_{1}(\alpha) = \frac{A_\mu}{\pi}W_{1} + 4A_\sigma\pi W_{2} -\frac{A_\rho}{\pi}  W_{3}
\end{equation}
where
\begin{equation}\notag
W_{1} =L_{1}(t)\md_1((\omega_{1})_0)(\alpha) - L_{2}(t)\md_1((\omega_{2})_0)(\alpha)
\end{equation}
and
\begin{equation}\notag
W_{2} =\Big(\frac{1}{L_{1}(t)} -\frac{1}{L_{2}(t)}\Big) (\theta_{1})_{\alpha\alpha} + \frac{1}{L_{2}(t)}((\theta_{1})_{\alpha\alpha} - (\theta_{2})_{\alpha\alpha})
\end{equation}
and
\begin{multline}\notag
W_{3} =(L_{1}(t)-L_{2}(t))\cos{(\alpha+\widehat{\vartheta}_{1}(0))}\theta_{1}(\alpha) \\+ L_{2}(t)[\cos{(\alpha+\widehat{\vartheta}_{1}(0))}-\cos{(\alpha+\widehat{\vartheta}_{2}(0))}]\theta_{1}(\alpha)\\ + L_{2}(t)\cos{(\alpha+\widehat{\vartheta}_{2}(0))}[\theta_{1}(\alpha)-\theta_{2}(\alpha)].
\end{multline}
We will estimate each of the terms  $W_{1}$,  $W_{2}$ and  $W_{3}$ in the following.

For $W_{1}$ we have using \eqref{mdsplit} that
\begin{multline*}
W_{1} = \theta_{1}(\alpha)\mH ( (\omega_{1})_0)(\alpha)+\imag\hspace{0.05cm}\mR((\omega_{1})_0)(\alpha)- \theta_{2}(\alpha)\mH ((\omega_{2})_0)(\alpha)+\imag\hspace{0.05cm}\mR((\omega_{2})_0)(\alpha).
\end{multline*}
It can be shown by the estimates in $\mR$ and $L_{1}-L_{2}$ that for $s \ge 0$ we have
\begin{equation}\label{D1est}
\|W_{1}\|_{\fsone} \leq \diffusint_{s}(\|\theta_{1}-\theta_{2}\|_{\fzerone} + |\widehat{\vartheta}_{1}(0)-\widehat{\vartheta}_{2}(0)|)+ \timeintE\|\theta_{1}-\theta_{2}\|_{\fsone}.
\end{equation}
For $W_{2}$, we have
\begin{equation}\label{W2est}
\|W_{2}\|_{\fsone} \leq \frac{1}{L_{1}(t)L_{2}(t)}C_{L}\|\theta_{1}-\theta_{2}\|_{\fzerone}\|\theta_{1}\|_{\mathcal{F}^{s+2,1}}+ \frac{1}{L_{2}(t)}\|\theta_{1} - \theta_{2}\|_{\mathcal{F}^{s+2,1}}.
\end{equation}
The important term in $W_{2}$ for the purposes of the uniqueness argument is the second term that has the difference $\|\theta_{1} - \theta_{2}\|_{\mathcal{F}^{s+2,1}}$.  For $W_{3}$, using \eqref{lengthdifference}, we obtain that
\begin{equation}\label{W3est2}
    \|W_{3}\|_{\fsone} \leq \diffusint_{s}(\|\theta_{1}-\theta_{2}\|_{\fzerone}+|\widehat{\vartheta}_{1}(0)-\widehat{\vartheta}_{2}(0)|)+ \timeintE\|\theta_{1}-\theta_{2}\|_{\fsone}.
\end{equation}
Hence, from \eqref{D1est}, \eqref{W2est} and \eqref{W3est2}, we obtain from \eqref{linearomegadiff} that
\begin{multline*}
\|(\omega_{1})_{1}-(\omega_{2})_{1}\|_{\fsone}\leq 
\diffusint_{s}(\|\theta_{1}-\theta_{2}\|_{\fzerone} +|\widehat{\vartheta}_{1}(0)-\widehat{\vartheta}_{2}(0)|)\\ + \timeintE \|\theta_{1}-\theta_{2}\|_{\fsone}  + \frac{4A_{\sigma}\pi}{L_{2}(t)}\|\theta_{1} - \theta_{2}\|_{\mathcal{F}^{s+2,1}}.
\end{multline*}
Note that the coefficient in front of $\|\theta_{1} - \theta_{2}\|_{\mathcal{F}^{s+2,1}}$ is the same as that in \eqref{omega1fss} in front of $\|\theta\|_{\mathcal{F}^{s+2,1}_{\nu}}$.
This completes the difference estimates for the linear terms in the vorticity strength given by \eqref{linearomegadiffnorm}.

Next we estimate differences of the nonlinear terms in the vorticity strength from \eqref{omegasplit}.  We decompose the terms as
\begin{equation}\label{omegadiff2}
(\omega_{1})_{\geq 2}-(\omega_{2})_{\geq 2} = W_{21} + W_{22} + W_{23}
\end{equation}
where the terms $W_{21}$, $W_{22}$, and $W_{23}$ are given by
\begin{equation}\notag
W_{21} = A_\mu\frac{L_{1}(t)}{\pi}\md_{\geq2}(\omega_{1})(\alpha) - A_\mu\frac{L_{2}(t)}{\pi}\md_{\geq2}(\omega_{2})(\alpha),
\end{equation}
and
\begin{multline}\notag
W_{22} = A_\rho\frac{L_{2}(t)}{\pi}\sin{(\alpha\!+\!\widehat{\vartheta}_{2}(0))}\sum_{j\geq1}\!\!\frac{(-1)^j(\theta_{2}(\alpha))^{2j}}{(2j)!}\\ - A_\rho\frac{L_{1}(t)}{\pi}\sin{(\alpha\!+\!\widehat{\vartheta}_{1}(0))}\sum_{j\geq1}\!\!\frac{(-1)^j(\theta_{1}(\alpha))^{2j}}{(2j)!},
\end{multline}
and
\begin{multline}\notag
W_{23} = A_\rho\frac{L_{2}(t)}{\pi}\cos{(\alpha+\widehat{\vartheta}_{2}(0))}\sum_{j\geq1}\frac{(-1)^j(\theta_{2}(\alpha))^{1+2j}}{(1+2j)!}\\-A_\rho\frac{L_{1}(t)}{\pi}\cos{(\alpha+\widehat{\vartheta}_{1}(0))}\sum_{j\geq1}\frac{(-1)^j(\theta_{1}(\alpha))^{1+2j}}{(1+2j)!}.
\end{multline}
First, we use \eqref{mdsplit} to observe that
\begin{multline}\label{W21expansion}
W_{21} = -A_{\mu}\Big(\theta_{1}(\alpha)\mH ( (\omega_{1})_{\geq1})(\alpha)+\imag\hspace{0.05cm} \mR((\omega_{1})_{\geq1})(\alpha)+\imag\hspace{0.05cm}\mS(\omega_{1})(\alpha)\\ -\theta_{2}(\alpha)\mH ( (\omega_{2})_{\geq1})(\alpha)-\imag\hspace{0.05cm} \mR((\omega_{2})_{\geq1})(\alpha)-\imag\hspace{0.05cm}\mS(\omega_{2})(\alpha) \Big)\\
= -A_{\mu}(W_{211}+W_{212}+W_{213}),
\end{multline}
where the differences of like terms with either subscript $1$ or $2$ are combined in each $W_{21j}$. Then for $s \ge 0$ we have
\begin{multline}\notag
\|W_{211}\|_{\fsone} \leq \bfcn(2,s) \|\theta_{1}-\theta_{2}\|_{\fsone}\|(\omega_{1})_{\geq 1}\|_{\fzerone} + \bfcn(2,s)\|\theta_{1}-\theta_{2}\|_{\fzerone}\|(\omega_{1})_{\geq 1}\|_{\fsone}\\ + \bfcn(2,s)\|\theta_{2}\|_{\fsone}\|(\omega_{1})_{1} -(\omega_{2})_{1}\|_{\fzerone} +  \bfcn(2,s)\|\theta_{2}\|_{\fzerone}\|(\omega_{1})_{1} -(\omega_{2})_{1}\|_{\fsone} \\ + \bfcn(2,s)\|\theta_{2}\|_{\fzerone}\|(\omega_{1})_{\geq 2} -(\omega_{2})_{\geq 2}\|_{\fsone} +\bfcn(2,s)\|\theta_{2}\|_{\fsone}\|(\omega_{1})_{\geq 2} -(\omega_{2})_{\geq 2}\|_{\fzerone}.
\end{multline}
Thus, using \eqref{omega1f0}, \eqref{omega1fss}, \eqref{omega2f0}, \eqref{omega2fss}  and \eqref{linearomegadiffnorm} we have
\begin{multline}\label{W211est}
\|W_{211}\|_{\fsone} \leq
\timeintE_{s}(\|\theta_{1}-\theta_{2}\|_{\fzerone}+|\widehat{\vartheta}_{1}(0)-\widehat{\vartheta}_{2}(0)|) +\timeintE_{0}\|\theta_{1}-\theta_{2}\|_{\fsone}
\\
+\diffusint_{s}\|\theta_{1}-\theta_{2}\|_{\mathcal{F}^{2,1}} 
+
\frac{4A_{\sigma}\pi}{L_{2}(t)}\bfcn(2,s)\|\theta_{1}-\theta_{2}\|_{\mathcal{F}^{2+s,1}} 
\\
+ 
\bfcn(2,s)\|\theta_{2}\|_{\fzerone}\|(\omega_{1})_{\geq 2} -(\omega_{2})_{\geq 2}\|_{\fsone}
+\bfcn(2,s)\|\theta_{2}\|_{\fsone}\|(\omega_{1})_{\geq 2} -(\omega_{2})_{\geq 2}\|_{\fzerone},
\end{multline}
where $\timeintE_{s} = \|\theta_{1},\theta_{2}\|_{\dot{\mathcal{F}}^{2+s,1}}\timeintE$  and some bounded constant $\timeintE$ and $C$ is a bounded constant depending on $\|\theta_{1},\theta_{2}\|_{\fzerone}$.

We now consider the term $W_{212}$.  We have from \eqref{Restimates} that
\begin{multline}\label{Rdiffgeq1}
\|\mR((\omega_{1})_{\geq 1}) - \mR((\omega_{2})_{\geq 1})\|_{\fsone}\\ \leq \bfcn(2,s){\crconstant}(\|(\omega_{1})_{\geq 1}-(\omega_{2})_{\geq 1}\|_{\fsone} \|\theta_{1}\|_{\fzerone} + \|(\omega_{1})_{\geq 1}-(\omega_{2})_{\geq 1}\|_{\fzerone} \|\theta_{1}\|_{\fsone}\\+\|(\omega_{2})_{\geq 1}\|_{\fsone} \|\theta_{1}-\theta_{2}\|_{\fzerone}+\|(\omega_{2})_{\geq 1}\|_{\fzerone} \|\theta_{1}-\theta_{2}\|_{\fsone}).
\end{multline}
Hence, using \eqref{omega1f0}, \eqref{omega1fss}, \eqref{omega2f0}, \eqref{omega2fss} and \eqref{linearomegadiffnorm}, we obtain
\begin{multline}\label{W212est}
\|W_{212}\|_{\fsone}\leq  \timeintE_{s}(\|\theta_{1}-\theta_{2}\|_{\fzerone}+|\widehat{\vartheta}_{1}(0)-\widehat{\vartheta}_{2}(0)|) +\timeintE_{0}\|\theta_{1}-\theta_{2}\|_{\fsone}\\
+ \timeintE\|\theta_{1},\theta_{2}\|_{\fsone}\|\theta_{1}-\theta_{2}\|_{\mathcal{F}^{2,1}}
+\frac{4A_{\sigma}{\crconstant} \pi}{L_{2}(t)}\bfcn(2,s)\|\theta_{1}-\theta_{2}\|_{\mathcal{F}^{2+s,1}}\|\theta_{1},\theta_{2}\|_{\fzerone}
\\ 
+ 
\bfcn(2,s){\crconstant} (\|\theta_{1}\|_{\fzerone}\|(\omega_{1})_{\geq 2} -(\omega_{2})_{\geq 2}\|_{\fsone} + \|\theta_{1}\|_{\fsone}\|(\omega_{1})_{\geq 2} -(\omega_{2})_{\geq 2}\|_{\fzerone}).
\end{multline}
This is the estimate for $W_{212}$.

Next for $W_{213}$ containing the difference in $\mathcal{S}$, we actually have to consider two differences.
We recall the splitting of $\mathcal{S}$ from \eqref{Ssum} and we will use $f_i = \omega_i$ below. First, it can be shown from  \eqref{tildeS} that
\begin{align}\label{Boperatordiff}
\|{\stwosplit}(f_{1}) - {\stwosplit}(f_{2})\|_{\fsone}
&\leq  B_{1} + B_{2}
\end{align}
where
\begin{equation}\notag
B_{1} = \Big\||k|^{s}\sum_{\substack{n,l\geq 0 \\ n+l\geq 2}}\frac{(-1)^ni^{l+n+1}(\ast^{l}\widehat{\theta}_{1}(k))-\ast^{l}\widehat{\theta_{2}}(k))}{l!} \ast \widehat{\mathcal{S}_{n}(f_{1})}(k) \Big\|_{\ell^{1}}
\end{equation}
and
\begin{equation}\notag
B_{2} = \Big\||k|^{s}\sum_{\substack{n,l\geq 0 \\ n+l\geq 2}}\frac{(-1)^{n} i^{l+n+1}\ast^{l}\widehat{\theta}_{2}(k)}{l!} \ast (\widehat{\mathcal{S}_{n}(f_{1})}(k)-\widehat{\mathcal{S}_{n}(f_{2})}(k)) \Big\|_{\ell^{1}}.
\end{equation}
For $B_{1}$, we obtain using similar arguments to Proposition \ref{uniquenessprop} that
\begin{multline}\notag
B_{1} \leq 
\sum_{\substack{n,l\geq 0 \\ n+l\geq 2}} \frac{b(l+1,s)}{(l-1)!} \Big( 
 \|\theta_{1},\theta_{2}\|_{\fzerone}^{l-1}\|\mathcal{S}_{n}(f_{1})\|_{\fsone} \|\theta_{1}-\theta_{2}\|_{\fzerone}
 \\
+ 
(l-1)\|\theta_{1},\theta_{2}\|_{\fzerone}^{l-2}\|\theta_{1},\theta_{2}\|_{\fsone}\|\mathcal{S}_{n}(f_{1})\|_{\fzerone}
\|\theta_{1}-\theta_{2}\|_{\fzerone}
\\ 
+ \|\theta_{1},\theta_{2}\|_{\fzerone}^{l-1}\|\mathcal{S}_{n}(f_{1})\|_{\fzerone}\|\theta_{1}-\theta_{2}\|_{\fsone}\Big).
\end{multline}
For $B_{2}$, we first consider the difference in the operator $\mathcal{S}_{n}$. By \eqref{Snfourier}, we have for two functions $f_{1}$ and $f_{2}$, with $a_{n}$ given by \eqref{an.def}, that
\begin{align*}
|i^{n}(\widehat{\mathcal{S}_{n}(f_{1})}-\widehat{\mathcal{S}_{n}(f_{2})})(k)| &= \sum_{k_{2},\ldots, k_{n+1}\in \mathbb{Z}}|I(k_{1},\ldots,k_{n+1})|\\
&\cdot\Big(|\hat{f_{1}}(k_{n+1})\prod_{j=1}^{n}P_{1}(k_{j}-k_{j+1}) - \hat{f_{2}}(k_{n+1})\prod_{j=1}^{n}P_{2}(k_{j}-k_{j+1})|\Big)\\
&\leq a_{n} |\hat{f}_{1}-\hat{f}_{2}|\ast^{n} |P_{1}| + a_{n} |\hat{f}_{2}|\ast|\ast^{n} P_{1}-\ast^{n} P_{2}|.
\end{align*}
Hence using the estimates as in \eqref{S1S2S3} we have
\begin{align}\label{B2est}
\begin{split}
B_{2} &\leq \sum_{\substack{n,l\geq 0 \\ n+l\geq 2}} a_{n} \Big\||k|^{s}\frac{\ast^{l}\widehat{\theta}_{2}(k)}{l!} \ast(|\hat{f}_{1}-\hat{f}_{2}|\ast^{n} |P_{1}|) \Big\|_{\ell^{1}} + \tilde{B}_{2}\\
&\leq \tilde{C}_3\|\theta_{2}\|_{\fzerone}^2\|f_{1}-f_{2}\|_{\fsone}+\tilde{C}_4\|\theta_{2}\|_{\fzerone}\|\theta_{2}\|_{\fsone}\|f_{1}-f_{2}\|_{\fzerone} + \tilde{B}_{2}
\end{split}
\end{align}
where
\begin{align}\notag
\tilde{B}_{2} = \sum_{\substack{n,l\geq 0 \\ n+l\geq 2}} a_{n} \Big\||k|^{s}\frac{\ast^{l}\widehat{\theta}_{2}(k)}{l!} \ast(|\hat{f}_{2}|\ast|\ast^{n} P_{1}-\ast^{n} P_{2}|) \Big\|_{\ell^{1}}.
\end{align}
We have for $s > 0$ that
\begin{multline}\notag
\||k|^{s}(P_{1}-P_{2})\|_{\ell^{1}} \leq \sum_{m\geq 1} \frac{\bfcn(m,s)}{(m-1)!}(\|\theta_{1},\theta_{2}\|_{\fzerone}^{m-1}\|\theta_{1}-\theta_{2}\|_{\fsone}\\ + (m-1)\|\theta_{1},\theta_{2}\|_{\fzerone}^{m-2}\|\theta_{1},\theta_{2}\|_{\fsone}\|\theta_{1}-\theta_{2}\|_{\fzerone})
\end{multline}
with an analogous estimate holding in the case $s=0$.
Hence we have
\begin{multline}\notag
\tilde{B}_{2} \leq \timeintE\Big[(\|\theta_{1},\theta_{2}\|_{\fsone}+\|f_{i}\|_{\fsone})\|\theta_{1}-\theta_{2}\|_{\fzerone} + (\|\theta_{1},\theta_{2}\|_{\fzerone}+\|f_{i}\|_{\fzerone})\|\theta_{1}-\theta_{2}\|_{\fsone}\Big].
\end{multline}
We also similarly estimate ${\sonesplit}$ as in \eqref{breveSsbound} to obtain
\begin{multline}\notag
\|{\sonesplit}(f_{1})(\alpha) -	{\sonesplit}(f_{2})(\alpha)\|_{\fsone} \leq \timeintE\Big[\|f_{i}\|_{\fsone}\|\theta_{1}-\theta_{2}\|_{\fzerone} + \|f_{i}\|_{\fzerone}\|\theta_{1}-\theta_{2}\|_{\fsone}\Big]\\ + \mathcal{C}_{\mR} C_{4}\|\theta_{1},\theta_{2}\|_{\fsone}\|\theta_{1},\theta_{2}\|_{\fzerone}\|f_{1}-f_{2}\|_{\fzerone} + \mathcal{C}_{\mR} C_{3}\|\theta_{1},\theta_{2}\|_{\fzerone}^{2}\|f_{1}-f_{2}\|_{\fsone}.
\end{multline}
In summary, for $s>0$, using \eqref{omega0diffest} and \eqref{linearomegadiffnorm}, we obtain
\begin{multline}\label{W213est}
    \|W_{213}\|_{\fsone} \leq \timeintE_{s}(\|\theta_{1}-\theta_{2}\|_{\fzerone} +|\hat{\vartheta_{1}}(0)-\hat{\vartheta_{2}}(0)|)+ \timeintE_{0}\|\theta_{1}-\theta_{2}\|_{\fsone}\\ + \diffusint_{s}\|\theta_{1}-\theta_{2}\|_{\dot{\mathcal{F}}^{2,1}} + C_{3}\|\theta_{1},\theta_{2}\|_{\fzerone}^{2}\frac{4A_{\sigma}\pi}{L_{2}(t)}\|\theta_{1}-\theta_{2}\|_{\dot{\mathcal{F}}^{2+s,1}}\\ + C_{4}\|\theta_{1},\theta_{2}\|_{\fzerone}\|\theta_{1},\theta_{2}\|_{\fsone}\|(\omega_{1})_{\geq 2}-(\omega_{1})_{\geq 2}\|_{\fzerone} + C_{3}\|\theta_{1},\theta_{2}\|_{\fzerone}^{2}\|(\omega_{1})_{\geq 2}-(\omega_{2})_{\geq 2}\|_{\fsone},
\end{multline}
where $C_{3}$ and $C_{4}$ are given by \eqref{C3} and ${\crconstant}$ is given by \eqref{CR}.

Further using Proposition \ref{uniquenessprop}, then $W_{22}$ and $W_{23}$ can be estimated by a bound of the type:
\begin{equation}\label{W22W23}
\|W_{2j}\|_{\fsone} \leq \diffusint_{s}(\|\theta_{1}-\theta_{2}\|_{\fzerone} +|\widehat{\vartheta}_{1}(0)-\widehat{\vartheta}_{2}(0)|)+\timeintE\|\theta_{1}-\theta_{2}\|_{\fsone}
\end{equation}
for $j=2,3$.
Hence, using \eqref{W211est}, \eqref{W212est}, \eqref{W213est}, we obtain from \eqref{omegadiff2} that
\begin{multline}\label{omeganonlinearestinitial}
\|(\omega_{1})_{\geq 2}-(\omega_{2})_{\geq 2}\|_{\fsone} \leq \timeintE_{s}(\|\theta_{1}-\theta_{2}\|_{\fzerone} +|\hat{\vartheta_{1}}(0)-\hat{\vartheta_{2}}(0)|) \\+ \timeintE_{0}\|\theta_{1}-\theta_{2}\|_{\fsone} + \diffusint_{s}\|\theta_{1}-\theta_{2}\|_{\dot{\mathcal{F}}^{2,1}} + \tilde{\Gamma}\|\theta_{1}-\theta_{2}\|_{\dot{\mathcal{F}}^{2+s,1}} \\
+|A_{\mu}|C_{12}\|\theta_{1},\theta_{2}\|_{\fsone}\|(\omega_{1})_{\geq 2}-(\omega_{2})_{\geq 2}\|_{\fzerone} +|A_{\mu}|C_{2}\|\theta_{1},\theta_{2}\|_{\fzerone}\|(\omega_{1})_{\geq 2}-(\omega_{2})_{\geq 2}\|_{\fsone}
\end{multline}
where
\begin{equation}\label{tildeGamma}
\tilde{\Gamma} =|A_{\mu}| \frac{4A_{\sigma}\pi}{L_{2}(t)} C_{2}\|\theta_{1},\theta_{2}\|_{\fzerone}
\end{equation}
with $C_{2}$ given by \eqref{C12}.

Computing an estimate analogous to \eqref{omeganonlinearestinitial} for $s=0$ yields the estimate
\begin{multline}\label{omega.diff.zero.est}
\|(\omega_{1})_{\geq 2}-(\omega_{2})_{\geq 2}\|_{\fsone} \leq \timeintE_{s}(\|\theta_{1}-\theta_{2}\|_{\fzerone} +|\hat{\vartheta_{1}}(0)-\hat{\vartheta_{2}}(0)|) \\+ \timeintE_{0}\|\theta_{1}-\theta_{2}\|_{\fsone} + \diffusint_{s}\|\theta_{1}-\theta_{2}\|_{\dot{\mathcal{F}}^{2,1}} + \Gamma\|\theta_{1}-\theta_{2}\|_{\dot{\mathcal{F}}^{2+s,1}}
\end{multline}
where
\begin{equation}\label{Gamma}
\Gamma = |A_{\mu}|\tilde{C}_{8}\frac{4A_{\sigma}\pi}{L_{2}(t)}
\end{equation}
for $\tilde{C}_{8}$ given by \eqref{tildeC8} and the other constants given by Definition \ref{uniquenessimplicitdef}. 
\end{proof}

\section{Estimates for the differences of the main nonlinear term}\label{subsec:uniq.nonlinear}
In this section, we show the following bound on the nonlinear difference term of \eqref{theta1theta2}.

\begin{prop}\label{nonlineardiffprop}
We have the following estimate for $\delta>0$ given by \eqref{delta} and for $\epsilon(\|\theta_{1},\theta_{2}\|_{\fhone})$ that can be chosen to be arbitrarily small:
\begin{equation*}
\|N_{1}-N_{2}\|_{\fhone} \leq {\timeintE}(\|\theta_{1}-\theta_{2}\|_{\fhone} + |\hat{\vartheta}_{1}(0)-\hat{\vartheta}_{2}(0)|) + (\delta+\epsilon) \|\theta_{1}-\theta_{2}\|_{\fsevenhone}
\end{equation*}
where ${\timeintE}>0$ is a time integrable coefficient depending on $\|\theta_{1},\theta_{2}\|_{\fhone}$ and $\|\theta_{1},\theta_{2}\|_{\fsevenhone}$.  Further $\epsilon = \epsilon(\|\theta_{1},\theta_{2}\|_{\fhone})>0$ can be chosen arbitrarily small.
\end{prop}

We remark that all the terms involving the upper bound of ${\timeintE}(\|\theta_{1}-\theta_{2}\|_{\fhone} + |\hat{\vartheta}_{1}(0)-\hat{\vartheta}_{2}(0)|)$ follow similarly to the estimates from Chapter \ref{secanalytic} also using the idea of Proposition \ref{uniquenessprop} and the vorticity estimates in Proposition \ref{prop:diff.vorticity.ests}.  Our proof below is focused on the estimates of the term $\|\theta_{1}-\theta_{2}\|_{\fsevenhone}$ and the constant $\delta>0$. In the proof of Proposition \ref{nonlineardiffprop}, when we compute the difference of nonlinear terms in $\fhone$, any terms where the difference occurs as $\|\theta_{1}-\theta_{2}\|_{\fsone}$ for $s<7/2$ can be absorbed by interpolation and the Young's inequality into the term ${\timeintE}\|\theta_{1}-\theta_{2}\|_{\fhone}$ and contributes an arbitrarily small term $\epsilon \|\theta_{1}-\theta_{2}\|_{\fsevenhone}$ to be absorbed into the linear decay, e.g. see the estimate to obtain \eqref{samplediff1}. This term is then taken care of by the Gronwall argument described in the comment below Theorem \ref{uniquenessproposition}. In this way, most of the nonlinear terms can be easily estimated and only the few terms of order $\|\theta_{1}-\theta_{2}\|_{\fsevenhone}$ need be computed.

\begin{proof}
For the nonlinear terms, we will denote the decomposition given in \eqref{N1N2N3} of $N_1$ for $\theta_{1}$ and $N_2$ for $\theta_{2}$ respectively by 
$$ N_{1} = N_{11}+N_{12}+N_{13},
\quad
N_{2} = N_{21}+N_{22}+N_{23}.$$
We now consider the differences of $N_{1j}- N_{2j}$ for $j=1,2,3$. We will only explicitly compute the constant in front of terms with difference $\| \theta_{1}-\theta_{2}\|_{\fsevenhone}$.  We make this idea clear in the following. Denoting $(U_{i})_{\geq 2}$ for the term containing $\theta_{i}$, we have
\begin{multline}\label{N1diff}
\|N_{11}-N_{21}\|_{\fhone} \leq \Big|\frac{\pi}{L_{1}(t)}-\frac{\pi}{L_{2}(t)}\Big|\|\frac{L_{1}(t)}{\pi}(U_{1})_{\geq 2}\|_{\fthreehone} \\+ \frac{1}{L_{2}(t)}\|\frac{L_{1}(t)}{\pi}(U_{1})_{\geq 2}-\frac{L_{2}(t)}{\pi}(U_{2})_{\geq 2}\|_{\fthreehone}.
\end{multline}
In the latter term of \eqref{N1diff}, we have a term of the form
\begin{multline}\label{Romegadiff}
    \|\mR_{1}((\omega_{1})_{\geq 1})-\mR_{2}((\omega_{2})_{\geq 1})\|_{\fthreehone} 
    \\
    \leq \ldots + \frac{4A_{\sigma}\pi}{L_{2}(t)}\|\mR_{1}((\theta_{1})_{\alpha\alpha})-\mR_{2}((\theta_{2})_{\alpha\alpha})\|_{\fthreehone}
    + \|\mR_{1}((\omega_{1})_{\geq 2})-\mR_{2}((\omega_{2})_{\geq 2})\|_{\fthreehone}.
\end{multline}
Above the dots ``$\ldots$'' represent that other terms are present which turn out to be lower order.  We similarly denote $\mR_{i}$ for the term \eqref{R} which contains $\theta_{i}$.   Then for the first term present in the upper bound above we have the estimate
\begin{align}\notag
\begin{split}
\|\mR_{1}((\theta_{1})_{\alpha\alpha})-\mR_{2}((\theta_{2})_{\alpha\alpha})\|_{\fthreehone} &\leq  \|\mR_{1}((\theta_{1})_{\alpha\alpha}-(\theta_{2})_{\alpha\alpha})\|_{\fthreehone}\\ &\hspace{0.5in}+ \|\mR_{1}((\theta_{2})_{\alpha\alpha})-\mR_{2}((\theta_{2})_{\alpha\alpha})\|_{\fthreehone}\\
\leq \sqrt{2}{\crconstant}&(\|\theta_{1}-\theta_{2}\|_{\fsevenhone}\|\theta_{1}\|_{\fzerone} + \|\theta_{1}-\theta_{2}\|_{\mathcal{F}^{2,1}}\|\theta_{1}\|_{\fthreehone}\\ &+\|\theta_{2}\|_{\fsevenhone}\|\theta_{1}-\theta_{2}\|_{\fzerone} + \|\theta_{2}\|_{\mathcal{F}^{2,1}}\|\theta_{1}-\theta_{2}\|_{\fthreehone}).
\end{split}
\end{align}
We therefore have
\begin{multline}\label{r1r2diffaa}
\|\mR_{1}((\theta_{1})_{\alpha\alpha})-\mR_{2}((\theta_{2})_{\alpha\alpha})\|_{\fthreehone}
\\
\leq \sqrt{2}{\crconstant}(\|\theta_{1}-\theta_{2}\|_{\fsevenhone}\|\theta_{1}\|_{\fzerone}+\|\theta_{1}\|_{\fhone}^{2/3} \|\theta_{1}\|_{\fsevenhone}^{1/3}\|\theta_{1}-\theta_{2}\|_{\fsevenhone}^{1/2}\|\theta_{1}-\theta_{2}\|_{\fhone}^{1/2})
\\
+\sqrt{2}{\crconstant}(\|\theta_{2}\|_{\fsevenhone}\|\theta_{1}-\theta_{2}\|_{\fzerone} +\|\theta_{2}\|_{\fhone}^{1/2} \|\theta_{2}\|_{\fsevenhone}^{1/2}\|\theta_{1}-\theta_{2}\|_{\fsevenhone}^{1/3}\|\theta_{1}-\theta_{2}\|_{\fhone}^{2/3}).
\end{multline}
Hence, if we apply Young's inequality, e.g.
\begin{multline*}
\|\theta_{1}\|_{\fhone}^{2/3} \|\theta_{1}\|_{\fsevenhone}^{1/3}\|\theta_{1}-\theta_{2}\|_{\fsevenhone}^{1/2}\|\theta_{1}-\theta_{2}\|_{\fhone}^{1/2}\\ \leq \frac{1}{4\epsilon}\|\theta_{1}\|_{\fsevenhone}^{2/3}\|\theta_{1}-\theta_{2}\|_{\fhone} + \epsilon \|\theta_{1}\|_{\fhone}^{4/3} \|\theta_{1}-\theta_{2}\|_{\fsevenhone}
\end{multline*}
we obtain
\begin{multline}\label{samplediff1}
\|\mR_{1}((\theta_{1})_{\alpha\alpha})-\mR_{2}((\theta_{2})_{\alpha\alpha})\|_{\fthreehone} \\ \leq \timeintE(\|\theta_{1},\theta_{2}\|_{_{\fsevenhone}}\|\theta_{1}-\theta_{2}\|_{\mathcal{F}^{0,1}} +c_{\epsilon}(\|\theta_{1},\theta_{2}\|_{\fsevenhone}^{2/3}+\|\theta_{1},\theta_{2}\|_{\fsevenhone}^{3/4})\|\theta_{1}-\theta_{2}\|_{\fhone}) \\+\epsilon (\|\theta_{1}\|_{\fhone}^{4/3}+\|\theta_{1}\|_{\fhone}^{3/2})\|\theta_{1}-\theta_{2}\|_{\fsevenhone}+ \sqrt{2}{\crconstant}\|\theta_{1}-\theta_{2}\|_{\fsevenhone}\|\theta_{1}\|_{\fzerone}
\end{multline}
for a constant $c_{\epsilon}$. The first two terms in \eqref{samplediff1} are linear in $\|\theta_{1}-\theta_{2}\|_{\fsone}$ for $0 \leq s \leq 1/2$ and the constants are time integrable on $[0,T]$ for any $T>0$ since $\|\theta_{1},\theta_{2}\|_{\fsevenhone}$ is integrable in time. The final two terms with the difference $\|\theta_{1}-\theta_{2}\|_{\fsevenhone}$ needs to be absorbed in the linear decay coming from $\mathcal{L}_{1}-\mathcal{L}_{2}$. For the other term in \eqref{Romegadiff}, we have similarly,
\begin{multline}
\|\mR_{1}((\omega_{1})_{\geq 2})-\mR_{2}((\omega_{2})_{\geq 2})\|_{\fthreehone} \leq \ldots + \sqrt{2}{\crconstant}\|\theta_{1},\theta_{2}\|_{\fzerone}\Gamma\|\theta_{1}-\theta_{2}\|_{\fsevenhone}\\ + \epsilon\|\theta_{1}-\theta_{2}\|_{\fsevenhone},
\end{multline}
where we use \eqref{Gamma} and $\epsilon=\epsilon(\|\theta_{1},\theta_{2}\|_{\fhone})$ is a constant that can be chosen arbitrarily small.

The only other terms containing a term like $\|\theta_{1}-\theta_{2}\|_{\fsevenhone}$ are the other two terms that also come from $N_{11}-N_{12}$, as can be observed from the terms which contain $\|\theta\|_{\fsevenhone}$ in the estimates of Chapter \ref{secanalytic}. The first term is
\begin{align}
\frac{\pi}{L_{2}(t)}\|\mH ( (\omega_{1})_{\geq2} )-\mH ((\omega_{2})_{\geq2})\|_{\fthreehone} &= \frac{\pi}{L_{2}(t)}\| (\omega_{1})_{\geq2}-(\omega_{1})_{\geq2}\|_{\fthreehone} \nonumber \\
&\leq \ldots + (\frac{\pi\Gamma}{L_{2}(t)} + \epsilon)\|\theta_{1}-\theta_{2}\|_{\fsevenhone}, \label{Hdiff}
\end{align}
where the unwritten terms are lower order due to being linear in $\|\theta_{1}-\theta_{2}\|_{\fsone}$ for $0 \leq s \leq 1/2$ with time integrable coefficients as done in \eqref{samplediff1}. Again, $\epsilon=\epsilon(\|\theta_{1},\theta_{2}\|_{\fhone})$ is a constant that can be chosen arbitrarily small. Similarly, the final term that we need to compute is
\begin{align}
\frac{\pi}{L_{2}(t)}\|\mathcal{S}(\omega_{1})-\mathcal{S}(\omega_{2})\|_{\fthreehone} &\leq \ldots +\frac{\pi}{L_{2}(t)}\|\mathcal{S}((\omega_{1})_{1})-\mathcal{S}((\omega_{2})_{1})\|_{\fthreehone} \nonumber\\
&\hspace{0.5in}+\frac{\pi}{L_{2}(t)}\|\mathcal{S}((\omega_{1})_{\geq 2})-\mathcal{S}((\omega_{2})_{\geq 2})\|_{\fthreehone} \nonumber \\
\leq \ldots& + (\frac{\pi\Gamma C_{3}\|\theta_{1},\theta_{2}\|_{\fzerone}^{2}}{L_{2}(t)}+\epsilon(\|\theta_{1},\theta_{2}\|_{\fhone}))\|\theta_{1}-\theta_{2}\|_{\fsevenhone}, \label{Sneededdiff}
\end{align}
where we use the estimate on $\mathcal{S}(f_{1})-\mathcal{S}(f_{2})$ computed to give \eqref{W213est} and so $C_{3}$ is given by \eqref{C3} and $\Gamma$ is given by \eqref{Gamma}.

All remaining nonlinear terms in the estimates of the evolution of $\theta_{1}-\theta_{2}$ are lower order due to having no futher upper bounds in terms of the highest order difference $\|\theta_{1}-\theta_{2}\|_{\fsevenhone}$.  Hence, in total the difference of nonlinear terms yields
\begin{equation}\notag
\|N_{1}-N_{2}\|_{\fhone} \leq {\timeintE}(\|\theta_{1}-\theta_{2}\|_{\fhone} + |\hat{\vartheta}_{1}(0)-\hat{\vartheta}_{2}(0)|) + (\tilde{\delta}+\epsilon) \|\theta_{1}-\theta_{2}\|_{\fsevenhone},
\end{equation}
where $\epsilon=\epsilon(\|\theta_{1},\theta_{2}\|_{\fhone})$ is a constant that can be chosen arbitrarily small; and for $\Gamma= |A_{\mu}|\tilde{C}_{8}\frac{4A_{\sigma}\pi}{L_{2}(t)}$ given by \eqref{Gamma}, we have $\tilde{\delta}$ given by
\begin{equation*}
\tilde{\delta} = \frac{\pi}{L_{2}(t)}(\sqrt{2}{\crconstant}\frac{4A_{\sigma}\pi}{L_{2}(t)} + b(2,3/2){\crconstant}\|\theta_{1}, \theta_{2}\|_{\fzerone}\Gamma + \Gamma + \Gamma C_{3}\|\theta_{1}, \theta_{2}\|_{\fzerone}^{2})
\end{equation*}
and so by \eqref{Lbound}, setting
\begin{multline}\label{delta}
\delta =
\\
\frac{A_{\sigma}}{C_{37}^{2}R^{2}}(\sqrt{2}{\crconstant}+ b(2,3/2){\crconstant}\|\theta_{1}, \theta_{2}\|_{\fzerone}|A_{\mu}|\tilde{C}_{8}+ |A_{\mu}|\tilde{C}_{8} +|A_{\mu}|\tilde{C}_{8} C_{3}\|\theta_{1}, \theta_{2}\|_{\fzerone}^{2}),
\end{multline}
we obtain the result of Proposition \ref{nonlineardiffprop}.
\end{proof}

\section{Proof of uniqueness}\label{subsec:uniq.proof}
We are now ready to prove Theorem \ref{uniquenessproposition}.

\begin{proof}[Proof of Theorem \ref{uniquenessproposition}]
From Proposition \ref{linearfourier}, the difference of the linear terms is
\begin{multline}\label{lineardifference}
\widehat{\mathcal{L}}_{1}(k) - \widehat{\mathcal{L}}_{2}(k) 
=-A_\sigma \frac{4\pi^2}{L_{2}(t)^2}k(k^2-1)(\widehat{\theta}_{1}(k)- \widehat{\theta}_{2}(k))
\\
-(1+A_\mu)A_\rho\frac{(k^2-1)(k+1)}{k(k+2)}e^{-i\widehat{\vartheta}_{2}(0)} (\widehat{\theta}_{1}(k+1)- \widehat{\theta}_{2}(k+1)) 
\\
+ \tilde{L}_{1}(k) 
+ \tilde{L}_{2}(k),
\end{multline}
where
\begin{equation*}
\tilde{L}_{1}(k) = -4\pi^2A_\sigma k(k^2-1)\widehat{\theta}_{1}(k)\Big(\frac{1}{L_{1}(t)^2}-\frac{1}{L_{2}(t)^2}\Big)
\end{equation*}
and 
\begin{equation*}
\tilde{L}_{2}(k) = -(1+A_\mu)A_\rho\frac{(k^2-1)(k+1)}{k(k+2)}(e^{-i\widehat{\vartheta}_{1}(0)}-e^{-i\widehat{\vartheta}_{2}(0)})\widehat{\theta}_{1}(k+1).
\end{equation*}
And similarly for $k=2$.
For $\tilde{L}_{1}$, 
we have
\begin{equation}\label{L1unique}
\|\tilde{L}_{1}\|_{\fhone}\leq 4\pi^{2}A_\sigma \Big|\frac{1}{L_{1}(t)^{2}} - \frac{1}{L_{2}(t)^{2}}\Big|\||k|^{3/2}(k^{2}-1)\widehat{\theta}_{1}(k)\|_{\ell^{1}}.
\end{equation}
For $\tilde{L}_{2}$, we have
\begin{multline}\label{L2unique}
\|\tilde{L}_{2}\|_{\fhone}
\\
\leq |(1+A_\mu)A_\rho| |e^{-i\widehat{\vartheta}_{1}(0)}-e^{-i\widehat{\vartheta}_{2}(0)}|\Big\|\frac{|k|^{1/2}|k^2-1||k+1|}{2|k||k+2|}|\widehat{\theta}_{1}(k+1)|  \Big\|_{\ell^{1}}.
\end{multline}
Using similar arguments as earlier, we obtain from \eqref{L1unique} and \eqref{L2unique} that
\begin{equation}\label{linearnonlinearparts}
\|\tilde{L}_{1},\tilde{L}_{2}\|_{\fhone} \leq {\timeintE}(\|\theta_{1}-\theta_{2}\|_{\fhone} + |\hat{\vartheta}_{1}(0)-\hat{\vartheta}_{2}(0)|)
\end{equation}
where ${\timeintE}$ is a time integrable coefficient. Hence, the new quantities from \eqref{linearnonlinearparts} do not need to be absorbed by the linear decay. The coefficient $\delta$ of the norm $\|\theta_{1}-\theta_{2}\|_{\fsevenhone}$ is less than the coefficients in \eqref{Nbound}, and hence, $\|\theta_{1},\theta_{2}\|_{\fhone}$ satisfying \eqref{condition} and taking $\epsilon$ arbitrarily small is sufficient for $(\delta+\epsilon)\|\theta_{1}-\theta_{2}\|_{\fsevenhone}$ from Proposition \ref{nonlineardiffprop} to be absorbed into the linear decay terms of \eqref{lineardifference} by following the similar procedure to Section \ref{subsecGlobal}.
\end{proof}

\appendix

\backmatter


\providecommand{\bysame}{\leavevmode\hbox to3em{\hrulefill}\thinspace}
\providecommand{\href}[2]{#2}


\end{document}